%% file: preprint.tex
\documentclass[10pt,leqno]{amsart}
\usepackage{indentfirst,csquotes}
\usepackage{listings}
\usepackage{mathtools}
\usepackage{adjustbox}
\usepackage{stmaryrd}
\usepackage{graphicx}
\usepackage{ulem}
\usepackage[most]{tcolorbox}
\usepackage{mdframed}
\usepackage{amssymb}
\usepackage{algorithm}
\usepackage{algpseudocode}
\usepackage{amsthm}
\usepackage{amsmath}
\usepackage{pbox}

\input{preamble.sty}

\definecolor{frameRed}{RGB}{255, 51, 51}
\definecolor{frameOrange}{RGB}{255, 128, 0}
\definecolor{frameGreen}{RGB}{0, 204, 0}

\usepackage{xcolor,paralist,hyperref,etoolbox}

\usepackage{titlesec}
\titleformat{\subsection}
{\normalfont\bfseries}
{\thesubsection}
{1em}
{}
\titleformat{\section}
{\normalfont\Large\bfseries}
{\thesection\quad}
{0pt}
{}

\usepackage{cleveref}

\newtheorem{theorem}{Theorem}[section]

\newtheorem{lemma}[theorem]{Lemma}
\newtheorem{assumption}[theorem]{Assumption}
\newtheorem{problem}[theorem]{Problem}
\newtheorem{remark}[theorem]{Remark}

\newtheorem{corollary}[theorem]{Corollary}

\numberwithin{equation}{section}

\makeatletter
\renewcommand\paragraph[1]{\@startsection{paragraph}{4}{\parindent}%
{0.1ex \@plus 0.1ex \@minus 0.2ex}
{-1em}
{\normalfont\itshape\normalsize}*{#1.\quad}}
\makeatother

\hypersetup{
    colorlinks=true,
    linkcolor=blue,
    filecolor=blue,
    urlcolor=blue,
    citecolor=blue
}

\begin{document}
    \title[Budgeted Multi-Level Monte Carlo Method]
    {\Large A Budgeted Multi-Level Monte Carlo Method for \\ Full Field Estimates of Multi-PDE Problems}
    \author{Niklas Baumgarten}
    \author{Robert Kutri}
    \author{Robert Scheichl}
    \date{\today}

    \maketitle

    \let\thefootnote\relax\footnotetext{\texttt{niklas.baumgarten@uni-heidelberg.de}}

    \input{src/abstract}

    \bigskip

    \input{src/introduction}

    \input{src/assumptions-and-notation}

    \input{src/budgeted-mlmc}

    \input{src/multi-x-fem}

    \input{src/numerical-experiments}

    \input{src/discussion-and-outlook}

    \input{src/acknowledgement}

    \bibliographystyle{ieeetr}

    \bibliography{bibliography}

\end{document}

%% file: src/abstract.tex
\begin{abstract}
    We present a high-performance budgeted multi-level Monte Carlo method
    for estimates on the entire spatial domain of multi-PDE problems with random input data.
    The method is designed to operate optimally within memory and CPU-time constraints
    and eliminates the need for a priori knowledge of the problem's regularity
    and the algorithm's potential memory demand.
    To achieve this, we build on the budgeted multi-level Monte Carlo framework
    [Baumgarten et al., \textit{SIAM/ASA J.~Uncertain.}, \textbf{12}, 2024]
    and enhance it with a sparse multi-index update algorithm operating on a dynamically
    assembled parallel data structure to enable estimates of the full field solution.
    We demonstrate numerically and provide mathematical proof
    that this update algorithm allows computing the full spatial domain
    estimates at the same CPU-time cost as a single quantity of
    interest, and that the maximum memory usage is similar to the
    memory demands of the deterministic formulation of the problem
    despite solving the stochastic formulation in parallel.
    We apply the method to a sequence of interlinked PDE problems,
    ranging from a stochastic partial differential equation for sampling
    random fields that serve as the diffusion coefficient in an
    elliptic subsurface flow problem, to a hyperbolic PDE
    describing mass transport in the resulting flux field.
\end{abstract}

%% file: src/introduction.tex
\section{Introduction}
\label{sec:introduction}

The acquisition and management of comprehensive data is a
significant challenge for complex computational systems, often
imposed by the limited amount of
available memory to hold and the lack of computing power to process this data.
Hence, the given computational resources act as a constraint on the
computational result.
This certainly is the case in
Uncertainty Quantification (UQ),
particularly when applied to challenging models such as
high-dimensional Partial Differential Equations (PDEs) with an
uncertain input due to measurement noise or parameter fluctuations.
Here, the sheer volume of data already introduces the difficulty.

A common approach to UQ for PDEs is to employ model order reduction and surrogate modeling.
Additionally, the focus has been on the isolation, prioritization and estimation
of quantities of interest (QoIs) associated with the system’s solution,
rather than spatially fully resolved estimates.
This paper intentionally moves away from model order reduction and from computing QoIs,
seeking to demonstrate that, through careful algorithmic design,
estimating statistical moments of a random field (RF)
solution on the complete spatial domain of the system
is no additional burden in terms of memory footprint or computing time.
We explain how this is shown mathematically
but also introduce the implementation
achieving this in practice.

While the theory on estimating the solution on the entire spatial domain is known,
we have not found a computational method that can reach this goal in a memory-efficient
and fully parallelized manner.
Our approach uses Multi-Level Monte Carlo (MLMC) methods, introduced
in~\cite{giles2008multilevel, heinrich2001multilevel} and further developed for
elliptic~\cite{barth2011multi, charrier2013finite, cliffe2011multilevel, teckentrup2013further}
and hyperbolic PDEs~\cite{mishra2012multi, mishra2016multi}.
We 
build upon {the theoretical work in}~\cite{bierig2015convergence, bierig2016estimation}
and combine it with the Budgeted Multi-Level Monte Carlo (BMLMC) method~\cite{baumgarten2024fully}.

One notable advantage of the MLMC method lies in its non-intrusive implementation.
This means that existing software, e.g.~finite element (FE) implementations,
can be leveraged without requiring substantial modifications to the code base.
{In contrast to other works, in particular \cite{baumgarten2024fully},
this paper sacrifices} the non-intrusive nature inherent in MLMC
by combining it with a parallel update protocol.
In doing so, we achieve precise control over the algorithm's memory consumption,
and thereby gain the capability to compute estimates of the solution across
the entire spatial domain on the fly without postprocessing
and without performance degradation.

The underpinning methodology, i.e.~the BMLMC method, incorporates various
concepts from continuation MLMC~\cite{collier2015continuation}
and adaptive MLMC~\cite{giles2015multilevel}.
The central idea is to impose the available computational resources as a budget
while minimizing the estimated error.
This transforms the search for the optimal sample sequence of the MLMC method
into an NP-hard knapsack problem.
A similar point of view is also taken in~\cite{schaden2020multilevel} and the
multi-fidelity method~\cite{peherstorfer2018survey}.
However, a crucial distinction between the BMLMC method and
the multi-fidelity method is that there is no initial cost
associated with the construction of the low-fidelity models.
Furthermore, the method is designed to use preassigned computational
resources on a parallel computing cluster optimally
through distributed dynamic programming (DDP) techniques.

Importantly, {the budget of the BMLMC method,
which in \cite{baumgarten2024fully} is constrained only} by CPU-time
$\abs{\cP} \cdot {\rT}_{\rB}$ with $\abs{\cP}$ being the total number of available CPUs
and $\rT_{\rB}$ {the computational time budget,
is extended} by incorporating an additional memory constraint $\mathrm{Mem}_{\rB}$.
As a result, this extension eliminates the need for prior
knowledge about the problem and the discretization,
ensuring successful execution of the method.

The theoretical framework for estimating statistical moments of RF solutions on the entire
spatial domain via MLMC is well-established~\cite{barth2011multi,bierig2015convergence,bierig2016estimation}.
However, implementation for large-scale applications
on distributed memory introduces new challenges.
Careless algorithmic design may lead to a dramatic increase in memory
consumption, degradation of the computational performance, or
require storage of huge amounts of data on a hard-drive.
In particular, a novel communication protocol
with appropriate distribution rules
is necessary to sustain data integrity throughout the simulation.
Thus, our {novel contributions -- significantly extending
  \cite{baumgarten2024fully} -- are summarized}
as follows:

\paragraph{Memory Constraint and Sparse Multi-Index Update Algorithm}
Our work offers a mathematical proof asserting that, through the use of a sparse multi-index update algorithm,
the memory usage remains comparable to that of solving the deterministic problem with the same spatial resolution.
This justifies imposing a memory constraint on the MLMC method, which
is a significant improvement {on \cite{baumgarten2024fully}}
in robustness and usability in computationally challenging settings.
We achieve this even for {full, spatially distributed} domain estimates without compromising
computing performance due to excessive communication overheads.
By extending the techniques to stably update statistical moments, first
published in~\cite{chan1983algorithms, welford1962note}
and further developed in~\cite{pebay2016numerically}, the RF estimates
can be computed in a single pass through the samples.
To support our claims, we provide empirical evidence in our numerical experiments.
A crucial element of the update algorithm is the introduction of a communication split index, denoted by $s$.
When combined with the discretization level, represented by the index $\ell$,
this leads to the method essentially being a variant of a sparse multi-index method.
Although our algorithm shares similarities with existing multi-index
methods~\cite{dick2019improved, haji2016multi, haji2018multilevel},
and sparse grids~\cite{bungartz2004sparse},
the construction of the multi-index set differs.
We build upon the sparsity inherent in the multi-sample
parallelization scheme of~\cite{baumgarten2024fully}.

\paragraph{Multi-PDE Problems and Parallel Sampling of RFs}
The proposed method can be applied to forward UQ problems comprising several
intricate PDE models with high-dimensional random input data.
The numerical study consists of a series of computations, starting with the generation of Gaussian RFs
by solving Stochastic Partial Differential Equations (SPDEs).
This is followed by the computation of the solution of a subsurface flow problem
and a hyperbolic mass transport problem in the resulting random flux field.
Similar applications are investigated
in~\cite{baumgarten2021parallel, kumar2018multigrid, muller2013multilevel}.

One significant benefit of our methodology and of the proposed distributed data
structure is that we can remain agnostic with regard to the specific finite element (FE)
space used for each partial differential equation (PDE).
In particular, we use Lagrange elements for the SPDE sampling,
a mixed formulation with Raviart-Thomas elements for the subsurface flow and
discontinuous Galerkin elements for the hyperbolic mass transport.
We make use of the SPDE approach for sampling GRFs with Mat{\'e}rn covariance,
as introduced in~\cite{lindgren2011explicit}.
To realize the sampling on the same distributed data structure without
enlarging the domain, we rely on new developments published in~\cite{kutri2024dirichlet}.
This is essential to satisfy the imposed memory constraints.

Moreover, we present a formalism that enables the coupling of disparate PDEs
within a unified framework, employing a currying technique.
This theoretical formulation of the computational chain is motivated by our software
design and underlines the tight correspondence between theory and implementation.
In this paper, we use the term ‘Multi-X Finite Element Method’ (or MX-FEM)
to encapsulate this functionality.
It comprises three key elements:
the multi-\textit{sample} FE space of~\cite{baumgarten2024fully}, the
sparse multi-\textit{index} FE update algorithm 
and the given multi-\textit{PDE} problem.
This term is inspired by the work in~\cite{beck2020hp, krais2021flexi}.

\paragraph{Possible Applications}
{Finally, we conclude with a summary of
possible applications of the methodology: (i)}
Computing central moments across the full spatial domain is essential
for accurate \textit{climate simulations and weather forecasts}.
These models, based on complex PDEs, are computationally expensive and
prone to uncertainty. {(ii) \ Uncertainties} in \textit{material science} can drastically affect the material properties.
Propagating and stochastically homogenizing these uncertainties is key to understanding the material.
Techniques from~\cite{khristenko2020statistical, duswald2024finite} for sampling
realistic structures can be easily integrated into our SPDE-based approach.
{(iii)} \ \textit{Gradient descent methods for
uncertain, PDE-constrained, distributed optimal control problems},
as developed in~\cite{geiersbach2020stochastic, martin2021complexity},
require the estimation of stochastic gradients,
which are RFs over the spatial domain.
This application originally motivated developing
the full field estimator presented here.

\paragraph{Outline}
In~\Cref{sec:assumptions-and-notation}, we establish the necessary assumptions and notation.
\Cref{sec:full-spatial-domain-budgeted-multi-level-monte-carlo-method}, mainly based
on~\cite{baumgarten2024fully} and~\cite{bierig2016estimation}, introduces the theoretical framework.
\Cref{subsec:introduction-to-multi-level-monte-carlo}
recalls the MLMC idea and introduces a posteriori error estimators
which can be computed via accumulative updates as shown in~\Cref{lemma:accumulation-of-z}.
These updates are used in the BMLMC method which is discussed
in~\Cref{subsec:budgeted-multi-level-monte-carlo-algorithm}
and finally summarized in~\Cref{alg:bmlmc}.
\Cref{sec:multi-x-finite-elelmnt-method} recalls the definition of a multi-\textit{sample}
FE space in~\Cref{subsec:multi-sample-finite-element-space} and explains
the sparse multi-\textit{index} FE update algorithm
in~\Cref{subsec:sparse-multi-index-finite-element-update-algorithm}.
The main result in this section is~\Cref{lemma:memory-consumption} which enables us to impose the memory constraint.
As another consequence, we deduce a lower bound on the smallest achievable root mean squared (RMSE) error
in~\Cref{corollary:bmlmc-corollary}, {complementing} the upper bound derived in~\cite{baumgarten2024fully}.
Furthermore, we discuss the multi-\textit{PDE} FE simulations formulated with a currying technique
in~\Cref{subsec:multi-pde-finite-element-simulations}, and most notably present \Cref{alg:mx-fem} with
memory efficient Gaussian RF sampling based on the ideas in~\cite{kutri2024dirichlet}.
In~\Cref{sec:numerical-experiments}, we present our empirical evidence and the numerical experiments
based on the software M++~\cite{baumgarten2021parallel} in
version~\cite{wieners2024mpp341}.
We finish with a short discussion in~\Cref{sec:discussion-and-outlook}.

%% file: src/assumptions-and-notation.tex
\section{Assumptions and Notation}
\label{sec:assumptions-and-notation}

We aim for a methodology to estimate RF solutions for intricate PDE models
where the randomness is inherited from the input data.
Our notation closely follows~\cite{baumgarten2023fully, baumgarten2024fully},
while definitions and the theoretical frameworks are similar to the ones
in~\cite{bierig2015convergence, bierig2016estimation}.
We assume that the problem of interest is imposed on
a bounded polygonal domain $\cD \subset \RR^d$,
with spatial dimension $d \in \set{1, 2, 3}$, and
on a complete probability space $(\Omega, \mathcal{F}, \PP)$.
The goal is to estimate moments of the solution on
the full spatial domain $\cD \subset \RR^d$ to some PDE
with randomly distributed input data.
Thus, we search for the RF solution $\bu \colon \Omega \times \cD \rightarrow \RR^J$,
$J \in \NN$ being the number of components of the solution field,
such that it satisfies the well-posed model of multiple coupled PDEs
\begin{equation}
    \label{eq:generic-pde-model}
    \cG[\omega] (\bu(\omega, \bx)) = \bb(\omega, \bx)
\end{equation}
where $\omega \in \Omega$ corresponds to an outcome
in the probability space and $\bx \in \cD$ denotes the spatial location.
The generic (non-linear) operator $\cG[\omega]$, representing the coupled PDE models,
and the forcing term $\bb(\omega, \bx)$ are subject to randomness due to uncertainty in the input data.
Therefore, we are interested in finding estimates to the above RF solution
in the Bochner space $\rL^{2}(\Omega, V)$ which contains all
$\rL^2$-integrable maps from the probability space $(\Omega, \mathcal{F}, \PP)$
to a Banach space $V$ and has the norm
\begin{align*}
    \norm{\bv}_{\rL^{2}(\Omega, V)}^2 \coloneqq
    \int_{\Omega} \norm{\bv(\omega)}_V^2 \rd \PP(\omega)
\end{align*}
Whenever we use this space as a solution space to some PDE, we assume that $V$ is a Hilbert space
with the inner product $\sprod{\cdot, \cdot}_V$.

To compute approximate estimates,
we assume that (i): every outcome $\omega^{(m)} \in \Omega$,
with $m=1, \dots, M$ as sample index,
can be represented by a finite dimensional vector $\by^{(m)} = (y_1, \dots, y_K)^\top \in \Xi \subset \RR^K$
and (ii): a FE approximation $\bu_\ell(\by^{(m)}, \cdot) \in V_{\ell}$
to the true solution $\bu(\omega, \cdot) \in V$, where $V_{\ell}$
is some appropriate FE space to $V$.

For example, an approximation for the mean field can be computed by a simple
Monte Carlo method of FE solutions, i.e.,
\begin{equation}
    V \ni \EE[\bu]
    \coloneqq \int_{\Omega} \bu(\omega, \cdot) \rd \PP
    \approx \frac{1}M \sum_{m=1}^M \bu_\ell(\by^{(m)}, \cdot)
    \eqqcolon E_{M}^{\text{MC}}[\bu_{\ell}] \in \rL^1(\Omega, V_\ell)
    \,,
    \label{eq:mc-estimator}
\end{equation}
where all $\by^{(m)} \in \Xi \subset \RR^K$ are assumed to be independent and identically distributed.
Note that the Monte Carlo estimator $E_{M}^{\text{MC}}[\bu_\ell] \in \rL^1(\Omega, V_\ell)$,
unlike the true mean field $\EE[\bu] \in V$, is a RF.
For simplicity, we will drop the explicit dependency on
$\bx \in \cD$ and $\omega \in \Omega$, but emphasize that we work on RFs throughout the paper.
We also occasionally write $\bu_\ell^{(m)}$ for short instead of $\bu_\ell(\by^{(m)}, \cdot)$.

As in~\cite{bierig2015convergence}, we quantify the quality of approximation~\eqref{eq:mc-estimator}
by the mean squared error (MSE), given by
\begin{align*}
    \mathrm{err}^{\mathrm{MSE}}
    \coloneqq \EE \squarelr{\norm{E_{M}^{\text{MC}}[\bu_\ell] - \EE[{\bu}]}_V^2}
    = \underbrace{M^{-1} \norm{\bu_\ell - \EE[\bu_\ell]}^2_{\rL^2(\Omega, V)}}_{\eqqcolon \err^{\mathrm{sam}}}
    + \underbrace{\norm{\EE[\bu_\ell - \bu]}_{V \phantom{(}}^2}_{\eqqcolon \err^{\mathrm{num}}}.
\end{align*}

To quantify the cost of approximation~\eqref{eq:mc-estimator}, we consider the $\epsilon$-cost $\rC_{\epsilon}$,
which represents the computational resources required to achieve a root mean squared error
$\mathrm{err}^{\mathrm{RMSE}} \coloneqq ({\mathrm{err}^{\mathrm{MSE}}})^{1/2}$ smaller than $\epsilon > 0$.
In~\cite{baumgarten2024fully}, this cost is quantified in units of CPU seconds,
a measurement of a real world quantity denoted by the random number $\rC^{\rC \rT}$.
In this paper, we extend this by a run-time measurement of the memory footprint in megabytes
and denote this cost by $\rC^{\mathrm{Mem}}$.
Typically, the theory is formulated for quantities of interest (QoI)
$\rQ \colon V \rightarrow \RR$ rather than the complete RF.
To be able to compare both approaches, we consider in this paper
also $\rQ_\ell \coloneqq \norm{\bu_\ell}_V$.

In the upcoming~\Cref{sec:full-spatial-domain-budgeted-multi-level-monte-carlo-method},
we recall the MLMC method which gets its computational advantage from the variance reduction
in estimating the RFs and QoIs
\begin{align}
    \label{eq:definition-vell}
    \bv_\ell \coloneqq \bu_\ell - \rP_{\ell - 1}^{\ell} \bu_{\ell - 1} \quad \bv_0 \coloneqq \bu_0
    \quad \text{ and } \quad
    \rY \coloneqq \rQ_{\ell} - \rQ_{\ell-1} \quad \rY_0 \coloneqq \rQ_{0},
\end{align}
respectively.
Here, we use a linear isometric
projection operator $\rP_{\ell - 1}^{\ell} \colon V_{\ell - 1} \rightarrow V_{\ell}$
which maps the FE solution $\bu_{\ell - 1} \in V_{\ell - 1}$ to the FE space $V_{\ell}$
(cf.~\cite{briant2023filtered} for a similar approach).
Finally, we collect all assumptions required for the MLMC estimation.

\begin{assumption}
    \label{assumption:mlmc}
    For the exponents $\alpha_{\bu}, \alpha_{\rQ}, \beta_{\bv}, \beta_{\rY}, \gamma_{\rC\rT}, \gamma_{\mathrm{Mem}} > 0$
    and the constants $c_{{\bu}}, c_{{\rQ}}$, $c_{{\bv}}, c_{{\rY}}$, $c_{{\rC\rT}}, c_{{\mathrm{Mem}}} > 0$
    all assumed to be independent of $\ell$, the FE approximation to problem~\eqref{eq:generic-pde-model}
    with the mesh diameter $h_\ell =2^{-\ell} h_0$ satisfies:

    \vspace{-2mm}

    \hspace{-3mm}
    \begin{minipage}{0.51\textwidth}
        \begin{subequations}
            \label{eq:assumptions-random-field}
            \begin{align}
                \label{eq:assumption-alpha-u}
                \norm{\EE[{\bu}_\ell - \bu]}_V \,\,\,\, &\leq \,\,\,\, c_{{\bu}} h_\ell^{\alpha_{\bu}} \\
                \label{eq:assumption-beta-v}
                \norm{\bv_\ell - \EE[\bv_\ell]}^2_{\rL^2(\Omega, V)} &\leq \,\,\,\, c_{{\bv}} h_\ell^{\beta_{\bv}} \\
                \label{eq:assumption-gamma-mem}
                \rC^{\mathrm{Mem}} (\bv_\ell) \quad\,\,\,\, &\leq c_{{\mathrm{Mem}}} h_\ell^{-\gamma_{\mathrm{Mem}}}
            \end{align}
        \end{subequations}
    \end{minipage}
    \begin{minipage}{0.45\textwidth}
        \begin{subequations}
            \label{eq:assumptions-qoi}
            \begin{align}
                \label{eq:assumption-alpha-Q}
                \abs{\EE[{\rQ}_\ell - \rQ]} \,\,\,\,\,\, &\leq \,\, c_{{\rQ}} h_\ell^{\alpha_{\rQ}} \\
                \label{eq:assumption-beta-Y}
                \EE[(\rY_\ell - \EE[\rY_\ell])^2] &\leq \,\, c_{{\rY}} h_\ell^{\beta_{\rY}} \\
                \label{eq:assumption-gamma-CT}
                \rC^{\rC\rT} ({\bv}_\ell^{(m)}) \,\,\,\,\,\, &\leq c_{{\rC\rT}} h_\ell^{-\gamma_{\rC\rT}}
            \end{align}
        \end{subequations}
    \end{minipage}
\end{assumption}

\begin{remark}
    \label{remark:assumptions}
    \begin{itemize}
        \item[(a)]
        The assumptions~\eqref{eq:assumptions-qoi} are the foundation for~\cite{baumgarten2024fully}.
        Here, we complement them by~\eqref{eq:assumptions-random-field},
        bearing in mind the importance of the memory footprint in the RF setting.

        \item[(b)]
        Following~\cite{giles2015multilevel}, we occasionally use a simplification to~\eqref{eq:assumption-alpha-u}
        and~\eqref{eq:assumption-alpha-Q}, considering $\EE[\bv_\ell]$ and $\EE[\rY_\ell]$, respectively.

        \item[(c)]
        In practice, the cost assumption~\eqref{eq:assumption-gamma-CT} is
        non-homogeneous and fairly challenging to measure in a parallel algorithm.
        We therefore consider a CPU-time estimator $\widehat{\rC}^{\mathrm{CT}}_{\ell}$ rather than the
        CPU-time cost of a single sample $\rC^{\rC\rT}({\bv}_\ell^{(m)})$.

        \item[(d)]
        As we will explore in~\Cref{subsec:sparse-multi-index-finite-element-update-algorithm},
        it is more useful in our framework to measure the memory consumption
        in~\eqref{eq:assumption-gamma-mem} in terms of FE cells $\abs{\cK_\ell}$ rather than
        the mesh-size $h_\ell$.
        This is because parallelization over the samples is achieved by introducing further FE cells
        that increase the measured memory footprint $\widehat{\rC}_{\ell}^{\mathrm{Mem}}$ on level $\ell$ as well.

        \item[(e)]
        It will also be more appropriate to think about~\eqref{eq:assumption-alpha-u} in terms of
        FE cells $\abs{\cK_\ell}$ rather than the mesh-size $h_\ell$ to
        show the lower bound of~\Cref{corollary:bmlmc-corollary}.
    \end{itemize}
\end{remark}

%% file: src/budgeted-mlmc.tex
\section{Budgeted Multi-Level Monte Carlo Method}
\label{sec:full-spatial-domain-budgeted-multi-level-monte-carlo-method}

In this section, we introduce the MLMC algorithm for the estimation
of mean-fields on the complete spatial domain $\cD \subset \RR^d$.
We begin with~\Cref{subsec:introduction-to-multi-level-monte-carlo} by covering the theory,
further notation and the used a posteriori error estimators.
Following on from that, in~\Cref{subsec:budgeted-multi-level-monte-carlo-algorithm},
we introduce~\Cref{alg:bmlmc} to optimally compute the entire spatial domain approximation
within a given CPU-time and memory budget.
Noteable results are~\Cref{lemma:relation-Y-v} and~\Cref{corollary:relation-beta},
which connect the assumptions~\eqref{eq:assumption-beta-v} and~\eqref{eq:assumption-beta-Y},
as well as the formulae given in~\Cref{lemma:accumulation-of-z}, which are used to argue that
the BMLMC method for RF estimates satisfies Bellman's optimality principle.

\input{src/introduction-to-multi-level-monte-carlo}

\input{src/budgeted-multi-level-monte-carlo-algorithm}

%% file: src/introduction-to-multi-level-monte-carlo.tex
\subsection{Introduction to Multi-Level Monte Carlo}
\label{subsec:introduction-to-multi-level-monte-carlo}

We start by recapitulating the classical multi-level Monte Carlo method.
To this end, we introduce models based on
different discretization levels $\ell=0,\dots,L$,
giving a hierarchy of nested FE meshes with the cells $\{\cK_{\ell}\}_{\ell=0}^L$
and with decreasing mesh widths $h_\ell =  2^{-\ell} h_0$.
For the fixed finest approximation on level $L$,
the approximate mean field $\EE[\bu_L] \in V_L$
can be expanded using~\eqref{eq:definition-vell} into
a telescoping sum over the levels:
\begin{align*}
    \EE[{\bu}_L]
    &= \rP_{0}^{L} \, \EE \squarelr{{\bu_{0}}}
    + \sum_{\ell=1}^L \rP_{\ell}^{L} \,
    \EE\squarelr{{\bu}_{\ell} - \rP_{\ell-1}^{\ell} {\bu}_{\ell-1}}
    = \sum_{\ell=0}^L \rP_{\ell}^{L} \EE[\bv_{\ell}]
\end{align*}
Here, we make use of the linear isotropic
transfer operators $\rP_{\ell}^L \colon V_\ell \rightarrow V_L$ defined by
the product $\rP_{\ell}^L \coloneqq \prod_{\ell'=\ell}^{L-1} P_{\ell'}^{\ell'+1}$.
Each expected value of $\bv_\ell$ can be
estimated individually with a MC method, resulting in what we refer to as
the full mean-field MLMC estimator in $\rL^1(\Omega, V_\ell)$
\begin{equation}
    \label{eq:mlmc-estimator-full-solution}
    E^{\text{MLMC}}_{\{M_\ell\}_{\ell=0}^L}[\bu_L]
    = \sum_{\ell=0}^L \rP_{\ell}^{L} E_{M_\ell}^{\text{MC}}[\bv_{\ell}]
    \quad \text{with} \quad
    E_{M_\ell}^{\text{MC}}[\bv_{\ell}] = \frac{1}{M_\ell} \sum_{m=1}^{M_\ell} \bv_\ell^{(m)},
\end{equation}
where $\{M_\ell\}_{\ell=0}^L$ denotes the number of samples on each level.
In order to satisfy assumption~\eqref{eq:assumption-beta-v}, we emphasize that it
is important to use the same input sample $\by^{(m)}$ in computing each
$\bv_\ell^{(m)} \coloneqq {\bu}_{\ell}^{(m)} - \rP_{\ell-1}^{\ell} {\bu}_{\ell - 1}^{(m)}$,
up to interpolation or projection.

As the estimators for different levels $\ell$ are independent,
the MSE admits the decomposition
(cf.~e.g.~\cite[Theorem 3.1]{bierig2015convergence})
\begin{equation}
    \label{eq:mse-mlmc}
    \mathrm{err}^{\mathrm{MSE}} \roundlr{E^{\text{MLMC}}_{\{M_\ell\}_{\ell=0}^L}[\bu_L]}
    = \underbrace{\sum_{\ell=0}^L M_{\ell}^{-1}  \norm{{\bv_{\ell} - \EE[\bv_\ell]}}^2_{\rL^2(\Omega, V)}}_{\coloneqq {\err}^{\mathrm{sam}}}
    + \,\, {\underbrace{\norm{\EE[\bu_{L} - \bu]}_V^2}_{\coloneqq {\err}^{\mathrm{num}}}}.
\end{equation}
Under~\Cref{assumption:mlmc}, the cost to reach an error accuracy of $0 < \epsilon < \re^{-1}$ for~\eqref{eq:mse-mlmc}
can be controlled by an appropriate choice for the number of samples on each level $\set{M_\ell}_{\ell=0}^L$.
An optimal choice gives (cf. \cite[Theorem 3.2]{bierig2015convergence} or~\cite{cliffe2011multilevel, giles2008multilevel})
\begin{align*}
    \err^{\mathrm{RMSE}} \roundlr{E^{\mathrm{MLMC}}_{\{M_\ell\}_{\ell=0}^L}[\bu_L]} < \epsilon
    \quad \text{and} \quad
    \rC_\epsilon \roundlr{E^{\mathrm{MLMC}}_{\{M_\ell\}_{\ell=0}^L}[\bu_L]} \lesssim
    \begin{cases}
        \epsilon^{-2} & \beta > \gamma, \\
        \epsilon^{-2} \log(\epsilon)^2 & \beta = \gamma, \\
        \epsilon^{-2-(\gamma-\beta)/\alpha} & \beta < \gamma.
    \end{cases}
\end{align*}
Here, $\alpha, \beta, \gamma > 0$ can follow~\eqref{eq:assumptions-random-field}
or~\eqref{eq:assumptions-qoi} depending on the computational goal.

The approach of the BMLMC method is somewhat dual to this classical approach,
as it aims to minimize the error with respect to a computational budget and not the
computational cost for a prescribed accuracy.
While this is motivated by the application in HPC-settings and the fact that computational resources are always limited,
it also renders the optimization for finding the optimal sequence of samples an NP-hard combinatorial
resource allocation problem.

Before turning to the algorithmic details to solve this optimization problem,
we note that the MSE~\eqref{eq:mse-mlmc} itself is intractable.
Thus, we estimate it using separate estimators for the
two components, i.e.,
$\widehat{\err}^{\mathrm{MSE}} = \widehat{\err}^{\mathrm{sam}} + \widehat{\err}^{\mathrm{num}} \approx
{\err}^{\mathrm{sam}} + {\err}^{\mathrm{num}} = {\err}^{\mathrm{MSE}}$.
We compute $\widehat{\err}^{\mathrm{num}}$ by extending the technique
introduced in~\cite{giles2015multilevel} to mean-fields by exploiting
assumption~\eqref{eq:assumption-alpha-u}.
Specifically, we use
\begin{equation}
    \label{eq:squared-bias-estimate-mlmc}
    \widehat{\err}^{\text{num}} \coloneqq
    \roundlr{\max \set{\frac{{\|E_{M_{\ell}}^{\text{MC}}[\bv_{\ell}]\|_V}}{2^{{\widehat{\alpha}}_{\bu}} - 1}
    \, 2^{-{\widehat{\alpha}}_{\bu} (L - \ell)} \colon \ell = 1, \dots, L}}^2.
\end{equation}
To increase robustness, the exponent $\widehat{\alpha}_{\bu}$, which is an estimate for $\alpha_{\bu}$,
is determined by solving a regression problem, using data collected on the fly.
It is given through
\begin{equation}
    \label{eq:alpha-fit}
    \min_{(\widehat{\alpha}_{\bu}, \, \widehat{c}_{\bu})} \quad \sum_{\ell=1}^L
    \Big(\log_2 {\|E_{M_{\ell}}^{\text{MC}}[\bv_{\ell}]\|_V} + \widehat{\alpha}_{\bu} \ell - \widehat{c}_{\bu} \Big)^2.
\end{equation}
The same formula is used to estimate the other exponents in~\Cref{assumption:mlmc}.
We further compute $\widehat{\err}^{\mathrm{sam}}$ by
\begin{equation}
    \label{eq:sampling-error-estimator}
    \widehat{\err}^{\mathrm{sam}} \coloneqq \sum_{\ell=0}^L (M_{\ell}^2 - M_{\ell})^{-1} \, z^2_{\ell}[\bv_\ell],
    \,\,\,\,
    \text{where}
    \,\,\,\,
    z^2_{\ell}[\bv_\ell] \coloneqq \sum_{m=1}^{M_{\ell}} \norm{\bv_{\ell}^{(m)} - E^{\text{MC}}_{M_{\ell}}[\bv_{\ell}]}_V^2
\end{equation}
is the second order sum in the estimator for $\norm{{\bv_{\ell} - \EE[\bv_\ell]}}^2_{\rL^2(\Omega, V)}$
including Bessel's correction.
Similar to this sampling error estimator, a variance estimator for QoIs~(cf.~\cite{baumgarten2024fully})
and a component-wise sample variance field are computed in $\RR$ and $\rL^2(\Omega, V_\ell)$, respectively.
The sample variance field estimator is denoted by:
\begin{align*}
    V^{\text{MC}}_{M_{\ell}}[\bu_\ell]
    \coloneqq (M_{\ell} - 1)^{-1} S^{2}_{M_{\ell}}[\bu_\ell]
    \quad \text{with} \quad
    S^{2}_{M_{\ell}}[\bu_\ell] \coloneqq \sum_{m=1}^{M_{\ell}}
    \roundlr{\bu_\ell^{(m)} - E^{\text{MC}}_{M_{\ell}}[\bu_\ell]}^2
\end{align*}
Lastly, we state a relation between the sampling error of the RF and the QoI.

\begin{lemma}
    \label{lemma:relation-Y-v}
    The variance of $\rY_\ell$ and the Bochner norm of $\bv_\ell$ are related by:
    \begin{align*}
        \VV[\rY_\ell] \coloneqq \EE[(\rY_\ell - \EE[\rY_\ell])^2],
        \quad
        \VV[\rY_\ell]
        \leq \norm{\bv_\ell - \EE[\bv_\ell]}_{\rL^2(\Omega, V)}^2
    \end{align*}
    Furthermore, the bias of the QoI and RF estimate share an upper bound:
    \begin{align*}
        \abs{\EE[\rQ_\ell - \rQ]} \leq \EE[\norm{\bu_\ell - \bu}_V]
        \quad \text{and} \quad
        \norm{\EE[\bu_\ell - \bu]}_V \leq \EE[\norm{\bu_\ell - \bu}_V]
    \end{align*}
\end{lemma}

\begin{proof}
    Using the isometry of the projection operator $\|\bu_{\ell-1}\|_V = \|{\rP_{\ell-1}^{\ell} \bu_{\ell - 1}}\|_V$,
    the definitions of $\rQ_\ell \coloneqq \norm{\bu_\ell}_V$, $\rY_\ell \coloneqq \rQ_\ell - \rQ_{\ell-1}$ and
    $\bv_\ell \coloneqq \bu_\ell - \rP_{\ell-1}^{\ell} \bu_{\ell - 1}$,
    as well as the reversed triangle inequality and Jensen's inequality,
    we obtain
    \begin{align*}
        \VV[\rY_\ell] \leq \VV[\|\bu_\ell - \rP_{\ell-1}^{\ell} \bu_{\ell-1}\|_V]
        \leq \EE[(\norm{\bv_\ell}_V - \norm{\EE[\bv_\ell]}_V)^2]
        \leq \norm{\bv_\ell - \EE[\bv_\ell]}_{\rL^2(\Omega, V)}^2
    \end{align*}
    The second bound follows from the triangle and the reversed triangle inequality.
\end{proof}

\begin{corollary}
    \label{corollary:relation-beta}
    \hspace{8mm}
    If we strengthen assumption~\eqref{eq:assumption-beta-Y} to $\VV[\rY_\ell] \sim h_\ell^{\beta_{\rY}}$
    and~\eqref{eq:assumption-beta-v} to
    $\norm{\bv_\ell - \EE[\bv_\ell]}_{\rL^2(\Omega, V)} \sim h_\ell^{\beta_{\bv}}$,
    we obtain $\beta_{\rY} \leq \beta_{\bv}$,
    i.e., the variance reduction for the QoI is at least as fast
    as the sampling error reduction for the RF.
\end{corollary}

\noindent Empirical evidence follows in~\Cref{sec:numerical-experiments}.

%% file: src/budgeted-multi-level-monte-carlo-algorithm.tex
\subsection{Budgeted Multi-Level Monte Carlo Algorithm}
\label{subsec:budgeted-multi-level-monte-carlo-algorithm}

With the theoretical foundations of the MLMC method at hand,
the objective is to identify a solution with the smallest
mean squared error whilst remaining within a CPU-time
$\abs{\cP} \cdot \rT_{\rB}$ and a memory ${\mathrm{Mem}}_{\rB}$ budget,
where $\abs{\cP}$ is the total
amount of processing units and $\rT_{\rB}$ is
an overall time constraint.
Formally, this leads to the following optimization problem:

\begin{problem}[MLMC Knapsack]
    \label{problem:approximated-mlmc-knapsack}
    Find the largest level $L$ and the sequence of sample numbers $\{M_\ell\}_{\ell=0}^L$,
    such that the estimated MSE is minimized,
    while staying within the CPU-time $\abs{\cP} \cdot \rT_{\rB}$
    and a memory ${\mathrm{Mem}}_{\rB}$ budget
    \begin{subequations}
        \label{eq:knapsack-mlmc}
        \begin{align}
            \label{eq:knapsack-mlmc-mse}
            \min_{(L, \{M_\ell\}_{\ell=0}^L)} \quad \widehat{\err}^{\mathrm{MSE}}
            &= \widehat{\err}^{\mathrm{sam}} + \widehat{\err}^{\mathrm{num}} \\
            \text{s.t.} \quad
            \sum_{\ell = 0}^L M_{\ell} \widehat{\rC}^{\mathrm{CT}}_{\ell}
            &\leq \abs{\cP} \cdot \rT_{\rB} \,,
            \label{eq:knapsack-mlmc-ct} \\[1mm]
            \widehat{c}_{\mathrm{Mem}} \cdot \abs{\cK_{L}}
            &< {\mathrm{Mem}}_{\rB} \,.
            \label{eq:knapsack-mlmc-mem}
        \end{align}
    \end{subequations}
\end{problem}
The quantities $\widehat{\rC}^{\mathrm{CT}}_{\ell}$ and $\widehat{c}_{\mathrm{Mem}}$
in the constraints~\eqref{eq:knapsack-mlmc-ct} and~\eqref{eq:knapsack-mlmc-mem}
are measured as the algorithm runs.
$\widehat{\rC}^{\mathrm{CT}}_{\ell}$ represents the
expected time of a single sample on level $\ell$ (cf.~\Cref{remark:assumptions} c)) and $\widehat{c}_{\mathrm{Mem}}$
is a constant translating the number FE cells to memory usage
(cf.~\Cref{remark:assumptions} d)).
For details regarding constraint~\eqref{eq:knapsack-mlmc-ct},
we refer to~\cite{baumgarten2024fully}.
The derivation of~\eqref{eq:knapsack-mlmc-mem} is given
in~\Cref{subsec:sparse-multi-index-finite-element-update-algorithm}.
To efficiently address~\Cref{problem:approximated-mlmc-knapsack},
we propose~\Cref{alg:bmlmc}, leveraging distributed dynamic programming (DDP) techniques.
Specifically, the algorithm decomposes the initial problem into overlapping subproblems,
which are then solved and distributed according to an optimal policy.
This approach involves the reuse of memoized and shared results to guide the process.

To generate the overlapping subproblems,
we adopt a decreasing sequence of accuracies $\epsilon_{\ttti}$,
as in the continuation MLMC method outlined in~\cite{collier2015continuation}.
Each subproblem is tackled as described in~\cite{giles2015multilevel}
to determine the optimal number of samples
with an optimal distribution rule as given in~\cite{baumgarten2024fully}.
This strategy is essential to manage the substantial computational load and
to face the NP-hardness of the problem.

Since~\Cref{alg:bmlmc} is based on estimation rounds indexed by $\ttti$,
we add this index to quantities we previously introduced.
That is, $M_{\ttti, \ell}$ is the accumulated number of samples, $\widehat{\err}_{\ttti}^{\text{sam}}$
is the estimated sampling error and $\widehat{\err}_{\ttti}^{\text{num}}$
is the estimator bias up to estimation round $\ttti$.
Further, the estimated computing-time cost
\begin{equation}
    \label{eq:ct-cost-estimation}
    \widehat{\rC}^{\mathrm{CT}}_{\ttti, \ell} \approx M_{\ttti, \ell}^{-1} \sum_{m=1}^{M_{\ttti, \ell}}
    \rC^{\mathrm{CT}}(\bv^{(m)}_\ell)
\end{equation}
and the largest level $L_{\ttti}$ are also indexed by $\ttti$.
The sample mean $E^{\text{MC}}_{\ttti, \ell}[\cdot]$ and
the sample variance $V^{\text{MC}}_{\ttti, \ell}[\cdot]$
of the RFs $\bu_{\ell}$ and $\bv_\ell$,
as well as the random variables $\|\bu_{\ell}\|_V$
and $\|\bv_\ell\|_V$, must be updated in each round.
However, due to memory constraints,
only the real-valued random variables can
afforded to be stored over the run-time.

Parameters to initialize the algorithm include an initial guess for a small
initial sample sequence $\{M_{0, \ell}^{\text{init}}\}_{\ell=0}^{{L_0}}$,
the preassigned CPU-time budget $\abs{\cP} \cdot \rT_{\rB}$, the memory constraint ${\mathrm{Mem}}_{\rB}$,
and two algorithmic hyperparameters:
the variance-bias tradeoff $\theta \in (0, 1)$ and the error reduction factor $\eta \in (0, 1)$.
Additionally, prior to executing the algorithm, the PDE models must be selected,
achieved by defining the inner body of the block \texttt{MX-FEM} in~\Cref{alg:bmlmc}.
Further details on this can be found in~\Cref{subsec:multi-pde-finite-element-simulations}
and in \Cref{alg:mx-fem} for a specific example.
For now, we provide a explanation of high-level functional pseudo-code given in~\Cref{alg:bmlmc}.

\paragraph{Data structures}
The algorithm maintains two data structures to store its state:
\texttt{data\_distr}, a distributed data structure to store the spatially distributed solution estimates,
e.g., the sample mean $E^{\text{MC}}_{\ttti, \ell}[\bu_\ell]$ on the level $\ell$,
and \texttt{data\_shared}, stored in a shared manner, e.g.,
containing the error estimates $\widehat{\err}_{\ttti}^{\mathrm{sam}}$ and $\widehat{\err}_{\ttti}^{\text{num}}$,
to guide the optimization.
Hence, \texttt{data\_distr} holds huge amounts of spatial data while
\texttt{data\_shared} only contains key quantities such that every
processing unit can hold a copy of.

\paragraph{The INIT function}
The algorithm starts with an initialization round,
denoted by the first \texttt{INIT} function in~\Cref{alg:bmlmc},
which takes the initial sample sequence $\{M_{0, \ell}^{\text{init}}\}_{\ell=0}^{{L_0}}$ as an input argument.
It proceeds to solve the multi-PDE problem on each level $\ell$ by invoking the \texttt{MX-FEM} function
on level $\ell$ with $M_{0, \ell}^{\text{init}}$.
The returned distributed data packages $\Delta \texttt{data\_distr}_\ell$ contain
on each level $\ell$ the first RF solution estimates, i.e., in the first round it contains
\begin{align*}
    \Delta \texttt{data\_distr}_\ell =
    \set{E^{\text{MC}}_{0, \ell}[\bu_\ell], S^{2}_{0, \ell}[\bu_\ell],
        E^{\text{MC}}_{0, \ell}[\bv_\ell], S^{2}_{0, \ell}[\bv_\ell], \dots}.
\end{align*}
With this first collected data, the algorithm initializes its state, estimates errors and exponents
with~\eqref{eq:squared-bias-estimate-mlmc},~\eqref{eq:alpha-fit},~\eqref{eq:sampling-error-estimator},
and stores the results in the data structures \texttt{data\_distr} and \texttt{data\_shared}.
Formally, this is expressed in the \texttt{Welford} function.
For further details, we refer to~\cite{baumgarten2024fully, pebay2016numerically}
and the last paragraph in this section.

Following the initialization round, the algorithm invokes the second \texttt{BMLMC} function with an
updated time budget and a reduced estimated mean squared error as new error target, i.e.~with:
\begin{align*}
    \underbrace{\rT_1 \leftarrow \rT_{\rB} - \sum_{\ell=0}^{L_{0}} M_{0, \ell}^{\text{init}}
        \widehat{\rC}^{\mathrm{CT}}_{0, \ell}}_{
        \text{reduced time-budget as new constraint}}
    \qquad
    \underbrace{\epsilon_1 \leftarrow \eta \cdot \widehat{\err}_{0}^{\text{MSE}}}_{\text{reduced MSE as new target}}
\end{align*}

\paragraph{The BMLMC function}
The \texttt{BMLMC} function in~\Cref{alg:bmlmc} essentially guides the optimization process.
In essence, this function oversees the compliance to the
constraints, determines the necessity of a new level $L_{\ttti}$ or additional
samples, and calls the \texttt{MX-FEM} and \texttt{Welford} functions to
collect and integrate new data, respectively.
In particular, the error control is done by the following conditions
\begin{align*}
    \underbrace{\widehat{\err}^{\text{num}}_{\ttti - 1} \geq (1 - \theta) \, \epsilon_{\ttti}^2}_{\text{bias too big, new } L_{\ttti} \text{ needed}}
    \quad \text{or} \quad
    \underbrace{\widehat{\err}^{\mathrm{sam}}_{\ttti - 1} \geq \theta \epsilon_{\ttti}^2}_{\text{sampling error too big, new samples needed}}.
\end{align*}
If the first condition is violated, the algorithm increases the level $L_{\ttti}$ and collects estimates on that level.
If the second condition is violated, the algorithm computes the optimal number of samples for the next round.
This is given by
\begin{equation}
    \label{eq:optimal-delta-Ml-estimated}
    \set{\Delta M_{\ttti, \ell}}_{\ell=0}^{L_{\ttti}}
    \coloneqq \set{\max\{{M}_{\ttti, \ell}^{\text{opt}} - M_{\ttti-1, \ell}, \, 0\}}_{\ell=0}^{L_{\ttti}},
\end{equation}
which uses the optimal total sample amount, derived in~\cite{giles2015multilevel}.
Using the second order sum of~\eqref{eq:sampling-error-estimator},
it is given on each level $\ell$ by
\begin{equation}
    \label{eq:optimal-Ml-estimated}
    M_{\ttti, \ell}^{\text{opt}}
    = \ceil {
        \roundlr{\sqrt{\theta} \epsilon_{\ttti}}^{-2}\!\!
        \sqrt{\frac{z^2_{{\ttti - 1, \ell}}[\bv_\ell]}{(M_{\ttti - 1, \ell} - 1) \, \widehat{\rC}^{\rC\rT}_{\ttti - 1, \ell}}}
        \roundlr{
            \sum_{\ell'=0}^{L_{\ttti}}
            \sqrt{\frac{z^2_{{\ttti - 1, \ell'}}[\bv_{\ell'}] \, \widehat{\rC}^{\rC\rT}_{\ttti-1, \ell'}}{M_{\ttti - 1, \ell'} - 1}} \,\,
        }
    }.
\end{equation}
The second order sum can be computed cumulatively over the rounds
as shown in \Cref{lemma:accumulation-of-z} in the last paragraph of this section.

If a new level is added or new samples are needed,
the algorithm checks whether this can be satisfied within
the time and memory constraints.
First, it checks whether the new samples can be afforded with
$\sum_{\ell=0}^{L_{\ttti}} \Delta M_{\ttti, \ell} \widehat{\rC}^{\text{CT}}_{\ttti - 1, \ell} > \rT_{\ttti}$,
i.e., whether the upcoming estimation is too expensive and adjusts the target if necessary,
but also
\begin{align*}
    \underbrace{\rT_{\ttti} < 0.05 \rT_{\rB}}_{
        \text{reached time constraint, return estimates}}
    \text{and} \qquad
    \underbrace{\widehat{c}_{\text{Mem}} \cdot \abs{\cK_{L_{\ttti}}} > {\mathrm{Mem}}_{\rB}}_{
        \text{reached memory budgetd, return estimates}}
\end{align*}
whether the constraints prevent further optimization.
A practical check for the time constraint is to ensure that
we return the solution as soon as 95\% of the time budget has been exhausted.
For a derivation on the memory constraint, we refer
to~\Cref{subsec:sparse-multi-index-finite-element-update-algorithm}.
In either case, the algorithm has found a solution
to \Cref{problem:approximated-mlmc-knapsack}
when reaching one of its constraints.

\paragraph{The Welford function}
Lastly, the \texttt{Welford} function in~\Cref{alg:bmlmc} is responsible for updating the data structures.
The details to this function are given in~\cite{baumgarten2024fully, pebay2016numerically}.
However, we present a special case of~\cite[Prop. 3.1]{pebay2016numerically} for
accumulating $z_{\ttti, \ell}[\bv_\ell]$.

\begin{lemma}
    \label{lemma:accumulation-of-z}
    The second order sum in~\eqref{eq:sampling-error-estimator} can be computed cumulatively by
    \begin{align*}
        z_{\ttti, \ell}^2[\bv_\ell] =
        z_{\ttti-1, \ell}^2[\bv_\ell] + \Delta z_{\ttti, \ell}^2[\bv_\ell]
        + \frac{M_{\ttti - 1, \ell} \Delta M_{\ttti, \ell}}{M_{\ttti, \ell}^{\mathrm{opt}}} \norm{\delta_{\ttti,\ell}[\bv_\ell]}_V^2
    \end{align*}

    \vspace{-5mm}

    \begin{align*}
        \text{with} \,\,\,
        \delta_{\ttti,\ell}[\bv_\ell] = \Delta &E^{\mathrm{MC}}_{\ttti, \ell}[\bv_\ell] - E^{\mathrm{MC}}_{{\ttti-1, \ell}}[\bv_\ell]
        , \,\,\,
        \Delta E^{\mathrm{MC}}_{\ttti, \ell}[\bv_\ell] = \frac{1}{\Delta M_{\ttti, \ell}}
        \sum_{m=1 + M_{\ttti-1, \ell}}^{M_{\ttti, \ell}^{\mathrm{opt}}} \bv_{\ell}^{(m)}
        \,\,\, \text{and} \\
        &\Delta z_{\ttti, \ell}^2[\bv_\ell]
        = \!\!\!\! \sum_{m=1 + M_{\ttti-1, \ell}}^{M_{\ttti, \ell}^{\mathrm{opt}}} \!\!\!
        \tnorm{\bv_{\ell}^{(m)}- \Delta E^{\mathrm{MC}}_{\ttti, \ell}[\bv_\ell]}_V^2 .
    \end{align*}
\end{lemma}

\begin{proof}
    The second order sum in~\eqref{eq:sampling-error-estimator} can be separated into
    \begin{align*}
        z_{\ttti, \ell}^2
        &= \hspace{-1.3mm} \sum_{m=1}^{M_{\ttti-1,\ell}} \norm{\bv_{\ell}^{(m)}
            - \tfrac{M_{\ttti-1, \ell}}{M_{\ttti, \ell}^{\mathrm{opt}}} E^{\mathrm{MC}}_{{\ttti-1, \ell}}[\bv_{\ell}]
            - \tfrac{\Delta M_{\ttti, \ell}}{M_{\ttti, \ell}^{\mathrm{opt}}}\Delta E^{\mathrm{MC}}_{\ttti, \ell}[\bv_{\ell}] }_V^2 \\
        &\quad+ \sum_{m=M_{\ttti-1,\ell}+1}^{M_{\ttti, \ell}^{\mathrm{opt}}} \norm{\bv_{\ell}^{(m)}
            - \tfrac{M_{\ttti-1, \ell}}{M_{\ttti, \ell}^{\mathrm{opt}}} E^{\mathrm{MC}}_{{\ttti-1, \ell}}[\bv_{\ell}]
            - \tfrac{\Delta M_{\ttti, \ell}}{M_{\ttti, \ell}^{\mathrm{opt}}}\Delta E^{\mathrm{MC}}_{\ttti, \ell}[\bv_{\ell}] }_V^2. \\
    \end{align*}
    With $\Delta E^{\mathrm{MC}}_{\ttti, \ell}[\bv_\ell] = \delta_{\ttti,\ell}[\bv_\ell] + E^{\mathrm{MC}}_{{\ttti-1, \ell}}[\bv_\ell]$
    and $E^{\mathrm{MC}}_{{\ttti-1, \ell}}[\bv_\ell] = \Delta E^{\mathrm{MC}}_{\ttti, \ell}[\bv_\ell] - \delta_{\ttti,\ell}[\bv_\ell]$, it is:
    \begin{align*}
        z_{\ttti, \ell}^2
        &= \hspace{-1.3mm} \sum_{m=1}^{M_{\ttti-1,\ell}} \norm{\bv_{\ell}^{(m)} - E^{\mathrm{MC}}_{{\ttti-1, \ell}}[\bv_{\ell}] - \tfrac{\Delta M_{\ttti, \ell}}{M_{\ttti, \ell}^{\mathrm{opt}}} \delta_{\ttti,\ell}[\bv_\ell]}_V^2 \\
        &\quad+ \sum_{m=M_{\ttti-1,\ell}+1}^{M_{\ttti, \ell}^{\mathrm{opt}}} \norm{\bv_{\ell}^{(m)} - \Delta E^{\mathrm{MC}}_{{\ttti, \ell}}[\bv_\ell] + \tfrac{M_{\ttti - 1, \ell}}{M_{\ttti, \ell}^{\mathrm{opt}}} \delta_{\ttti,\ell}[\bv_\ell] }_V^2 \\
    \end{align*}
    The binomial theorem gives
    \begin{align*}
        z_{\ttti, \ell}^2
        &= \hspace{-1.3mm} \sum_{m=1}^{M_{\ttti-1,\ell}} \norm{\bv_{\ell}^{(m)} - E^{\mathrm{MC}}_{{\ttti-1, \ell}}[\bv_{\ell}]}_V^2 + \norm{\tfrac{\Delta M_{\ttti, \ell}}{M_{\ttti, \ell}^{\mathrm{opt}}} \delta_{\ttti,\ell}[\bv_\ell]}_V^2 \\
        &\quad+ \sum_{m=M_{\ttti-1,\ell}+1}^{M_{\ttti, \ell}^{\mathrm{opt}}} \norm{\bv_{\ell}^{(m)} - \Delta E^{\mathrm{MC}}_{{\ttti, \ell}}[\bv_\ell]}_V^2 + \norm{\tfrac{M_{\ttti - 1, \ell}}{M_{\ttti, \ell}^{\mathrm{opt}}} \delta_{\ttti,\ell}[\bv_\ell]}_V^2 \\[1mm]
        &= z_{\ttti-1, \ell}^2 + \Delta z_{\ttti, \ell}^2 + \tfrac{M_{\ttti - 1, \ell} \Delta M_{\ttti, \ell}}{M_{\ttti, \ell}^{\mathrm{opt}}} \norm{\delta_{\ttti,\ell}[\bv_\ell]}_V^2,
    \end{align*}
    where the mixed terms
    cancel out by expansion of the product and the definition of the estimators
    $E^{\mathrm{MC}}_{{\ttti-1, \ell}}[\bv_{\ell}]$ and $\Delta E^{\mathrm{MC}}_{{\ttti, \ell}}[\bv_\ell]$.
\end{proof}

Following the arguments in~\cite{baumgarten2024fully},
with this accumulative formula for $z_{\ttti, \ell}^2[\bv_\ell]$,
the algorithm satisfies Bellman's optimality principle,
i.e., each subproblem is solved optimally based on preexisting solutions.
Similar formulas are used for the other quantities in the data structures.

\begin{algorithm}
    \caption{BMLMC with MX-FEM}
    \label{alg:bmlmc}

    \vspace{-2mm}

    \begin{align*}
        &\texttt{data\_shared} = \tightset{\ttti \mapsto \tightset{
            \widehat{\err}_{\ttti}^{\mathrm{sam}},
            \widehat{\err}_{\ttti}^{\text{num}},
            \dots,
            \tightset{M_{\ttti, \ell},
                \widehat{\rC}_{\ttti, \ell}^{\text{CT}},
                \|{E^{\text{MC}}_{\ttti, \ell}[\bu_{\ell}]}\|_V,
                \dots}_{\ell=0}^{L_{\ttti}}}} \\
        &\texttt{data\_distr} = \tightset{
            E^{\text{MC}}_{\ttti, \ell}[\bu_\ell],
            E^{\text{MC}}_{\ttti, \ell}[\bv_\ell],
            V^{\text{MC}}_{\ttti, \ell}[\bu_\ell],
            V^{\text{MC}}_{\ttti, \ell}[\bv_\ell],
            \dots}_{\ell=0}^{L_{\ttti}} \\[1mm]
        &\texttt{function } \texttt{INIT}(\{M_{0, \ell}^{\text{init}}\}_{\ell=0}^{L_0}) \colon
        \hspace{1.3cm} \text{// Initial estimation round in call-stack}\\
        &\quad
        \begin{cases}
            \texttt{for } \ell = 0, \dots, {L_0} \colon
            &\hspace{1.2cm} \Delta \texttt{data\_distr}_{\ell} \leftarrow \texttt{MX-FEM}(M_{0, \ell}^{\text{init}}, \ell) \\
            \texttt{Welford}(\Delta \texttt{data\_distr})
            &\hspace{1.2cm} \texttt{return } \texttt{BMLMC}(\rT_{\rB} - \sum_{\ell=0}^{L_{0}} \rC_\ell, \, \eta \cdot \widehat{\err}_{0}^{\text{MSE}}) \\
        \end{cases} \\[1mm]
        &\texttt{function } \texttt{BMLMC}(\rT_{\ttti}, \epsilon_{\ttti}) \colon
        \hspace{1.95cm} \text{// Recursive minimization of error} \\
        &\quad
        \begin{cases}
            \texttt{if } \rT_{\ttti} < 0.05 \, \rT_{\rB} \colon
            &\hspace{0.22cm} \texttt{return } \widehat{\err}^{\text{MSE}}_{\ttti - 1} \\
            \texttt{if } \widehat{\err}^{\text{num}}_{\ttti - 1} \geq (1 - \theta) \, \epsilon_{\ttti}^2 \colon
            &\hspace{0.22cm} {L_{\ttti}} \leftarrow {L_{\ttti}} + 1 \\
            \texttt{if } \widehat{\err}^{\mathrm{sam}}_{\ttti - 1} \geq \theta \, \epsilon_{\ttti}^2 \colon
            &\hspace{0.22cm} \{\Delta M_{\ttti, \ell}\}_{\ell=0}^{L_{\ttti}}
            \leftarrow \{\max\{{M}_{\ttti, \ell}^{\text{opt}} - M_{\ttti - 1, \ell}, \, 0\}\}_{\ell=0}^{L_{\ttti}} \\
            \texttt{if } \sum_{\ell=0}^{L_{\ttti}} \Delta M_{\ttti, \ell} \widehat{\rC}^{\text{CT}}_{\ttti - 1, \ell} > \rT_{\ttti} \colon
            &\hspace{0.22cm} \texttt{return } \texttt{BMLMC}(\rT_{\ttti}, 0.5 {(\epsilon_{\ttti} + \epsilon_{\ttti - 1})}) \\
            \texttt{if } \widehat{c}_{\mathrm{Mem}} \cdot \abs{\cK_{0, L_{\ttti}}} > {\mathrm{Mem}}_{\rB} \colon
            &\hspace{0.22cm} \texttt{return } \widehat{\err}^{\text{MSE}}_{\ttti - 1} \\
            \texttt{for } \ell = 0, \dots, {L_{\ttti}} \colon
            &\hspace{0.22cm} \Delta \texttt{data\_distr}_{\ell} \leftarrow \texttt{MX-FEM}(\Delta M_{\ttti, \ell}, \ell) \\
            \texttt{if Welford}(\Delta \texttt{data\_distr}) \colon
            &\hspace{0.22cm} \texttt{return } \texttt{BMLMC}(\rT_{\ttti} - \sum_{\ell=0}^{L_{\ttti}} \rC_\ell, \, \eta \cdot \epsilon_{\ttti}) \\
            \texttt{else} \colon
            &\hspace{0.22cm} \texttt{return } \texttt{BMLMC}(\rT_{\ttti}, \, \eta \cdot \epsilon_{\ttti})
        \end{cases} \\[1mm]
        &\texttt{function Welford} (\Delta \texttt{data\_distr}) \colon
        \hspace{0.35cm} \text{// Update if new data is available} \\
        &\quad
        \begin{cases}
            \texttt{if } (\Delta \texttt{data\_distr} = \emptyset) \colon
            &\hspace{-5.23cm}\texttt{return false} \\[1.5mm]
            \texttt{for } \ell = 0, \dots, {L_{\ttti}} \colon
            &\hspace{-5.23cm}\text{// Accumulation of statistics on each level} \\[1mm]
            \quad
            \begin{cases}
                \hspace{2mm} M_{\ttti, \ell}^{\text{opt}}
                &\leftarrow \hspace{1mm}
                M_{\ttti-1, \ell} + \Delta M_{\ttti, \ell} \\
                \hspace{1.1mm} \delta_{{\ttti, \ell}}[\bu_\ell]
                &\leftarrow \hspace{1mm}
                \Delta E^{\mathrm{MC}}_{\ttti, \ell}[\bu_\ell] - E^{\mathrm{MC}}_{{\ttti-1, \ell}}[\bu_\ell] \\
                \hspace{1.1mm} z_{\ttti, \ell}^2[\bu_\ell]
                &\leftarrow \hspace{1mm} z_{\ttti-1, \ell}^2[\bu_\ell] + \Delta z_{\ttti, \ell}^2[\bu_\ell]
                + \frac{M_{\ttti - 1, \ell} \Delta M_{\ttti, \ell}}{M_{\ttti, \ell}^{\mathrm{opt}}} \norm{\delta_{\ttti,\ell}[\bu_\ell]}_V^2 \\[-1mm]
                E^{\text{MC}}_{\ttti, \ell}[\bu_\ell]
                &\leftarrow \hspace{1mm}
                E^{\text{MC}}_{\ttti - 1, \ell}[\bu_\ell] + \tfrac{\Delta M_{\ttti, \ell}}{M_{\ttti, \ell}^{\text{opt}}} \delta_{{\ttti, \ell}}[\bu_\ell] \\
                \hspace{0.7mm} S_{\ttti, \ell}^2[\bu_\ell]
                &\leftarrow \hspace{1mm}
                S_{{\ttti - 1, \ell}}^2[\bu_\ell] + \Delta S^2_{\ttti, \ell}[\bu_\ell]
                + \tfrac{M_{\ttti-1, \ell} \Delta M_{\ttti, \ell}}{M_{\ttti, \ell}^{\text{opt}}} \delta_{\ttti, \ell}^2[\bu_\ell] \\
                V^{\text{MC}}_{\ttti, \ell}[\bu_\ell]
                &\leftarrow \hspace{1mm}
                (M_{\ttti, \ell}^{\text{opt}} - 1)^{-1} S_{\ttti, \ell}^2[\bu_\ell]  \\
                &\,\,\,\vdots
            \end{cases} \\
            \texttt{evaluate}~\eqref{eq:squared-bias-estimate-mlmc},~\eqref{eq:alpha-fit},~\eqref{eq:sampling-error-estimator},~\dots
            &\hspace{-5.23cm}\texttt{return true}
        \end{cases} \\[1mm]
        &\texttt{function MX-FEM} (\Delta M_{\ttti, \ell}, \ell) \colon
        \hspace{1.35cm} \text{// Specific example in \Cref{alg:mx-fem}}\\
        &\quad
        \begin{cases}
            \hspace{1.47cm} s \hspace{1.47cm} \leftarrow \lceil\log_2 (\min \{\abs{\cP}, \, \Delta M_{\ttti, \ell}, \, M^{\mathrm{Mem}}_{\ttti, \ell} \})\rceil \\[1mm]
            (M_{s, \ell}, \, \abs{\cK_{s, \ell}} \, \Delta M_{s, \ell}) \leftarrow
            (2^s, \, \abs{\cK_{0, 0}} 2^{\ell \cdot d + s}, \, \lceil 2^{-s} \Delta M_{\ttti, \ell} \rceil) \\[1mm]
            \texttt{for } m = 1, \dots, \Delta M_{s, \ell} \colon
            \hspace{1.45cm} \text{// Sequential division of work if necessary} \\
            \quad
            \begin{cases}
                \texttt{run on } (M_{s,\ell}, \cK_{s, \ell}) \texttt{ in parallel} \colon \\[1mm]
                \quad
                \begin{cases}
                    \by^{(m)}_{\ell} &\leftarrow [\hspace{7.13mm}]
                    \begin{cases}
                        \vspace{1mm} \qquad \qquad \quad \, \vdots \vspace{1mm}
                    \end{cases} \\[3mm]
                    \bu_\ell^{(m)} &\leftarrow [\by^{(m)}_{\ell}]
                    \begin{cases}
                        \texttt{Find } \bu_\ell^{(m)} \in V_{\ell}(\cK_{{s, \ell}}^{(m)}) \texttt{ s.t.:} \\
                        \quad \cG_\ell [\by_{\ell}^{(m)}] \,\, (\bu_\ell^{(m)}) = \bb_\ell
                    \end{cases}
                \end{cases}
            \end{cases} \\
            \texttt{postprocess and return } \Delta \texttt{data\_distr}_{\ell}
        \end{cases}
    \end{align*}
\end{algorithm}

%% file: src/multi-x-fem.tex
\section{Multi-X FEM}
\label{sec:multi-x-finite-elelmnt-method}

In this section, we expand the ideas introduced in~\cite{baumgarten2023fully, baumgarten2024fully},
to generate comprehensive spatial domain estimates to multi-\textit{PDE} simulations
while adhering to a memory constraint.
For context, we briefly revisit the formulation of a
multi-\textit{sample} FE space in~\Cref{subsec:multi-sample-finite-element-space}.
Subsequently, we illustrate our approach of a sparse multi-\textit{index} FE update algorithm
through an introductory example
and by giving an analysis of the method's memory consumption in~\Cref{lemma:memory-consumption}.
This in particular gives justification for constraint~\eqref{eq:knapsack-mlmc-mem}, i.e.,
that the maximal memory footprint can be controlled
by the amount of FE cells on the highest level $L$.
The following \Cref{corollary:bmlmc-corollary} provides
abound to~\cite[Prop.~2.5]{baumgarten2024fully}.
We finish in~\Cref{subsec:multi-pde-finite-element-simulations}
by outlining our approach to multi-\textit{PDE} FE simulations,
and by summarizing all concepts under the name
‘Multi-X Finite Element Method’ (or MX-FEM),
where the \textit{X} stands for \textit{sample}, \textit{index} and \textit{PDE}.
The authors of~\cite{beck2020hp, krais2021flexi} use a similar naming concept
for their work, although the term multi-X isn't mentioned there.

\input{src/multi-sample-finite-element-space}

\input{src/multi-level-finite-element-update-algorithm}

\input{src/multi-pde-finite-element-simulations}

%% file: src/multi-sample-finite-element-space.tex
\subsection{Multi-Sample FE Space}
\label{subsec:multi-sample-finite-element-space}

A multi-sample FE space is used to approximate multiple samples
$(\bu_{s, \ell})_{m=1}^{M_{s, \ell}} \in V_{s, \ell}(\cK_{s, \ell})$
of a PDE solution on a parallel cluster.
To explain how many samples $M_{s, \ell}$ and
 FE cells $\cK_{s, \ell}$ are used to construct the space
the estimation round $\ttti$ is kept fixed in this section.
We replace the index $\ttti$ by a new index $s \in \NN_0$,
which represents the number of bisections of
the set of processing units $\cP$ and is
referred to as the communication split.
For simplicity, $\abs{\cP}$ is assumed to be a power of two
to maintain equal size of the subsets after several splits.

As usual for parallel FE methods, the spatial domain $\cD \subset \RR^d$ is decomposed into
finitely many cells $\cK_{s, \ell}$ and distributed over the processing units $\cP$.
This is combined with a parallelization over $M_{s, \ell} \coloneqq 2^s$ samples by
grouping the processing units into disjoint subsets, index by $m=1,\dots,M_{s, \ell}$,
each of size $|\cP_{s, \ell}^{(m)}| = 2^{-s} |\cP|$.
The communication split $s \in \NN_0$ is determined at the beginning of each \texttt{MX-FEM}
call in~\Cref{alg:bmlmc} by the formula (see also~\cite{baumgarten2024fully})
\begin{equation}
    s = \ceil{\log_2 \roundlr{\min \set{\abs{\cP}, \Delta M_{\ttti, \ell}, M_{\ttti, \ell}^{\mathrm{Mem}}}}}.
    \label{eq:comm-split-formula}
\end{equation}
An example of how the optimal communication split $s \in \NN_0$ is determined
is given in~\Cref{subsec:sparse-multi-index-finite-element-update-algorithm}.
For now, $M_{\ttti, \ell}^{\mathrm{Mem}}$ is assumed to be sufficiently large
such that the minimum in~\eqref{eq:comm-split-formula} is never taken by it.
As a result of~\eqref{eq:comm-split-formula}, the choice of $s$ minimizes
the amount of communication in every estimation round while maximizing the
use of the given computational resources.

If $s=0$, the domain $\cD$ is distributed over all processing units
without sample parallelization, i.e, $M_{s, \ell} = 1$.
If $s=\log_2 \abs{\cP}$ then $M_{s, \ell} = \abs{\cP}$, i.e.,
all processing units are distributed over the samples,
each assigned to its own complete domain.

If the minimum in~\eqref{eq:comm-split-formula} does not take the value
$\Delta M_{\ttti, \ell}$\,, common on lower levels,
we additionally split up the
work into $\Delta M_{s, \ell}$ sequentially computed samples such that
\begin{equation}
    \label{eq:sample-amount}
    \Delta M_{s, \ell} = \ceil{\frac{\Delta M_{\ttti, \ell}}{M_{s, \ell}}} = \ceil{2^{-s} \Delta M_{\ttti, \ell}}\,.
\end{equation}

Furthermore, the set of all cells $\cK_{s, \ell}$ is given by the union
\begin{align*}
    \cK_{s, \ell} =
    \bigcup_{m=1, \dots, M_{s, \ell}^{\phantom{(m)}}}
    \cK_{s, \ell}^{(m)} \,\,\, =
    \bigcup_{m=1, \dots, M_{s, \ell}^{\phantom{(m)}}}
    \bigcup_{p \in \cP^{(m)}_{s, \ell}} \,\,
    \bigcup_{K \in \cK_{s, \ell}^{{(m, p)}}} K
    \,,
\end{align*}
with $K$ as open subsets of $\cD$ where $K \cap K' = \emptyset$ for $K \neq K'$.
Here, $\cK_{s, \ell}^{{(m, p)}}$ is the set of cells on the $p$-th processing unit
and of the $m$-th sample, and $\cP^{(m)}_{s, \ell}$ is the set of processing units working on the $m$-th sample.
As a result, the total number of FE cells on the multi-index $(s, \ell)$ can be expressed with respect to
the number of FE cells of a single sample on the lowest level $\ell = 0$ by
\begin{equation}
    \label{eq:cell-amount}
    \abs{\cK_{s, \ell}} = \abs{\cK_{0, 0}} \cdot 2^{\ell \cdot d + s}.
\end{equation}
With $\cK_{s, \ell}$ and $M_{s, \ell}$ defined we can now
construct a multi-sample FE space
\begin{align*}
    V_{s, \ell}(\cK_{s, \ell}) = V_{s, \ell}(\cK_{s, \ell}^{(1)}) \times \dots \times V_{s, \ell}(\cK_{s, \ell}^{(M_{s,\ell})})
    = \prod_{m=1}^{M_{s,\ell}} V_{s, \ell}(\cK_{s, \ell}^{(m)}),
\end{align*}
where $V_{s, \ell}(\cK_{s, \ell}^{(m)}) \coloneqq \tset{\bv_\ell \in V_{\ell} \colon \, \bv_\ell |_K \in V_{\ell,K}, \, \forall K \in \cK_{s, \ell}^{(m)}, \, \cP_{s, \ell}^{(m)} \subset \cP}$
is a FE space for a single sample, defined on the cells in $\cK_{s, \ell}^{(m)}$
and $V_{\ell,K}$ is an arbitrary local FE space.
Hence, with the coefficients
$\bmu = (\bmu_1^{(1)}, \dots, \bmu_{N_{\ell}^{h}}^{(1)}, \dots, \bmu_1^{(M_{s,\ell})}, \dots,
\bmu_{N_{\ell}^{h}}^{(M_{s, \ell})})^\top \in \RR^{M_{s, \ell} \cdot N_{\ell}^{h}}$
a multi-sample FE solution is represented by
\begin{align*}
{(\bu_{s, \ell})}
    _{m=1}^{M_{s, \ell}}
    = \Bigg(\sum_{n=1}^{N_{\ell}^{h}} \bmu_n^{(m)} \bpsi_n^{(m)} \Bigg)_{m=1}^{M_{s, \ell}} \in V_{s, \ell} \coloneqq V_{s, \ell}(\cK_{s, \ell}) \,,
\end{align*}
where $\bpsi_n^{(m)}$ are basis functions of the global FE space
of dimension $N_{\ell}^{h}$.
In this paper, we will employ continuous Lagrange elements,
Raviart-Thomas elements and discontinuous Galerkin (dG) elements all defined
on the same set of cells $\cK_{s, \ell}$.

This formulation is implemented in~\cite{wieners2024mpp341},
where the parallelization over the
samples is realized on the coefficient vector of the FEM.
This enables the highly adaptive parallelization scheme
needed in the BMLMC method.

%% file: src/multi-level-finite-element-update-algorithm.tex
\subsection{Multi-Index FE Update}
\label{subsec:sparse-multi-index-finite-element-update-algorithm}

We first provide an example to give some intuition on the construction of the
multi-sample FE space and how this is incorporated into the multi-index FE update
algorithm for our multi-level Monte Carlo method.

\paragraph{Exemplary Estimation Round}

\begin{figure}[t]
    \centering
    \includegraphics[width=0.8\textwidth]{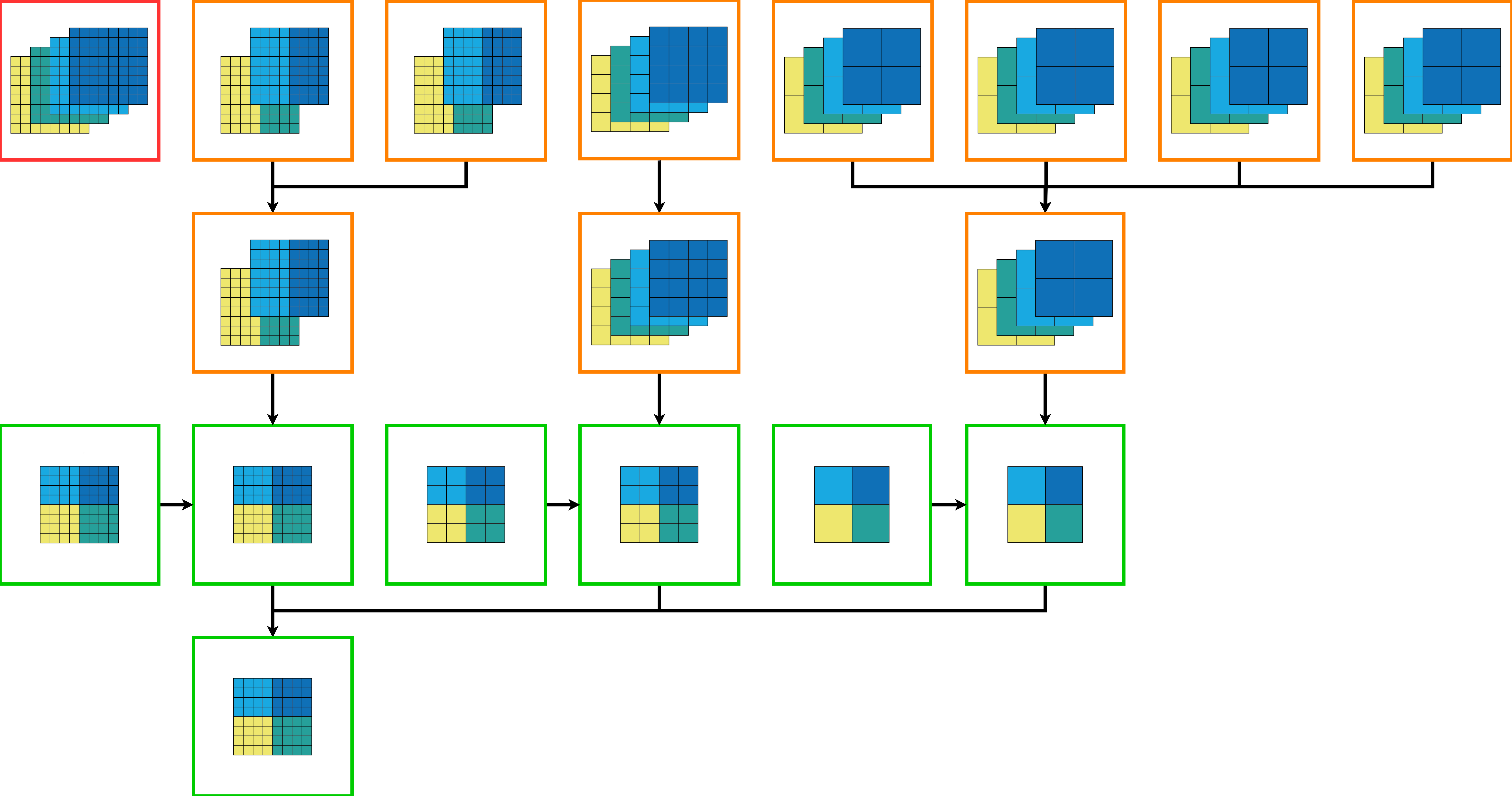}
    \caption{Exemplary estimation round of the full domain update algorithm.
    Meshes are either held permanently in memory (green), temporarily
    created (orange) or not initilized due to a high memory demand (red).
    Each color represents one of four processing units.
    The black arrows represent the flow of data within the algorithm.}
    \label{fig:full-solution-update-example}
\end{figure}

At the beginning of estimation round $\ttti$,
with data from the prior estimation round accessible,
multiple samples across three levels need to be computed.
Specifically, in this example, 
$\Delta M_{\ttti, 2} = 3$, $\Delta M_{\ttti, 1} = 4$ and $\Delta
M_{\ttti, 0} = 16$ with $\abs{\cK_{0, 2}} = 64$, $\abs{\cK_{0, 1}} =
16$ and $\abs{\cK_{0, 0}} = 4$ cells on levels $\ell=2, 1$ and $0$, respectively. 
We assume that these computations do not exceed the computational time
budget.

We refer to~\Cref{fig:full-solution-update-example} for a visual
representation to explain the update algorithm, in this example operating
on four processing units $\abs{\cP} = 4$,
each identifiable by a unique color.
The colored frames indicate whether the associated data structure is
permanently (green), temporarily (orange), or never (red) allocated.
The black arrows represent the flow of data and the order of state updates
within the algorithm.
We explain~\Cref{fig:full-solution-update-example} from  bottom to top,
starting with the result and walking back along the data updates.
The mesh on the bottom represents the solution
estimates on the entire spatial domain discretized on level $\ell = 2$.
This mesh is exclusively distributed across the spatial domain, i.e., $s=0$,
using all $\abs{\cP} = 4$ processing units.

The computation of this solution is described by equation~\eqref{eq:mlmc-estimator-full-solution},
which entails summing across all levels and appropriately
prolonging the solution to the FE space on level $\ell = 2$.
Consequently, the penultimate row in~\Cref{fig:full-solution-update-example}
visually represents the cumulative data from all estimators on all levels.
Crucially, all meshes are exclusively distributed across the spatial
domain to conserve memory and allow for permanent storage of the state.
Additionally, it is essential to note that the outcomes from the preceding
estimation round $\ttti - 1$ need to be incorporated
with the current estimation.
This is depicted by the horizontal black arrows in the penultimate row,
representing the state update achieved by combining computations at each level.

The second row from the top in~\Cref{fig:full-solution-update-example}
signifies an intermediate state where all computations from the current
estimation round merge in a data structure that prioritizes sample parallelization.
This is where~\eqref{eq:comm-split-formula} becomes relevant:
\begin{align*}
    \begin{array}{@{}l@{\,}l@{\,\,\,}l@{\,\,\,}l@{\,\,\,}l@{\,\,\,}l@{\,\,\,}l@{\,}l@{}}
        \ell = 0 \colon & M_{\ttti, 0}^{\mathrm{Mem}} = 32 & \Rightarrow & s = \ceil{\log_2 \roundlr{\min \set{4, 16, 32}}} & = 2
        & \Rightarrow & M_{s,\ell} = 4, & \abs{\cK_{s, \ell}} = 16,    \\
        \ell = 1 \colon & M_{\ttti, 1}^{\mathrm{Mem}} = 8 & \Rightarrow & s = \ceil{\log_2 \roundlr{\min \set{4, 4, 8}}} & = 2
        & \Rightarrow & M_{s,\ell} = 4, & \abs{\cK_{s, \ell}} = 64 ,   \\
        \ell = 2 \colon & M_{\ttti, 2}^{\mathrm{Mem}} = 2 & \Rightarrow & s = \ceil{\log_2 \roundlr{\min \set{4, 3, 2}}} & = 1
        & \Rightarrow & M_{s,\ell} = 2, & \abs{\cK_{s, \ell}} = 128.
    \end{array}
\end{align*}
Here, we have additionally assumed that the maximum number of storable samples is
$M_{\ttti, 0}^{\mathrm{Mem}}=32$, $M_{\ttti, 1}^{\mathrm{Mem}}=8$ and
$M_{\ttti, 2}^{\mathrm{Mem}}=2$, respectively.
Details on a way how to estimate these sample numbers are given below.

While this data structure reduces communication optimally,
it necessitates significantly more memory for storage.
Consequently, it is allocated only temporarily,
transitioning to domain-distributed meshes at the end of each estimation round.
Thus, the transitioning involves allocating responsibilities
for distinct subdomains to individual processing units.
A single processing unit then collects results from all other units
for its designated subdomain, while transmitting information
about the remaining domain to other processors.

Lastly, we explain the top row of~\Cref{fig:full-solution-update-example},
starting on the right with the meshes at level $\ell=0$.
Given the objective of computing $\Delta M_{\ttti, 0} = 16$ samples on this specific level,
and considering that the construction of the meshes prioritizes
distribution across samples with $s=2$ rather than the spatial domain,
we compute $\Delta M_{2,0} = 4$ batches sequentially with $M_{2, 0} = 4$
parallelized samples by following~\eqref{eq:sample-amount}.
Upon completion of each batch, the data is accumulated
before transitioning to domain-distributed meshes.
In the middle of the row,
we depict the computation of $\Delta M_{\ttti, 1}=4$ samples on level $\ell=1$.
Once more, to optimize communication efficiency, the computation is divided,
assigning each processing unit the responsibility of handling one sample.
Finally, on the far left of the row,
we illustrate the computation of $\Delta M_{\ttti, 2}=3$ samples.
In contrast to the other levels, the computation is not straightforwardly
divided across samples due to memory constraints.
In this example, the mesh on level $\ell=2$ cannot be held four times
in memory (cf.~the red frame in~\Cref{fig:full-solution-update-example}).
Consequently, with $s=1$, we only compute $M_{1, 2} = 2$ in parallel while
using two processing units for each sample.
Since $\Delta M_{\ttti, 2}=3$ is requested,
we split the computation into two sequential batches $\Delta M_{s, \ell}=2$ to uphold~\eqref{eq:sample-amount}.

In conclusion, the multi-index pair $(s, \ell)$ determines the memory usage and
communication efficiency of the algorithm.
The way the algorithm is constructed, leads to a sparse index set of $(s, \ell)$ pairs,
since higher levels $\ell$ lead by~\eqref{eq:assumption-beta-v} to fewer samples and therefore,
by~\eqref{eq:comm-split-formula} to smaller communication splits $s$ and vise versa.
We refer to~\Cref{sec:numerical-experiments} for illustrations of this sparse index set.

\paragraph{Memory Analysis} The following paragraph is concerned
with the maximal memory footprint $\rC_{\max}^{\mathrm{Mem}}$
of the sparse multi-index FE update algorithm explained via
the example above.
To this end, let $\rC^{\mathrm{Mem}}_{s, \ell}$ represent the memory usage needed to store
one multi-sample FE solution in megabytes.
This memory usage is proportional to the number of
FE cells used in the multi-sample FE space of~\Cref{subsec:multi-sample-finite-element-space}.
Therefore, we reformulate assumption~\eqref{eq:assumption-gamma-mem} with respect to the number of FE
cells, considering the communication split $s$ and inserting~\eqref{eq:cell-amount},
which yields
\begin{equation}
    \label{eq:memory-usage-estimatior}
    {\rC}^{\mathrm{Mem}}_{s, \ell} \leq {c}_{\mathrm{Mem}} {\abs{\cK_{s, \ell}}} =
    c_{\mathrm{Mem}} \abs{\cK_{0, 0}} \cdot 2^{\ell \cdot d + s}.
\end{equation}

Furthermore, we determine the maximal number of samples
that can be held in memory on every level $\ell$ and
within the current estimation round $\ttti$ by
\begin{equation}
    \label{eq:maximal-samples-by-memory}
    M_{\ttti, \ell}^{\mathrm{Mem}} =
    \floor{
        \frac{{\mathrm{Mem}}_{\rB} - \widehat{\rC}_{\max}^{\mathrm{Mem}}}{{\widehat{\rC}}^{\mathrm{Mem}}_{0, \ell}}
    }.
\end{equation}
This quantity is computed as soon as memory estimates
$\widehat{\rC}^{\mathrm{Mem}}_{0, \ell}$ and measurements of the maximal memory
usage $\widehat{\rC}_{\max}^{\mathrm{Mem}}$ become available.
Thus, using formula~\eqref{eq:comm-split-formula} the total memory
footprint ${\rC}_{\max}^{\mathrm{Mem}}$ can be controlled:

\begin{lemma}[Maximal memory footprint]
  \label{lemma:memory-consumption}
  Suppose enough computing time is available to reach the memory limit
  and \eqref{eq:memory-usage-estimatior} holds with a constant
  $c_{\mathrm{Mem}}$ independent of $\ell$ and $s$. Suppose the
  maximal number of samples $M^{\mathrm{Mem}}_{\ttti, \ell}$ that can be held in memory is
  given through~\eqref{eq:maximal-samples-by-memory},
  estimated using data from previous rounds.
  Then, the maximal memory footprint of the overall algorithm
    is bounded by
    \begin{align*}
        \rC_{\max}^{{\mathrm{Mem}}} < \widehat{c}_{\mathrm{Mem}} \abs{\cK_{0, L}}.
    \end{align*}
\end{lemma}


\begin{proof}
    Assuming that we are in a memory constraint setting, i.e., that the minimum in~\eqref{eq:comm-split-formula}
    is taken by $M_{\ttti, \ell}^{\mathrm{Mem}}$, given by~\eqref{eq:maximal-samples-by-memory}:
    \begin{align*}
        s = \log_2 (M_{\ttti, \ell}^{\mathrm{Mem}})
        \quad \text{with} \quad
        M_{\ttti, \ell}^{\mathrm{Mem}} =
        \floor{
            \tfrac{{\mathrm{Mem}}_{\rB} - \widehat{\rC}_{\max}^{\mathrm{Mem}}}{{\widehat{\rC}}^{\mathrm{Mem}}_{0, \ell}}
        }.
    \end{align*}
    If the algorithm runs long enough such that the measured left over memory
    $\mathrm{Mem}_{\rB} - \widehat{\rC}_{\max}^{\mathrm{Mem}}$ is smaller than the memory estimate
    $\widehat{\rC}^{\mathrm{Mem}}_{0, L}$ on the largest level without communication split,
    then $M_{\ttti, L}^{\mathrm{Mem}} = 0$ by~\eqref{eq:maximal-samples-by-memory} and the algorithm will stop.
    The goal is to bound the maximal memory footprint $\rC_{\max}^{\mathrm{Mem}}$ over the total run-time.
    This is done by
    \begin{align*}
        \rC_{\max}^{\mathrm{Mem}} &=
        \underbrace{\text{\fcolorbox{frameGreen}{white}{${\sum_{\ell=0}^L \rC^{\mathrm{Mem}}_{0, \ell}}$}}}_{\text{Permanently allocated}}
        \, + \quad \, \underbrace{\text{\fcolorbox{frameOrange}{white}{$\max \set{\rC^{\mathrm{Mem}}_{s, \ell}
            \colon \ell = 0,\dots,L, \, s = 0, \dots, \log_2 (M_{\ttti, \ell}^{\mathrm{Mem}})}$}}}_{\text{Dynamically allocated}} \\
        &\leq c_{\mathrm{Mem}} \cdot \left(\sum_{\ell=0}^L \abs{\cK_{0, \ell}}
        \,\, + \,\, \max \set{\abs{\cK_{\log_2 (M_{\ttti, \ell}^{\mathrm{Mem}}), \ell}} \colon \ell = 0,\dots,L}\right) \\
        &= c_{\mathrm{Mem}} \abs{\cK_{0, 0}} \cdot \left(
        \sum_{\ell=0}^L 2^{\ell \cdot d} + \max \set{M_{\ttti, \ell}^{\mathrm{Mem}} \cdot 2^{\ell \cdot d} \colon \ell = 0,\dots,L} \right). \\
    \end{align*}
    We can focus without loss of generality on the largest level $L$ since all levels are controlled
    by~\eqref{eq:comm-split-formula} and~\eqref{eq:maximal-samples-by-memory}.
    Then, we see
    \begin{align*}
        \rC_{\max}^{\mathrm{Mem}}
        \leq c_{\mathrm{Mem}} \abs{\cK_{0, 0}} 2^{L \cdot d} \cdot \left(
        \sum_{\ell=0}^L 2^{(\ell - L) \cdot d} + M_{\ttti, L}^{\mathrm{Mem}} \right)
        < c_{\mathrm{Mem}} \left(2 + M_{\ttti, L}^{\mathrm{Mem}} \right) \cdot \abs{\cK_{0, 0}} 2^{L \cdot d}
    \end{align*}
    Since $M_{\ttti, L}^{\mathrm{Mem}} = 0$ when the algorithm stops,
    and with $\widehat{c}_{\mathrm{Mem}} \coloneqq 2 \cdot c_{\mathrm{Mem}}$, this is gives
    the final result $\rC_{\max}^{\mathrm{Mem}} < \widehat{c}_{\mathrm{Mem}} \cdot \abs{\cK_{0, L}}$.

\end{proof}

As a consequence of~\Cref{lemma:memory-consumption}, we can impose the memory
budget $\rC_{\max}^{\mathrm{Mem}} < {\mathrm{Mem}}_{\rB}$ by the condition
$\widehat{c}_{\mathrm{Mem}} \abs{\cK_{0, L}} < {\mathrm{Mem}}_{\rB}$,
i.e., the constraint~\eqref{eq:knapsack-mlmc-mem} to the knapsack problem.
Here, $\widehat{c}_{\mathrm{Mem}} \sim c_{\mathrm{Mem}}$ is a constant independent of $s$ and $\ell$
which can be estimated during runtime as well.

\begin{remark}
    Since we can bound the total memory consumption by the number of FE cells $\abs{\cK_{0, L}}$
    on the largest level $L$ multiplied with some moderate constant,
    we say that \textit{the memory footprint of the BMLMC method is optimal},
    since the memory required grows asymptotically like the
    largest possible deterministic computation.
\end{remark}

Finally, we build upon~\Cref{lemma:memory-consumption} and~\cite[Proposition 2.5]{baumgarten2024fully}
and show that the error of every
feasible solution to~\Cref{problem:approximated-mlmc-knapsack}
has to obey an upper and a lower bound.

\begin{corollary}[Lower and Upper Bound for Budgeted MLMC]
    \label{corollary:bmlmc-corollary}
    The error of every feasible solution to~\Cref{problem:approximated-mlmc-knapsack}
    is bounded from below due to the memory constraint~\eqref{eq:knapsack-mlmc-mem}
    and from above due to the computing-time constraint~\eqref{eq:knapsack-mlmc-ct}, i.e.,
    \begin{equation}
        \label{eq:upper-and-lower-bound}
        (\widehat{c}_{\mathrm{Mem}}^{\,-1} {\mathrm{Mem}}_{\rB})^{-\alpha}
        \,\, \lesssim \,\,  \mathrm{err}^{\mathrm{MSE}} \,\, \lesssim \,\,
        \big(\lambda_{\mathrm{s}} + \lambda_{\mathrm{p}}
        \abs{\cP}^{-\delta}\big) \cdot \rT_{\rB}^{-\delta}
    \end{equation}
    with the convergence rate $\delta = \min \set{\tfrac12, \frac{\alpha}{2 \alpha + (\gamma - \beta)}}$,
    the parallelizable fraction of the code $\lambda_{\mathrm{p}} \in [0, 1]$,
    the unparallelizable fraction of the code $\lambda_{\mathrm{s}} = 1 - \lambda_{\mathrm{p}}$
    and the memory constant $\widehat{c}_{\mathrm{Mem}}$ translating a single discretization
    cell to its memory footprint.
\end{corollary}

\begin{proof}
    For a proof and a short discussion of the right-hand side of~\eqref{eq:upper-and-lower-bound},
    we refer to~\cite[Proposition 2.5]{baumgarten2024fully}.
    Since reducing $\err^{\mathrm{sam}}$ can be done without exceeding the memory constraint,
    we focus without loss of generality on the case where the bias in~\eqref{eq:mse-mlmc}
    is dominant, i.e., we consider $\err^{\mathrm{num}} \lesssim \mathrm{err}^{\mathrm{MSE}}$.
    Strengthening the assumptions~\eqref{eq:assumption-alpha-u} and following~\Cref{remark:assumptions} e),
    we suppose $\abs{\cK_{0,L}}^{-\alpha} \! \sim \norm{\EE[\bu_{L} - \bu]}_V$.
    With the memory constraint~\eqref{eq:knapsack-mlmc-mem}, we get
    \begin{align*}
        (\widehat{c}_{\mathrm{Mem}}^{\,-1} {\mathrm{Mem}}_{\rB})^{-\alpha} \lesssim \mathrm{err}^{\mathrm{MSE}}
    \end{align*}
\end{proof}

\Cref{corollary:bmlmc-corollary} allows deducing hardware requirements for a given problem.
Problems with low regularity, i.e., with small convergence rates $\alpha$, require more memory.
Utilizing distributed, parallel computing can help to decrease both bounds,
as more CPUs usually come with more nodes and more memory.
And clearly, more CPUs decrease the bound on
the right-hand side of~\eqref{eq:upper-and-lower-bound} up
to a parallelization bias~\cite{baumgarten2024fully}.

%% file: src/multi-pde-finite-element-simulations.tex
\subsection{Multi-PDE FE Simulations}
\label{subsec:multi-pde-finite-element-simulations}

Our particular focus lies on problems consisting of a sequence of coupled PDEs. That is, the solution
of a given PDE is needed, e.g.~as a coefficient, in the formulation of a subsequent one.
We refer to such problems as multi-{PDE} simulations.
Performing multi-PDE simulations using FEM,
imposes new challenges to the method and the implementation.
In particular, different PDEs with different characteristics require
different discretization schemes.
At the same time, memory constraints should be upheld and
data exchange between processes avoided as much as possible.
Our solution is to share the underpinning multi-sample mesh between different PDEs.
As an example, we study a specific multi-PDE problem in this article
where we solve (batch)-sample wise the following problems
(cf.~\cite{baumgarten2023fully, baumgarten2021parallel, kumar2018multigrid, muller2013multilevel}
for other, similar setups):
\begin{enumerate}
    \item[(a)]
    generating Gaussian RF realizations $\by^{(m)}$ with Matérn
    covariance functions based on a stochastic PDE approach,

    \item[(b)]
    determining the Darcy flow field $\bq^{(m)}$ for log-normal permeability coefficients $\exp(\by^{(m)})$
    using the result of the first step as the input,

    \item[(c)]
    and solving a hyperbolic mass transport problem with the given
    Darcy flow field $\bq^{(m)}$ from the previous step.
\end{enumerate}

In the large-scale, distributed setting
each of these steps exhibits its own characteristic challenges.
In the following, we detail our approach and summarize this sequence of
computational tasks in~\Cref{alg:mx-fem},
where our notation makes use of a \textit{currying} formulation,
inspired by the implementation released in~\cite{wieners2024mpp341}.
In the upcoming paragraphs, we elaborate on the individual steps (a)--(c) and
the corresponding functions $\texttt{SPDESampling}$, $\texttt{MixedSolve}$
and $\texttt{TransportSolve}$ in the algorithm.

\smallskip

\begin{algorithm}
    \caption{MX-FEM (Example implementation)}
    \label{alg:mx-fem}
    {
        \vspace{-0.2cm}
        \begin{align*}
            &\texttt{function MX-FEM}(\Delta M_{\ttti, \ell}, M^{\mathrm{Mem}}_{\ttti, \ell}, \ell) \colon
            \hspace{0.15cm} \text{// Gets called in \Cref{alg:bmlmc}} \\[-1mm]
            &\quad
            \begin{cases}
                \hspace{1.47cm} s \hspace{1.47cm} \leftarrow \lceil\log_2 (\min \{\abs{\cP}, \, \Delta M_{\ttti, \ell}, \, M^{\mathrm{Mem}}_{\ttti, \ell} \})\rceil \\[1mm]
                (M_{s, \ell}, \, \abs{\cK_{s, \ell}} \, \Delta M_{s, \ell}) \leftarrow
                (2^s, \, \abs{\cK_{0, 0}} 2^{\ell \cdot d + s}, \, \lceil 2^{-s} \Delta M_{\ttti, \ell} \rceil) \\[1mm]
                \texttt{for } m = 1, \dots, \Delta M_{s, \ell} \colon
                \hspace{1.41cm} \text{// Sequential division of work if necessary} \\
                \quad
                \begin{cases}
                    \texttt{run for } M_{s,\ell} \texttt{ on } \cK_{s, \ell}^{(m)} \texttt{ in parallel} \colon \\
                    \quad \begin{cases}
                              \quad \quad \brho^{(m)}_{s, \ell} &\leftarrow \, \texttt{TransportPDESolve}(\cK^{(m)}_{s, \ell}) \\
                              \quad \quad \brho^{(m)}_{s, \ell-1} &\leftarrow  \, \texttt{TransportPDESolve}(\cK^{(m)}_{s, \ell-1}) \\
                              \Delta\texttt{data\_sample}_{s, \ell}^{(m)} &\leftarrow  \,
                              \{\bu_{s, \ell}^{(m)} \leftarrow \brho_{s, \ell}^{(m)},
                              \bv_{s, \ell}^{(m)} \leftarrow \brho_{s, \ell}^{(m)} - \rP_{\ell-1}^{\ell}\brho_{s, \ell-1}^{(m)}\} \\
                              \Delta\texttt{data\_comm}_{s, \ell}^{(m)} &\leftarrow  \, \texttt{Welford}(\Delta\texttt{data\_sample}_{s, \ell}^{(m)}) \\
                    \end{cases} \\\\[-3mm]
                    \Delta\texttt{data\_distr}_{0, \ell} \leftarrow
                    \texttt{Welford}(\Delta\texttt{data\_comm}_{s, \ell}^{(m)}) \\
                \end{cases} \\\\[-3mm]
                \texttt{return } \Delta\texttt{data\_distr}_{0, \ell} \\
            \end{cases} \\[1mm]
            &\texttt{function TransportSolve}(\cK^{(m)}_{s, \ell}) \colon
            \hspace{0.35cm} \text{// Solve linear transport model} \\[-1mm]
            &\quad
            \begin{cases}
                \bq_{s, \ell}^{(m)}, \bp_{s, \ell}^{(m)} \hspace{2mm}\leftarrow \texttt{MixedPDESolve}(\cK^{(m)}_{s, \ell}) \\[1mm]
                \hspace{4mm} \brho_{s, \ell}^{(m, 0)} \hspace{4.3mm} \leftarrow \tsprod{\bpsi^{\text{dG}}_{n}, \bpsi_{n'}^{\text{dG}}} \\[1mm]
                \texttt{for } n = 1, \dots, N_{\ell}^{\tau} \colon \hspace{1.92cm} \text{// Time stepping scheme} \\
                \quad
                \begin{cases}
                    \,\,\, \brho_{s, \ell}^{(m, n)} \,\,\, \leftarrow [\bq_{s, \ell}^{(m)} \!, \brho_{s, \ell}^{(m, n-1)}]
                    \begin{cases}
                        \texttt{find } \brho_{s, \ell}^{(m,n)}
                        \texttt{ for } \rL, \rF \texttt{ from \eqref{eq:transport-matrices}:}  \\[1mm]
                        \,\,\,\, (\rL + \frac{\tau_{\ell}}{2} {\rF}) \brho_{s, \ell}^{(m, n)}
                        = (\rL - \frac{\tau_{\ell}}{2} \rF) \brho_{s, \ell}^{(m, n-1)}
                    \end{cases} \\\\[-3mm]
                    \brho_{s, \ell}^{(m, n-1)} \leftarrow \brho_{s, \ell}^{(m, n)}
                \end{cases} \\\\[-3mm]
                \brho_{s, \ell}^{(m)} \leftarrow \brho_{s, \ell}^{(m, N_{s, \ell}^{\tau})} \\[1mm]
                \texttt{return } \brho_{s, \ell}^{(m)}
            \end{cases} \\[1mm]
            &\texttt{function MixedSolve}(\cK^{(m)}_{s, \ell}) \colon
            \hspace{1.05cm} \text{// Solve subsurface Darcy model} \\[-1mm]
            &\quad
            \begin{cases}
                \ba_{s, \ell}^{(m)} &\leftarrow \exp(\texttt{SPDESampling}(\cK^{(m)}_{s, \ell}))^{-1} \\[1mm]
                \bq_{s, \ell}^{(m)}, \bp_{s, \ell}^{(m)} &\leftarrow [\ba_{s, \ell}^{(m)}]
                \begin{cases}
                    \texttt{find } \bq_{s, \ell}^{(m)}, \bp_{s, \ell}^{(m)}
                    \texttt{ for } \rA, \rB, \rB^\top \texttt{ from \eqref{eq:mixed-matrices}:} \\[1mm]
                    \quad -\rA[\ba_{s, \ell}^{(m)}] \bq_{s, \ell}^{(m)} + \rB \bp_{s, \ell}^{(m)} = \bp_{s, \ell}^{\mathrm{D}} \\
                    \quad \hspace{0.95cm} \rB^\top \bq_{s, \ell}^{(m)} \hspace{1.35cm} = 0 \\
                \end{cases} \\\\[-3mm]
                \texttt{return} &\bq_{s, \ell}^{(m)}, \bp_{s, \ell}^{(m)}
            \end{cases}\\[1mm]
            &\texttt{function SPDESampling}(\cK^{(m)}_{s, \ell}) \colon
            \hspace{0.62cm} \text{// Generate Gaussian RF} \\[-1mm]
            &\quad
            \begin{cases}
                \rW \leftarrow (\texttt{MassLumping}(\tsprod{\bpsi_n^{\mathrm{Lag}}, \bpsi_{n'}^{\mathrm{Lag}}})^{-1/2} \\[1mm]
                \texttt{for } b = 1, \dots, 2^d \colon
                \hspace{2.1cm} \text{// Iterate over boundary conditions} \\[1mm]
                \quad
                \begin{cases}
                    \bw_{s, \ell}^{(m, b)} &\leftarrow \rW \bxi^{(m, b)}_{s, \ell}
                    \hspace{1.25cm} \text{// Where } \bxi^{(m, b)}_{s,\ell} \text{ is normally distributed} \\[1mm]
                    \by_{s, \ell}^{(m, b)} &\leftarrow [\bw^{(m, b)}_{s, \ell}]
                    \begin{cases}
                        \texttt{find } \by_{s, \ell}^{(m, b)}
                        \texttt{ for }  \rK
                        \texttt{ from \eqref{eq:spde-matrices}:} \\[1mm]
                        \quad \rK \by_{s, \ell}^{(m, b)}
                        = \bw_{s, \ell}^{(m, b)} \\
                    \end{cases} \\\\[-3mm]
                    \by_{s, \ell}^{(m)} &\leftarrow \by_{s, \ell}^{(m)} + \by_{s, \ell}^{(m, b)} \\
                \end{cases} \\\\[-3mm]
                \texttt{return } 2^{-d / 2} \, \by_{s, \ell}^{(m)}
            \end{cases}
        \end{align*}
    }
\end{algorithm}

\paragraph{Generating Gaussian RF Realisations}
A key requirement implied by the multi-PDE setting
is that the memory usage of the sampling algorithm for the Gaussian RF
$\by^{(m)}$ must not exceed that of the subsequent solvers for the problems in (b) and~(c).
This excludes sampling methods based on dense factorisations of the covariance matrix
due to their superlinear scaling with respect to the problem size.
Although Karhunen-Lo{\'e}ve expansions \cite{khoromskij2009application, schwab2006karhunen} and similar
reduced basis techniques~\cite{harbrecht2012low} 
provide favourable runtime performance,
the setup time to find the principle components and the memory footprint disqualifies
them from our considerations as well.
A similar argument holds true for the FFT-based circulant embedding
method~\cite{dietrich1997fast}, 
where the amount of embedding required to make the matrix circular and to achieve positive-definiteness,
corresponds to a non-negligible additional memory overhead~\cite{bachmayr2020unified}. 

The sampling method which naturally fits into our multi-PDE setting is the
stochastic PDE (SPDE) approach as introduced in \cite{lindgren2011explicit}.
The method is based on the observation, that stochastically stationary solutions
$\by(\omega, \bx)$ to the linear SPDE
\begin{equation}
    \label{eq:spde}
    \big(\kappa^2 - \Delta \big)^{\zeta} \by(\omega, \bx) = \cW(\omega, \bx), \qquad \bx \in \RR^d,
\end{equation}
with $\zeta, \kappa > 0$ and with $\cW$ denoting a Gaussian white noise process on $\RR^d$,
possess covariance functions of the Matérn family, i.e.,  for any $\bx_1, \bx_2 \in \RR^d$,
\begin{equation}
    \label{eq:matern-covariance}
    \Cov(\bx_1, \bx_2) = \frac{\sigma^2}{2^{\nu - 1} \Gamma(\nu)} (\kappa \br)^{\nu} \cK_{\nu}(\kappa \br), \quad
    \br = \norm{\bx_1 - \bx_2}_2, \quad \kappa = \frac{\sqrt{2\nu}}{\lambda}.
\end{equation}
Here, $\nu = 2 \zeta - d/2$ controls the smoothness of the field,
$\lambda > 0$ is the correlation length and $\sigma$ is a global
scaling factor for the variance.
Furthermore, $\cK_{\nu}$ is the modified Bessel function of the second kind and $\Gamma$
is the Gamma function.

In~\cite{lindgren2011explicit} a FE approximation to~\eqref{eq:spde}
is proposed to efficiently compute approximate realizations of the Gaussian RF $\by(\omega, \bx)$.
The main benefits of this in our setting
are that we can rely on well-established theory
and leverage the existing multi-sample parallelization of the FE method for
an efficient and scalable generation of Gaussian RF realisations.
Details regarding the discretization of~\eqref{eq:spde}, in particular
regarding the fractional exponent $\zeta$ and the Gaussian white noise
$\cW$ can be found in~\cite{croci2018efficient}. 

One challenge in transitioning from the continuous SPDE on $\RR^d$ to its FE
discretization on a bounded domain $\cD \subset \RR^d$ is that artificial
boundary conditions cause aliasing effects near the domain boundary.
These can be mitigated by computing on a larger domain,
reducing the error exponentially with overlap size~\cite{khristenko2019analysis}.
However, again due to memory constraints, enlarging the mesh is not feasible.
Instead, we compute averages of $2^d$ SPDE solutions with  alternating
Dirichlet and Neumann conditions, as in~\cite{kutri2024dirichlet}.
This method allows for isotropic Gaussian RF sampling on the original
mesh on $\cD$, at the expense of having to compute $2^d$ realizations.

The particular procedure $\texttt{SPDESampling}$ is outlined in~\Cref{alg:mx-fem}.
First, we assemble the FE mass matrix $\rM$ using standard Langrange
FEM. We mass-lump it and compute its inverse square root
\begin{align*}
    \rW \coloneqq (\texttt{MassLumping}(\rM))^{-1/2}
    \quad \text{with} \quad
    \rM \coloneqq \tsprod{\bpsi_n^{\mathrm{Lag}}, \bpsi_{n'}^{\mathrm{Lag}}}\,.
\end{align*}
An approximation~$\bw_{s, \ell}^{(m)}$ to the Gaussian white noise $\cW$ can then be
computed by applying $W$ to a vector~$\bxi_{s, \ell}^{(m)}$ of standard normally
distributed random variables~\cite{lindgren2011explicit, bolin2020numerical}.
Finally, we can use a multigrid-preconditioned conjugate gradient method to
solve for realizations $\by_{s, \ell}^{(m)}$ of the
SPDE~\eqref{eq:spde} with stiffness matrix
\begin{equation}
    \label{eq:spde-matrices}
    \rK \coloneqq \tsprod{\kappa^2 \bpsi_{n}^{\mathrm{Lag}},
      \bpsi_{n'}^{\mathrm{Lag}}} +
    \tsprod{\nabla \bpsi_{n}^{\mathrm{Lag}}, \nabla \bpsi_{n'}^{\mathrm{Lag}}}.
\end{equation}

\paragraph{Solving the Darcy system in a mixed formulation}
The third function $\texttt{MixedSolve}$ in~\Cref{alg:mx-fem} outlines how to approximate the pressure
$\bp \colon \Omega \times \cD \rightarrow \RR$ and the flux field
$\bq \colon \Omega \times \cD \rightarrow \RR^d$ of the Darcy system
\begin{equation}
    \label{eq:darcy}
    \pdeProblem{
        - \exp(\by(\omega, \bx)) \, \nabla \bp(\omega, \bx) &=&
        \bq(\omega, \bx) & \text{on } \cD \\
        \div(\bq(\omega, \bx)) &=& 0 &\text{on } \cD
      }
\end{equation}
with the log-normal permeability field $\exp(\by(\omega, \bx))$ from the previous step,
as well as mixed Dirichlet and Neumann
boundary conditions $\bp(\omega, \bx) = \bp^{\text{D}}$ and  $\bq(\omega, \bx) \cdot \bn = \bq^{\text{N}}$
on $\partial \cD_{\mathrm{D}}$ and $\partial \cD_{\mathrm{N}}$,
respectively, with $\partial \cD = \partial \cD_{\mathrm{D}} \cup \partial \cD_{\mathrm{N}}$.

A popular choice for the discretization of \eqref{eq:darcy}
is the mixed FE  method, using the Raviart-Thomas space
$V_{s, \ell}^{\mathrm{RT}}(\partial \cD_{\mathrm{N}})$ and the space
of piecewise constant functions $V_{s, \ell}^{\text{FV}}$
for the flux $\bq_{s,\ell}^{(m)}$ and the pressure
$\bp_{s,\ell}^{(m)}$, respectively. For details (with a similar
notation) we refer to~\cite{baumgarten2023fully, baumgarten2021parallel}.
For further reading and error estimations,
we refer to standard literature on mixed FE
methods e.g.~\cite{boffi2013mixed}. 
Based on these references and with $\ba_{s, \ell}^{(m)} \coloneqq \exp(\by^{(m)}_{s, \ell})^{-1}$
we define the matrices
\begin{equation}
    \label{eq:mixed-matrices}
    \rA[\ba_{s, \ell}^{(m)}] \coloneqq \tsprod{\ba_{s, \ell}^{(m)} \bpsi_{n}^{\mathrm{RT}}, \bpsi_{n'}^{\mathrm{RT}}},
    \quad
    \rB \coloneqq \tsprod{\bpsi_{K}^{\text{FV}}, \div \bpsi_{n}^{\text{RT}}},
    \quad
    \rB^{\top} \coloneqq \tsprod{\div \bpsi_{n}^{\text{RT}}, \bpsi_{K}^{\text{FV}}}
\end{equation}
and the right-hand side vector $\bp^{\rD}_{s. \ell} \coloneqq \tsprod{\bp^{\rD}, \bpsi_{n}^{\text{RT}} \cdot \bn}$.
The solution to~\eqref{eq:darcy} can then be approximated
as outlined in the $\texttt{MixedSolve}$ function
of~\Cref{alg:mx-fem}.

\smallskip

\paragraph{Solving the hyperbolic transport problem}
Finally, we compute the mass distribution $\brho \colon \Omega \times \cD \times [0, T] \rightarrow \RR$
of the hyperbolic transport problem
\begin{equation}
    \label{eq:hyperbolic-transport}
    \pdeProblem{
        \partial_t \brho(\omega, \bx, t) + \div(\bq(\omega, \bx) \brho(\omega, \bx, t)) &=& 0 & \text{on } \cD \times (0, T] \\
        \brho(\omega, \bx, t) &=& \brho^{0} & \text{on } \cD
    }
\end{equation}
where $\bq(\omega, \bx)$ is the flux field from the Darcy system~\eqref{eq:darcy},
and $\brho^{0}$ is the initial condition.
To simplify the notation, no inflow boundary conditions or forcing
terms are considered, but they can easily be included.
We refer to~\cite{hochbruck2015efficient, di2011mathematical}
for analysis on such systems, and to~\cite{baumgarten2023fully, baumgarten2021parallel}
for details on the discontinuous Galerkin (dG) space that we employ.
Here, we only state the definitions of the
mass matrix $\rL$ and the flux matrix $\rF[\bq_{s, \ell}^{(m)}]$
which depends upon the flux field $\bq_{s, \ell}^{(m)}$:
\begin{equation}
    \label{eq:transport-matrices}
    \rL \coloneqq \tsprod{\bpsi_{n}^{\text{dG}}, \bpsi_{n'}^{\text{dG}}}, \,\,\,\,
    \rF[\bq_{s, \ell}^{(m)}] \coloneqq -\tsprod{\bq_{s, \ell}^{(m)} \bpsi_{n}^{\text{dG}}, \nabla \bpsi_{n'}^{\text{dG}}}
    + \sum_{F \in \cF} \tsprod{\bn_F \cdot (\bq_{s, \ell}^{(m)} \bpsi_{n}^{\text{up}}), \bpsi_{n'}^{\text{dG}}}.\!\!
\end{equation}
In the assembly of the flux matrix $\rF$ the upwind flux
$\bpsi_{n}^{\text{up}}$ is chosen on the cell faces $F \in \cF$ to
ensure stability of the scheme (cf.~details~e.g.~in~\cite{hochbruck2015efficient}).

After discretizing in space, it still remains to
discretize~\eqref{eq:hyperbolic-transport} in time.
Here, we employ the unconditionally stable,
second-order accurate implicit midpoint rule
with step size $\tau_\ell = T/N_\ell^{\tau}$, which provides solutions
$\brho_{s, \ell}^{(m, n)}$ at time points
$t_n = n \tau_\ell$ with $n = 1, \dots, N_\ell^{\tau}$.
However, other time stepping schemes can be used as well.
The arising linear systems for this implicit time stepping scheme are solved
using a parallel GMRES method with a point-block Jacobi preconditioner.

\paragraph{Multi-X FEM}
In conclusion, we have now motivated all three components of the Multi-X FEM:
the parallel computation of multiple domain distributed samples,
the memory-efficient accumulation on a multi-index set,
and the treatment of a coupled PDEs sequence.
The first function \texttt{MX-FEM} in~\Cref{alg:mx-fem} is the entry point
into this functionality.
It is the implementation
of~\eqref{eq:comm-split-formula},~\eqref{eq:sample-amount},~\eqref{eq:cell-amount}
and the subsequent parallel Monte Carlo method applied to the multi-PDE problem.
Note however, that within the \texttt{MX-FEM} further variants of the
\texttt{Welford} functionality have to be invoked
to appropriately compute the statistics.
For details, we refer again
to~\cite{baumgarten2023fully, pebay2016numerically}.

%% file: src/numerical-experiments.tex
\section{Numerical Experiments}
\label{sec:numerical-experiments}
In this section, we illustrate two numerical examples to support the findings
and to demonstrate the functionality of the method introduced in the previous sections.
All experiments are based on the software M++~ in version~\cite{wieners2024mpp341}
using four nodes of the HoreKa supercomputer where on each node
we reserve $64$ CPUs and $128$Gbyte of memory resulting in $\abs{\cP} = 256$ and $\mathrm{Mem}_{\rB}=512$Gbyte.
To keep the presentation of the results lean, we will not conduct
detailed numerical experiments on the different methods, scaling properties or hyperparameters.
We refer to~\cite{baumgarten2023fully, baumgarten2024fully} for a detailed discussion on these topics
and commit to the $\rL^2$-norm throughout the experiments.

\input{src/problem-configuration-and-illustration}

\input{src/computational-results}

%% file: src/problem-configuration-and-illustration.tex
\subsection{Problem Configuration and Illustrations}
\label{subsec:problem-configuration-and-illustration}
We consider the multi-PDE problem outlined
in~\Cref{subsec:multi-pde-finite-element-simulations}
on the domain $\cD = (0, 1)^2$.
For the transport problem~\eqref{eq:hyperbolic-transport} we choose the time interval  $[0, T] = [0, 0.5]$ and the initial condition
\begin{align*}
    \brho^0(\bx) = \exp \roundlr{-(\bx - \bx^0)^\top \Sigma \, (\bx - \bx^0)}
    \,\,\, \text{with} \,\,\, \bx^0 = (0.5, 0.8)^\top
    \,\, \text{and} \,\,\Sigma =
    \begin{pmatrix}
        0.3^2 & 0     \\
        0     & 0.1^2
    \end{pmatrix}.
\end{align*}
The vector field $\bq(\omega, \bx)$ is given through the approximation of~\eqref{eq:darcy}
with homogeneous Dirichlet boundary conditions on $\partial \cD_{\rD} = \set{\bx \in \partial \cD \colon x_2 = 0}$,
the deterministic Neumann condition $\bq \cdot \bn \equiv -1$ on the boundary
$\partial \cD_{\rN} = \set{\bx \in \partial \cD \colon x_2 = 1}$
and a homogeneous Neumann boundary condition everywhere else on $\partial \cD$.
As a result, the vector field $\bq(\omega, \bx)$ points from the top to the bottom of the domain
with fluctuations in the transport velocity and random disturbances in the $x_1$-direction.
These fluctuations are inherited from the log-normal permeability field $\exp (\by(\omega, \bx))$ where the
GRF possesses a Mat\'ern covariance function~\eqref{eq:matern-covariance}
with $\lambda = 0.3$ and a Bessel covariance $\nu = 1$ such that
$\zeta = 1$ in two dimensions.

\begin{figure}[p]
    \centering

    \includegraphics[width=0.215\textwidth]{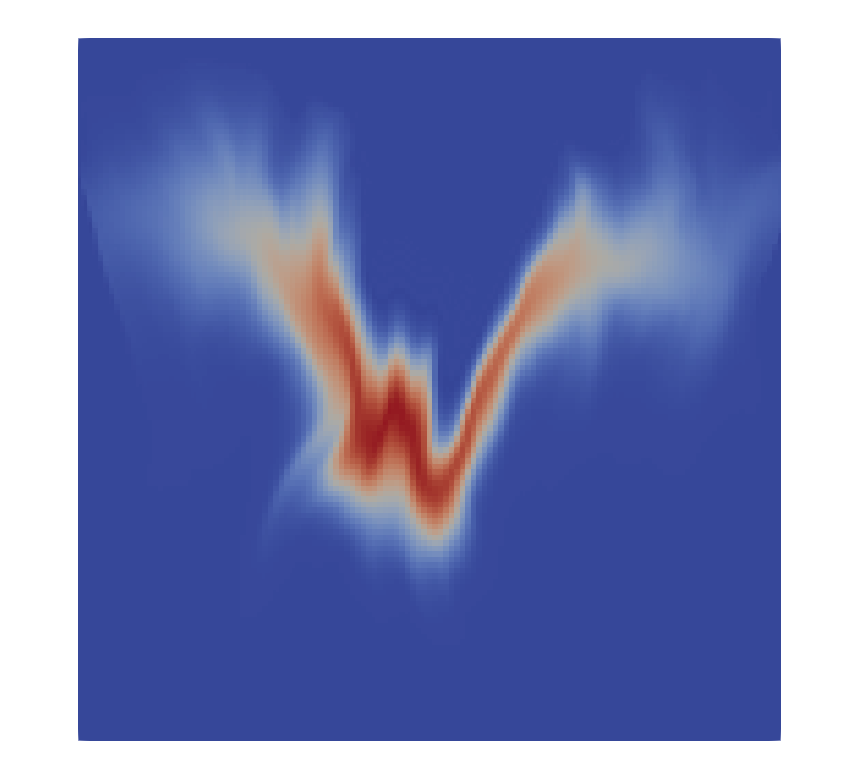}
    \hspace{-5mm}
    \includegraphics[width=0.215\textwidth]{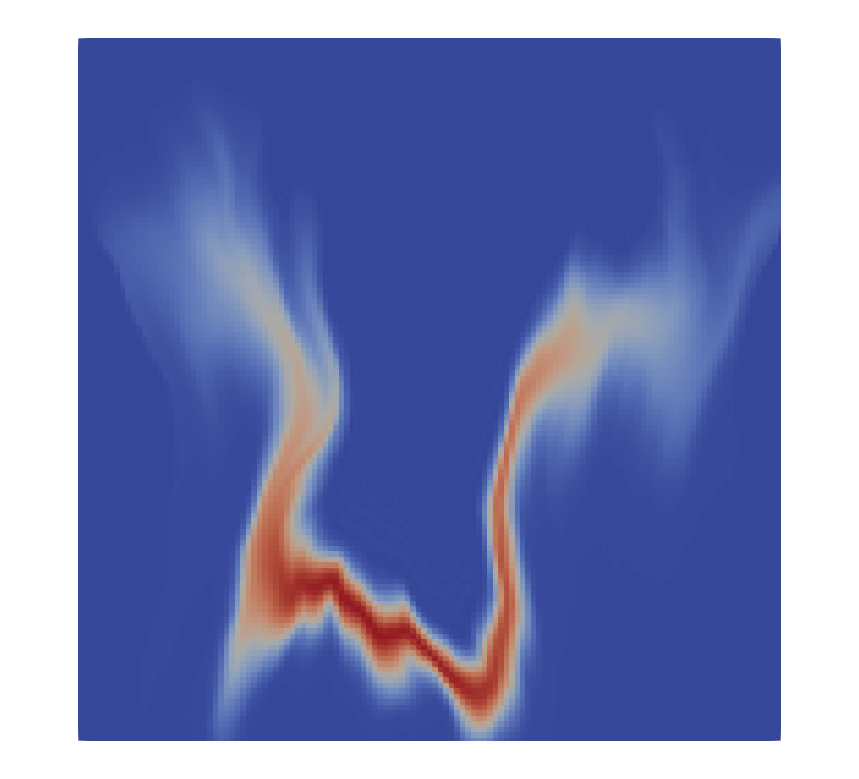}
    \hspace{-5mm}
    \includegraphics[width=0.215\textwidth]{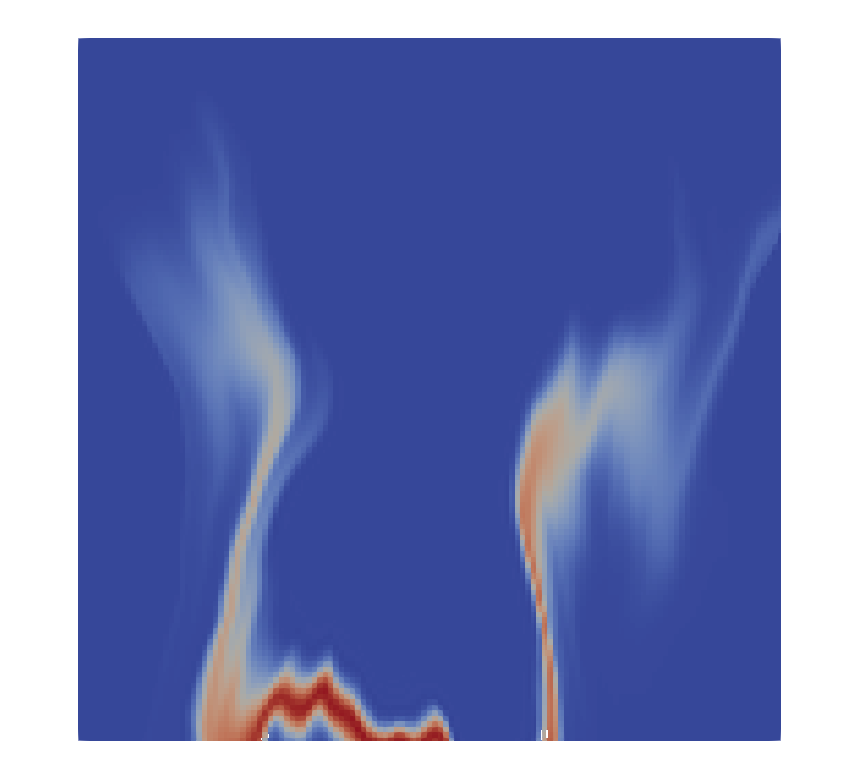}
    \hspace{-5mm}
    \includegraphics[width=0.215\textwidth]{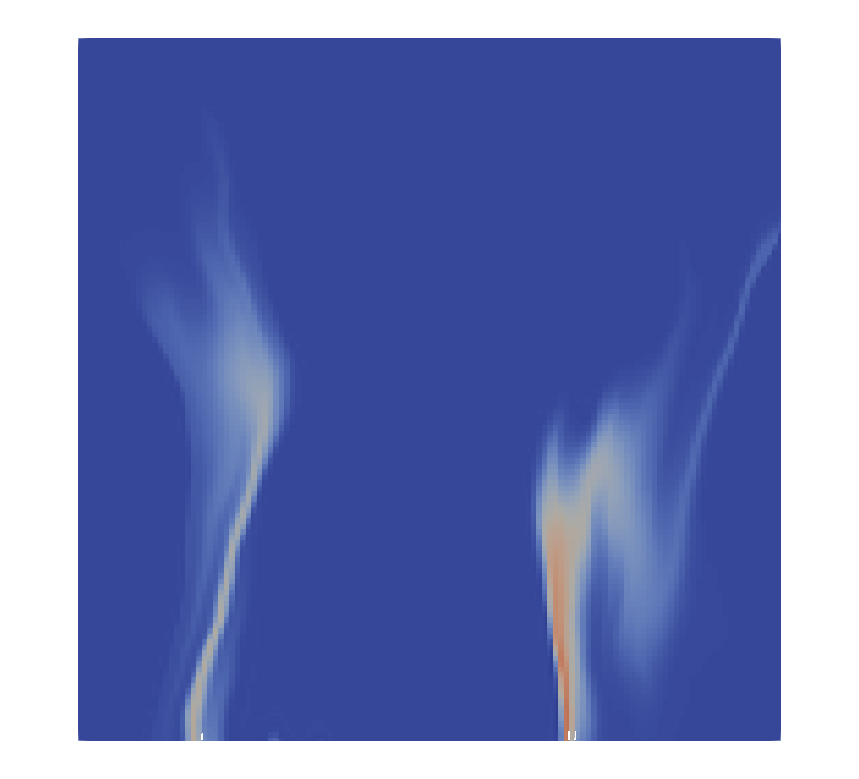}
    \hspace{-5mm}
    \includegraphics[width=0.215\textwidth]{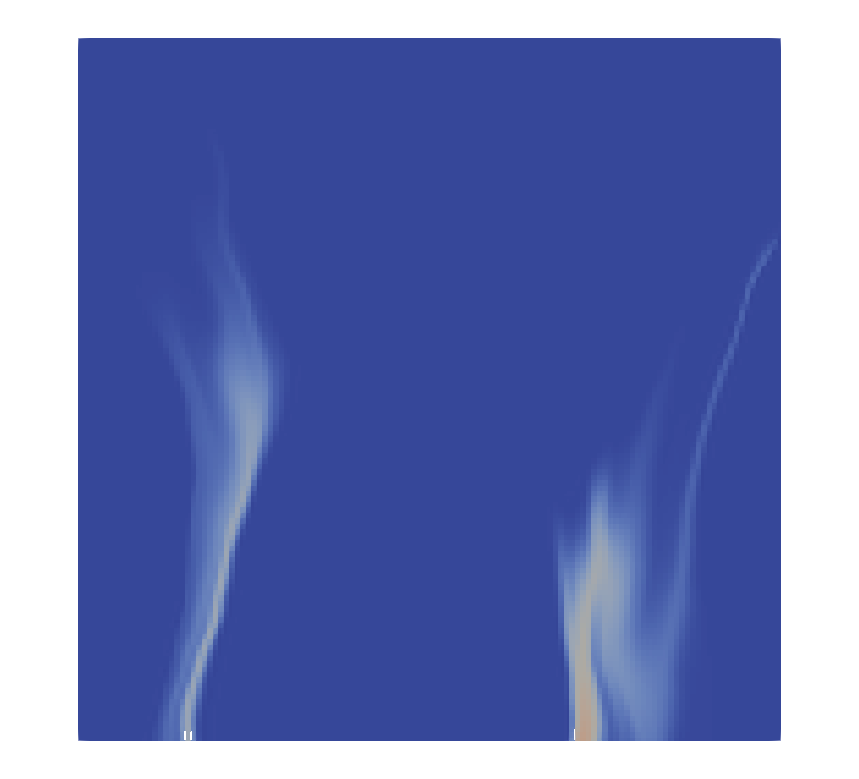}
    \vspace{-10mm}
    \caption{First sample solution $m=1$ on level $\ell=8$.}
    \label{fig:sample1}

    \smallskip

    \includegraphics[width=0.215\textwidth]{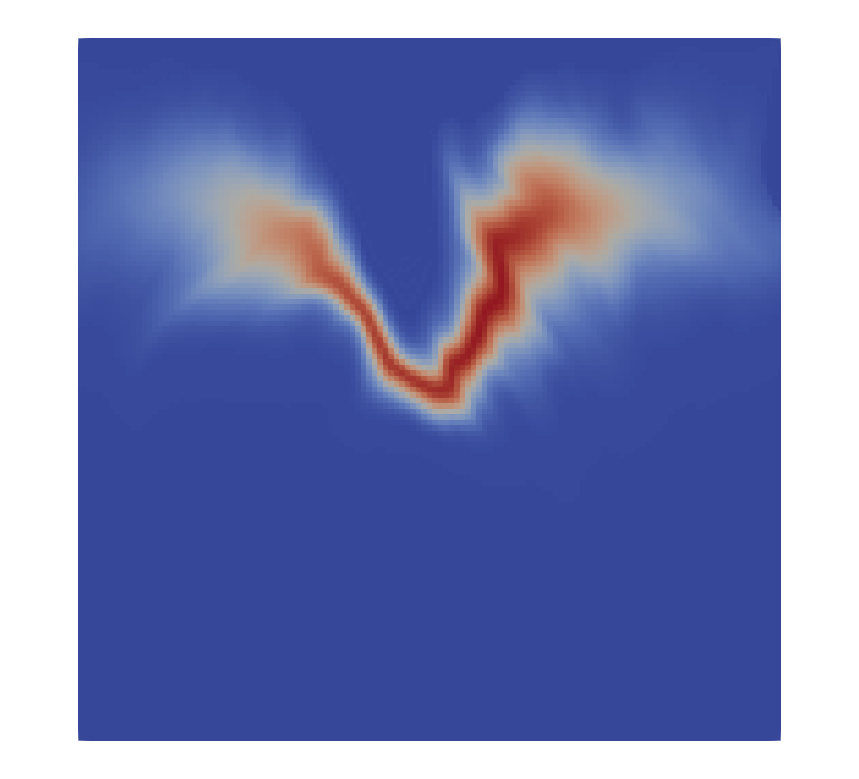}
    \hspace{-5mm}
    \includegraphics[width=0.215\textwidth]{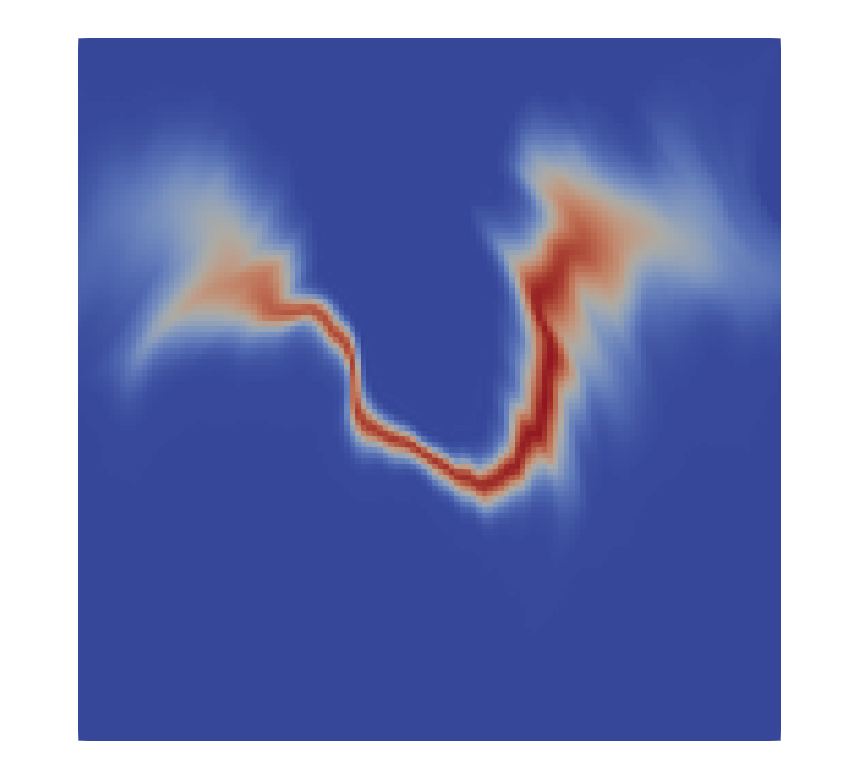}
    \hspace{-5mm}
    \includegraphics[width=0.215\textwidth]{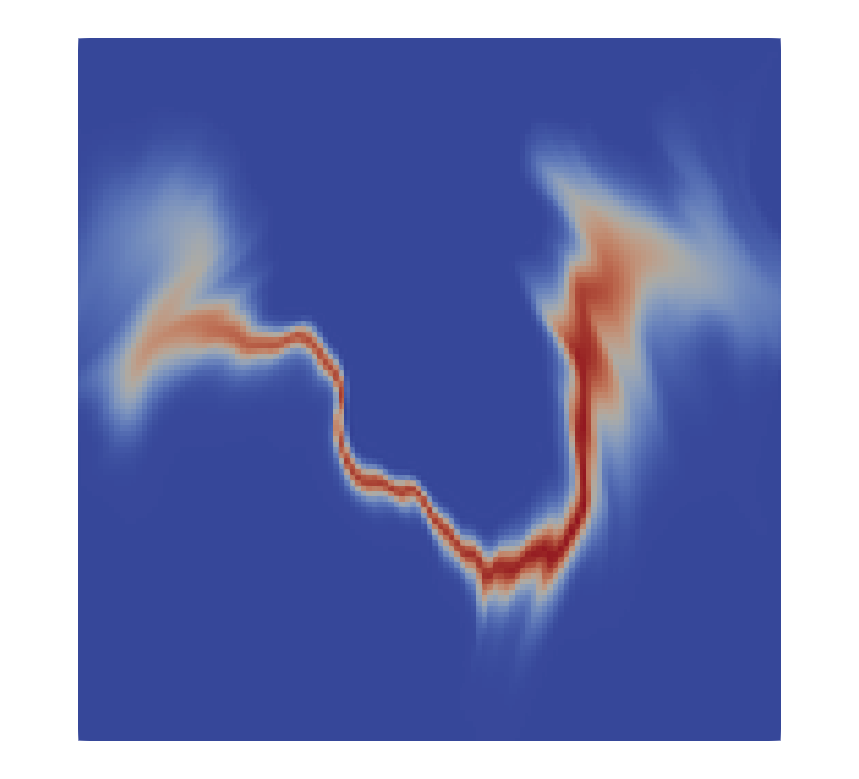}
    \hspace{-5mm}
    \includegraphics[width=0.215\textwidth]{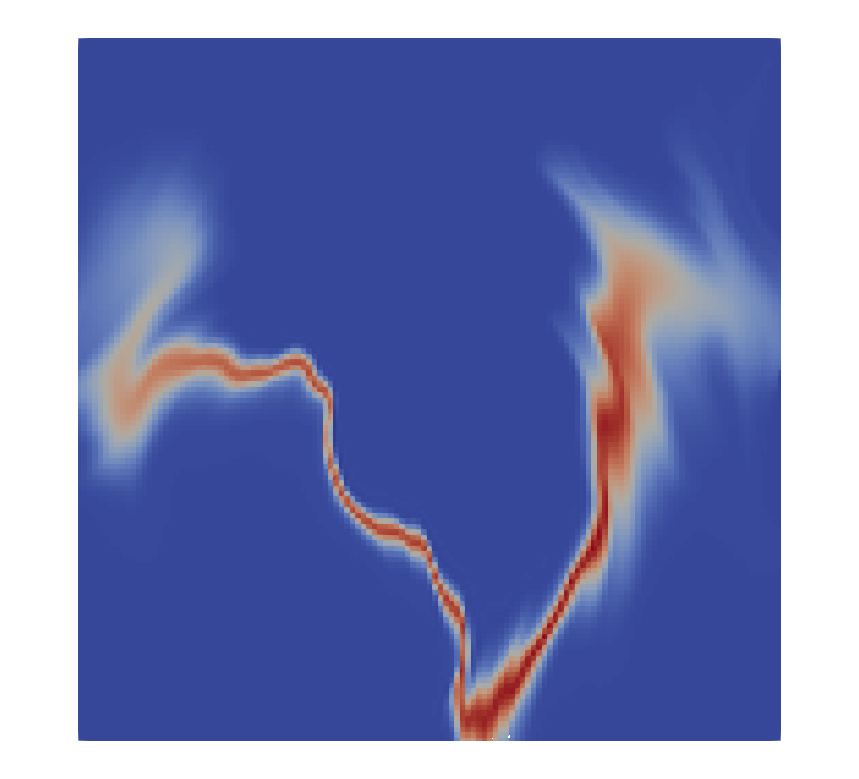}
    \hspace{-5mm}
    \includegraphics[width=0.215\textwidth]{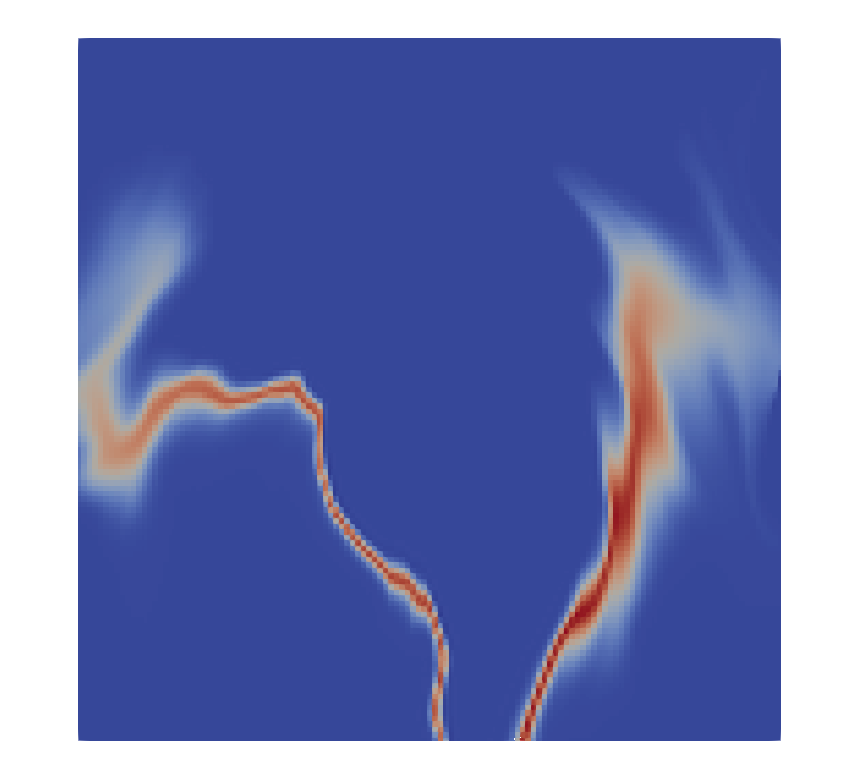}
    \vspace{-10mm}
    \caption{Second sample solution $m=2$ on level $\ell=8$.}
    \label{fig:sample2}

    \smallskip

    \includegraphics[width=0.215\textwidth]{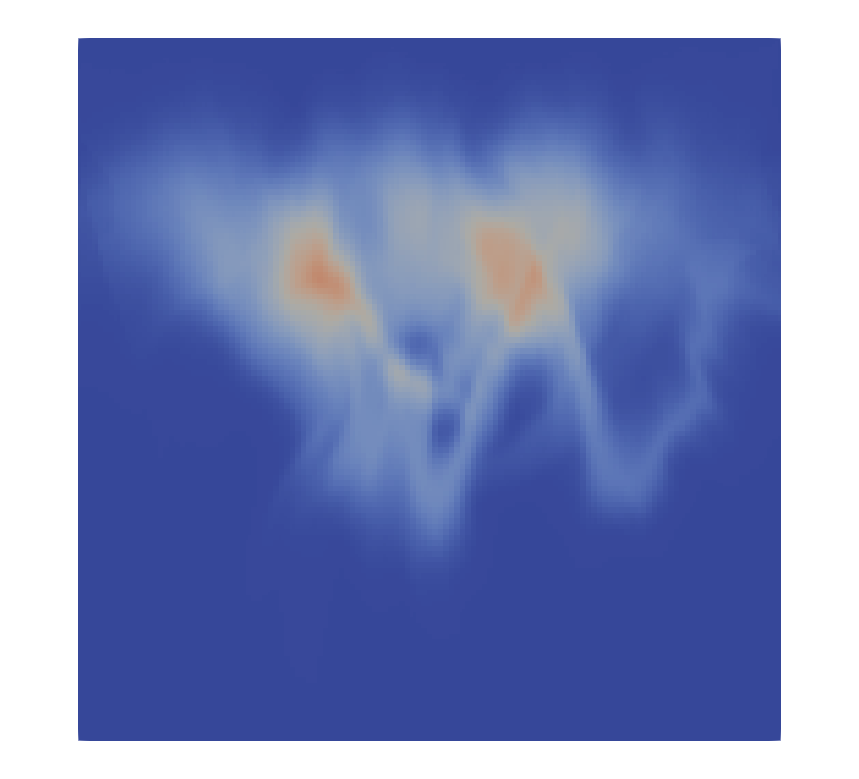}
    \hspace{-5mm}
    \includegraphics[width=0.215\textwidth]{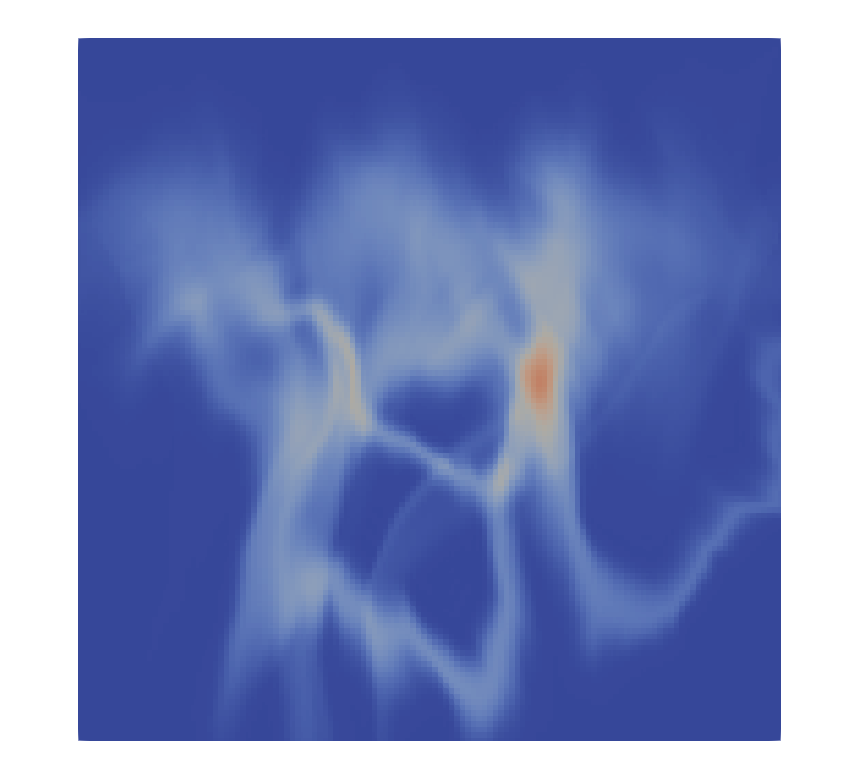}
    \hspace{-5mm}
    \includegraphics[width=0.215\textwidth]{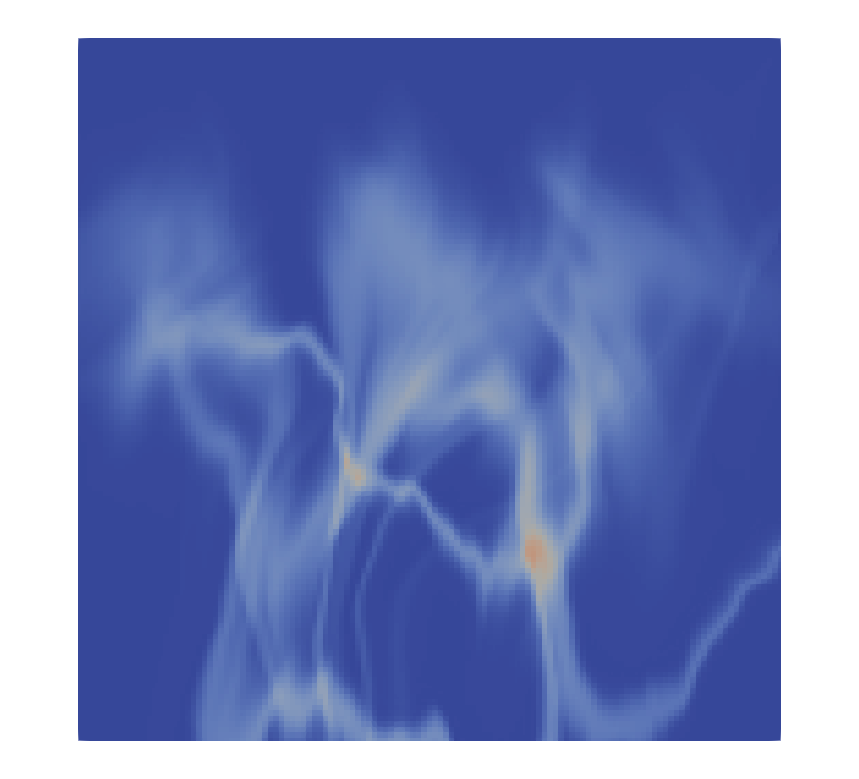}
    \hspace{-5mm}
    \includegraphics[width=0.215\textwidth]{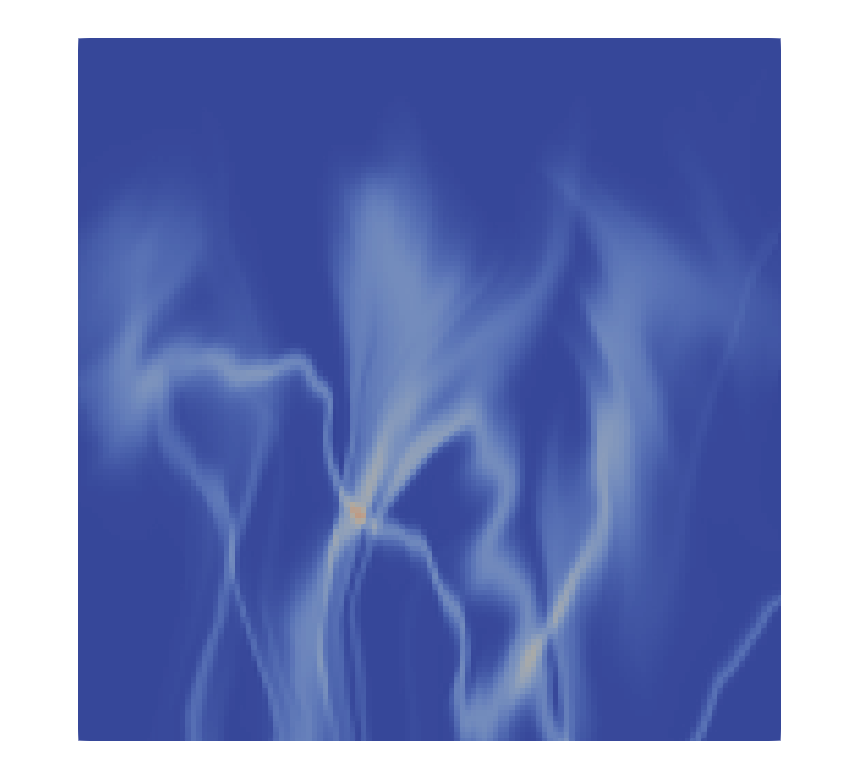}
    \hspace{-5mm}
    \includegraphics[width=0.215\textwidth]{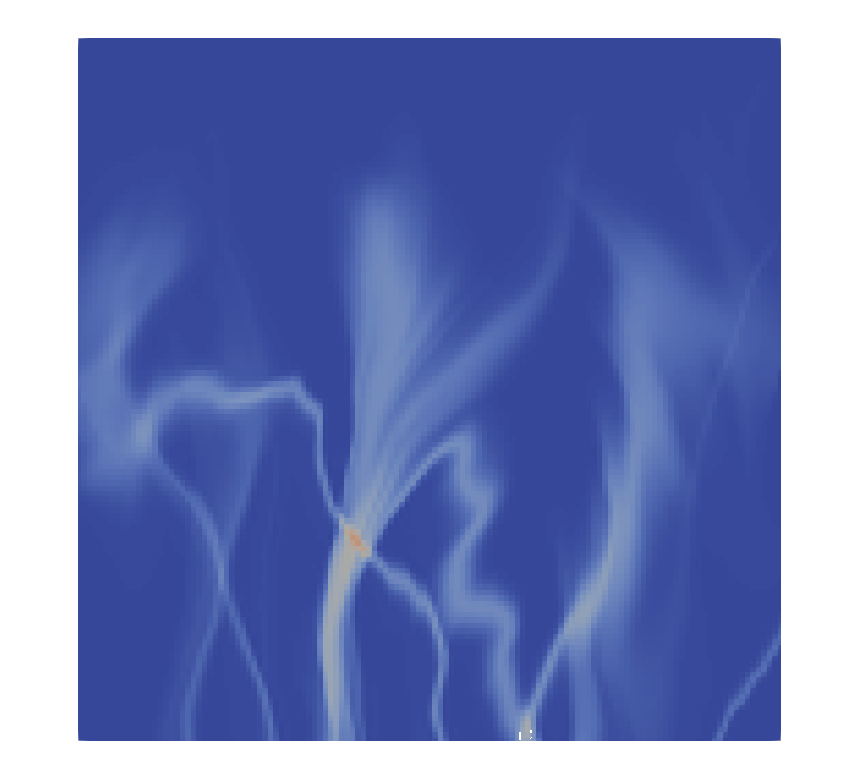}
    \vspace{-10mm}
    \caption{Estimated mean field on level $\ell=8$ of four parallel samples.}
    \label{fig:mean-para}

    \smallskip

    \includegraphics[width=0.215\textwidth]{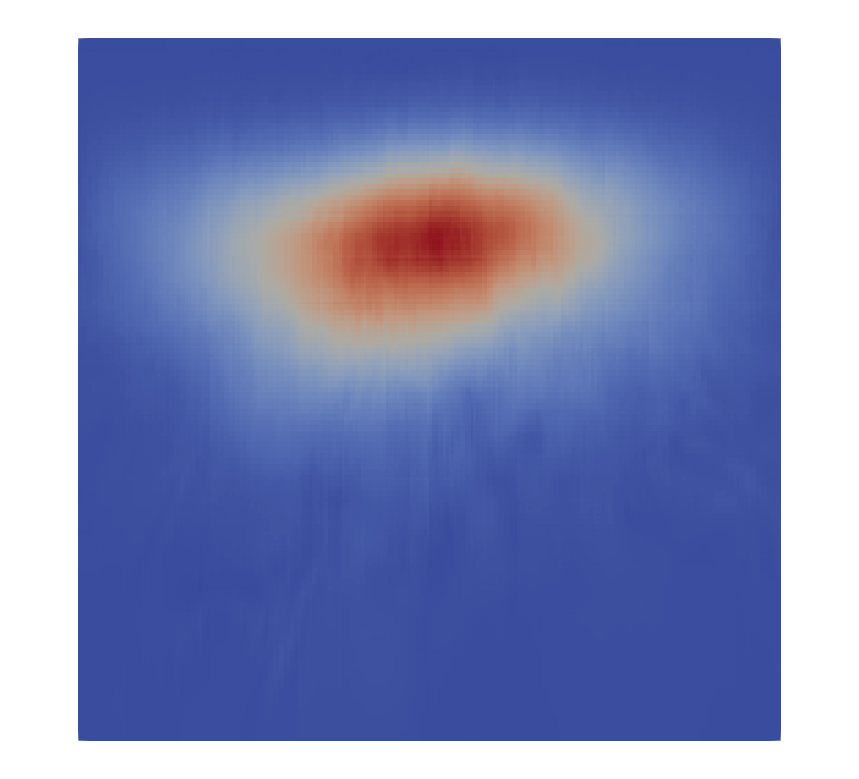}
    \hspace{-5mm}
    \includegraphics[width=0.215\textwidth]{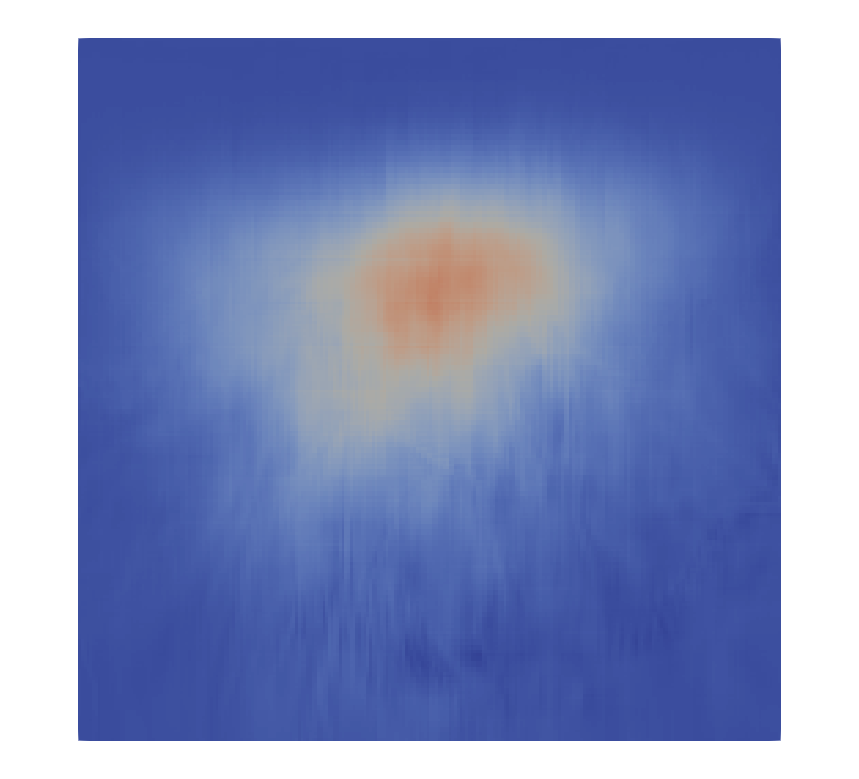}
    \hspace{-5mm}
    \includegraphics[width=0.215\textwidth]{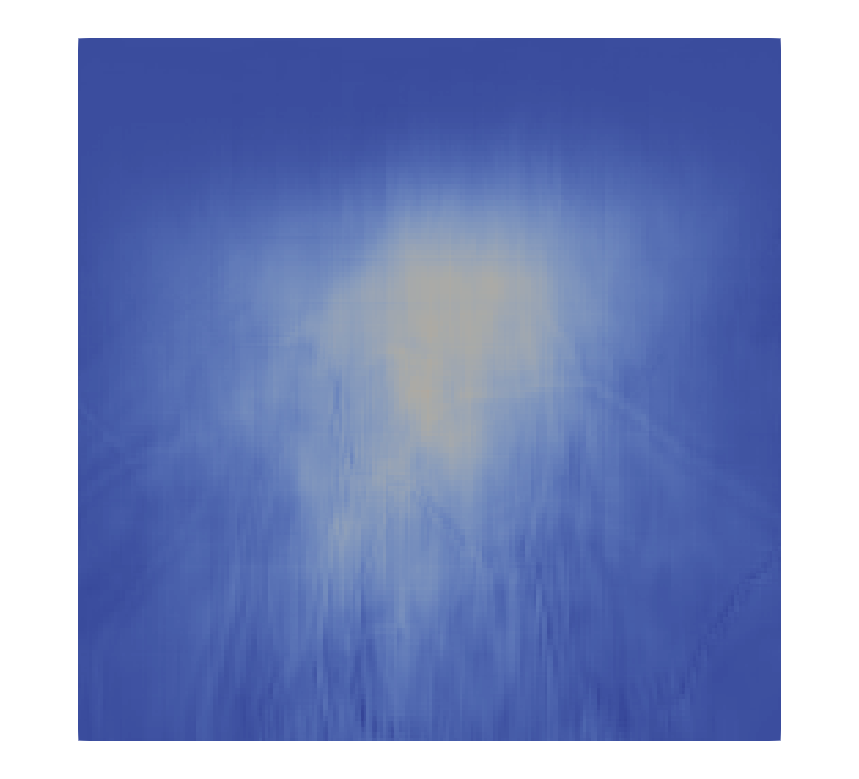}
    \hspace{-5mm}
    \includegraphics[width=0.215\textwidth]{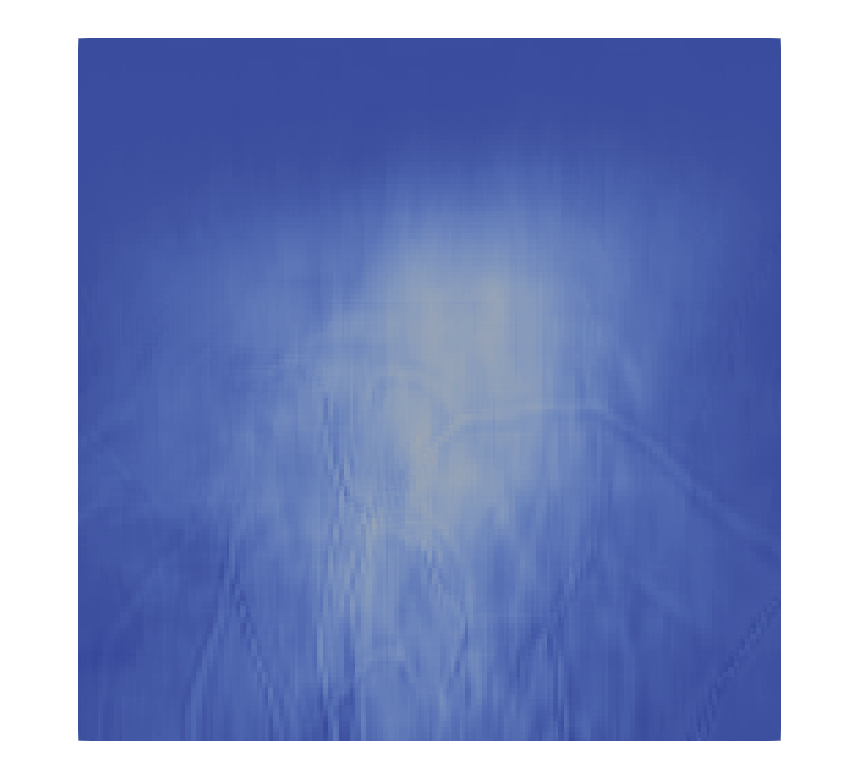}
    \hspace{-5mm}
    \includegraphics[width=0.215\textwidth]{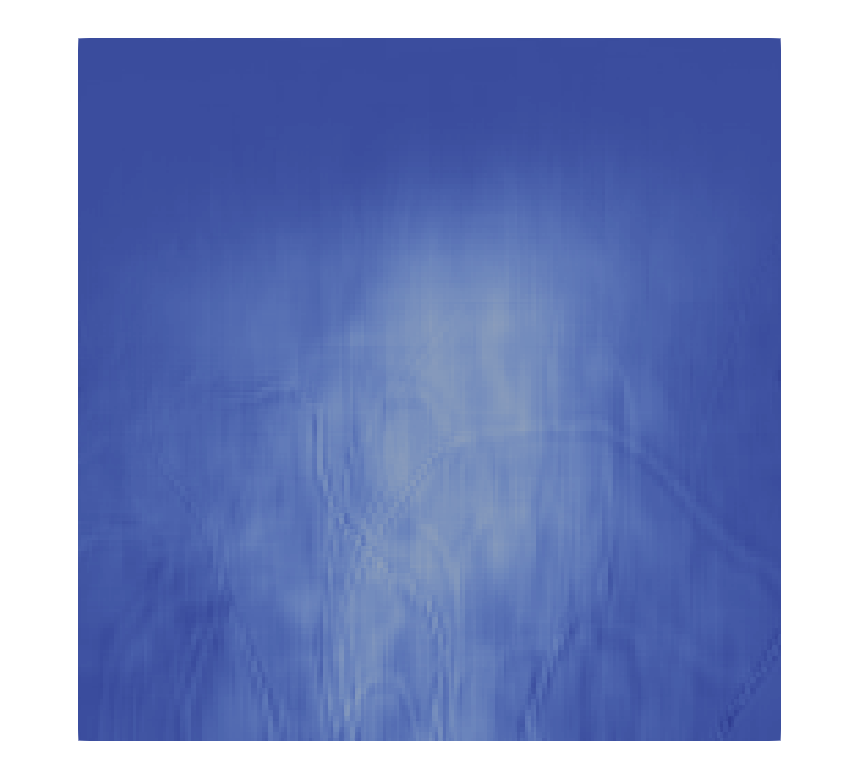}
    \vspace{-10mm}
    \caption{Multi-level mean field approximation after the initial estimation round.}
    \label{fig:mean-init}

    \smallskip

    \includegraphics[width=0.215\textwidth]{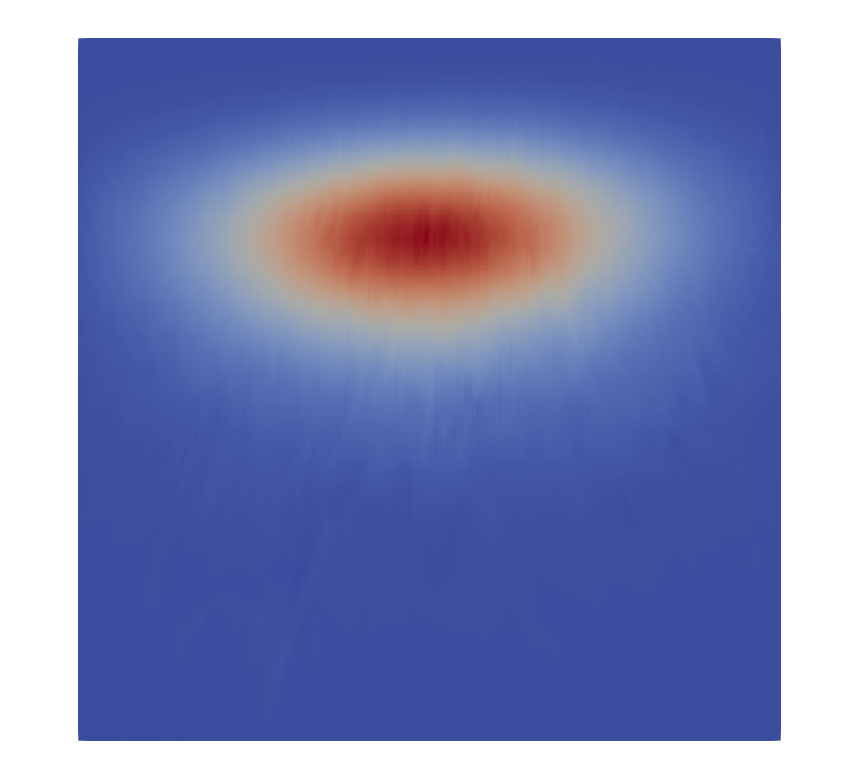}
    \hspace{-5mm}
    \includegraphics[width=0.215\textwidth]{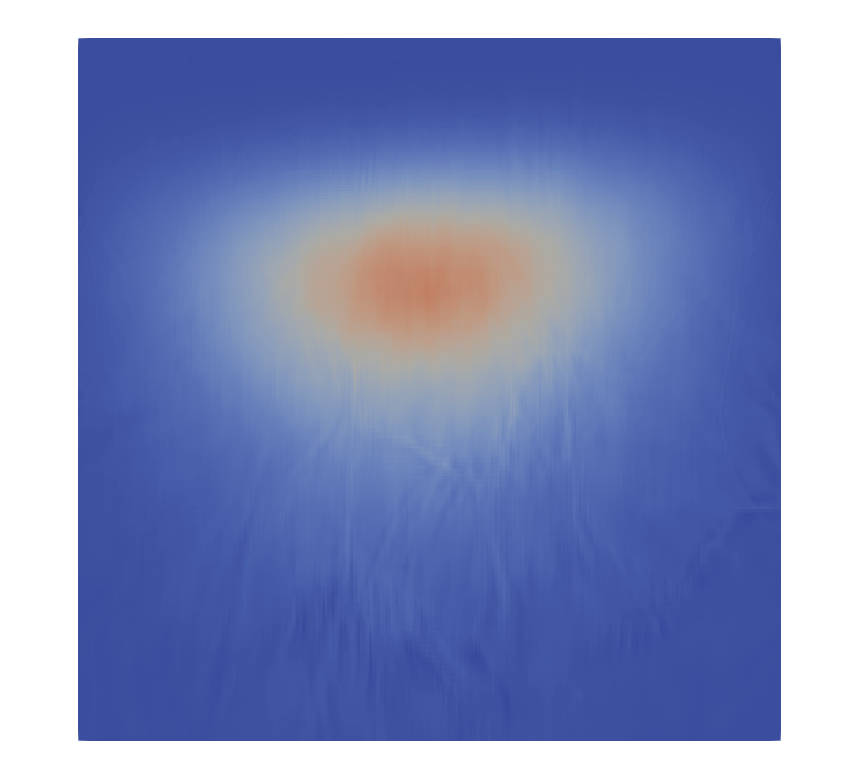}
    \hspace{-5mm}
    \includegraphics[width=0.215\textwidth]{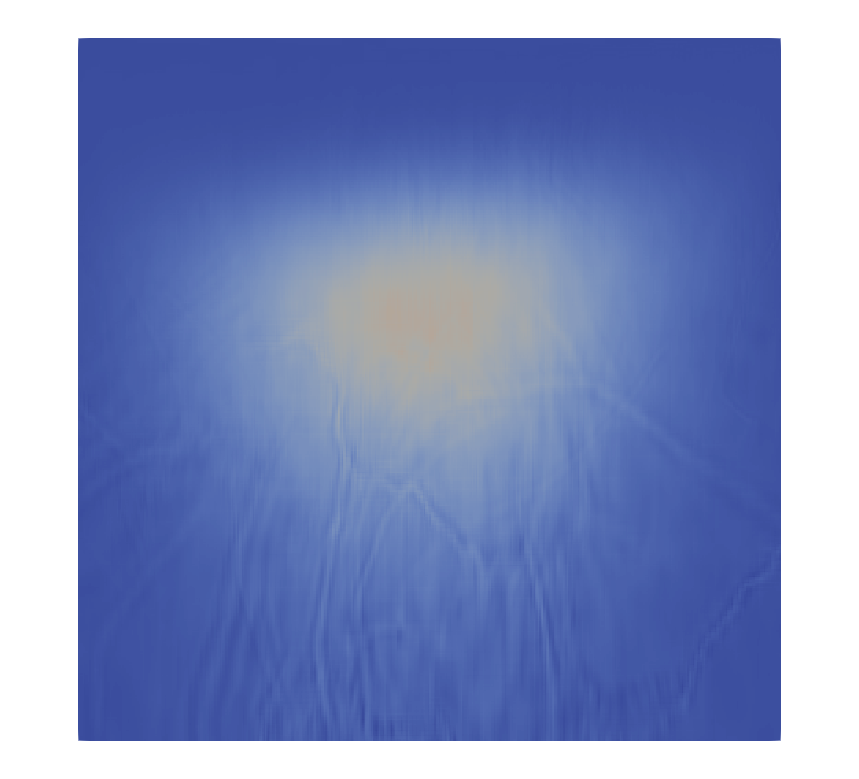}
    \hspace{-5mm}
    \includegraphics[width=0.215\textwidth]{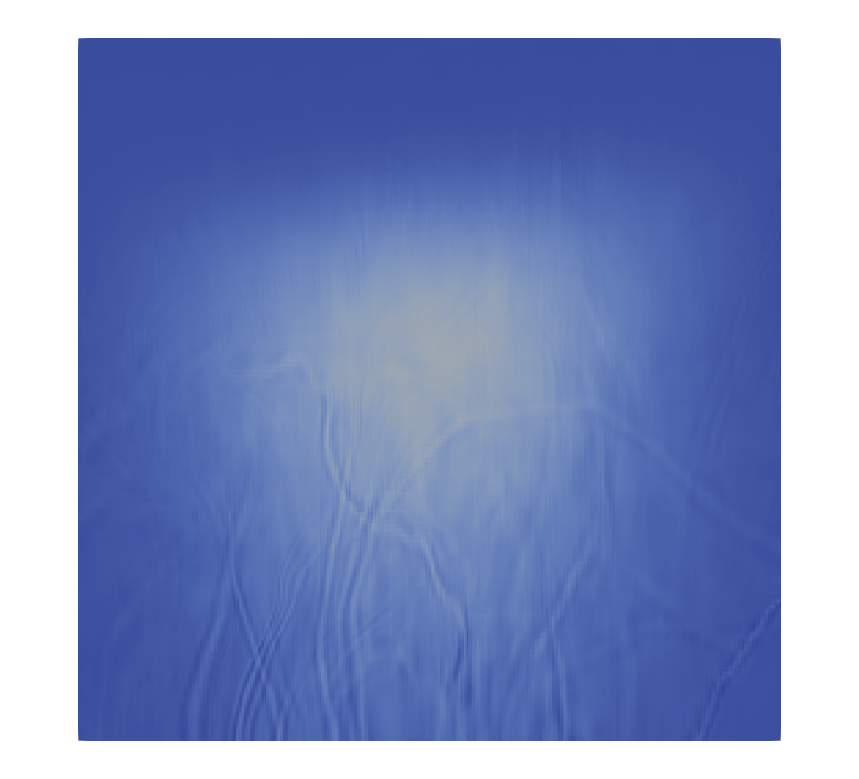}
    \hspace{-5mm}
    \includegraphics[width=0.215\textwidth]{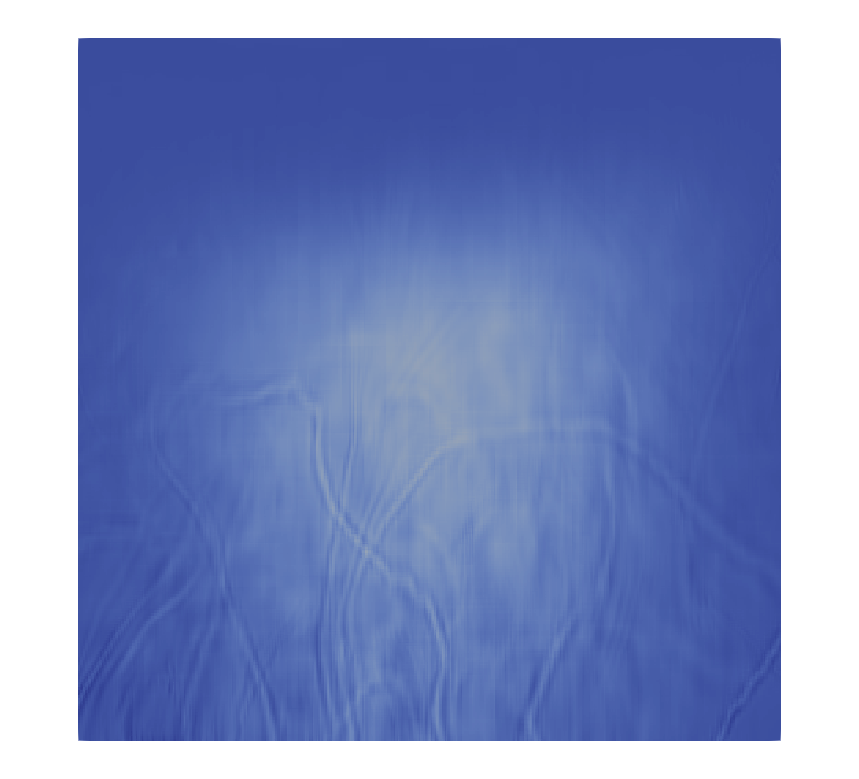}
    \vspace{-10mm}
    \caption{Multi-level mean field approximation at an intermediate estimation round.}
    \label{fig:mean-intermediate}

    \smallskip

    \includegraphics[width=0.215\textwidth]{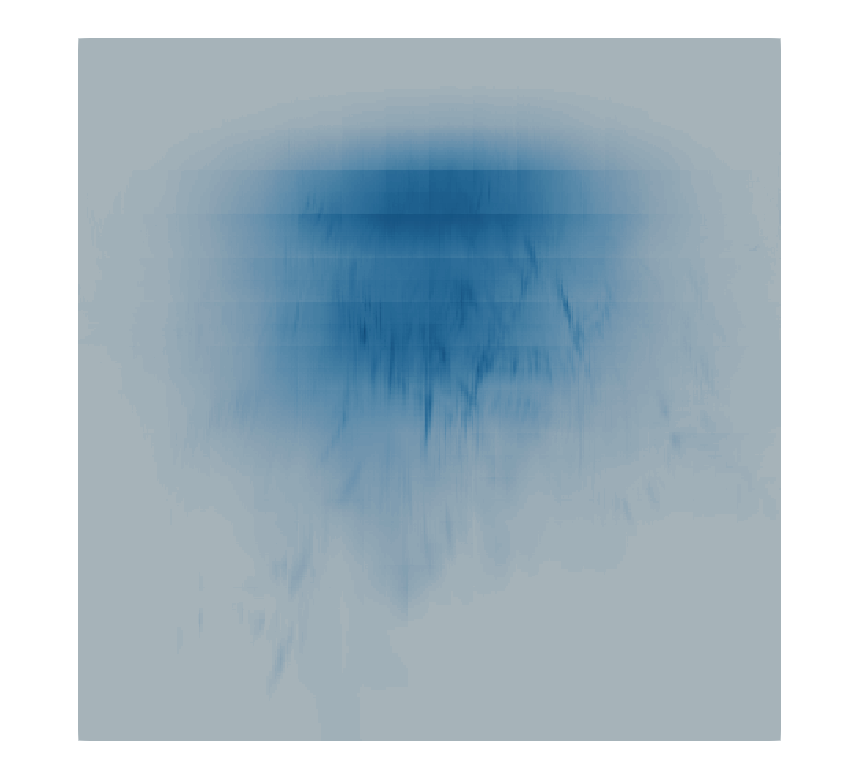}
    \hspace{-5mm}
    \includegraphics[width=0.215\textwidth]{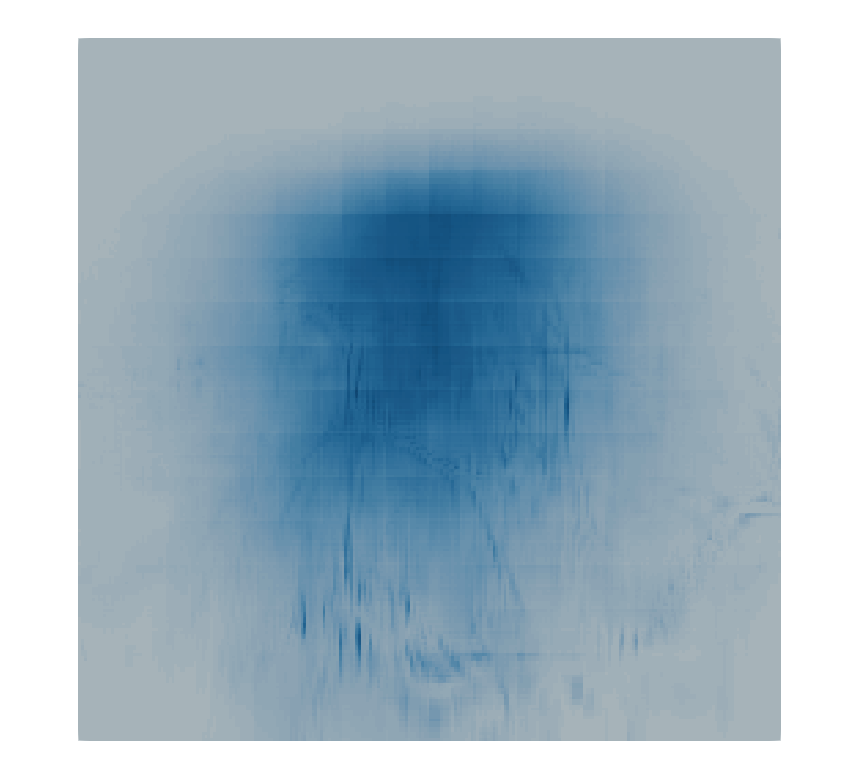}
    \hspace{-5mm}
    \includegraphics[width=0.215\textwidth]{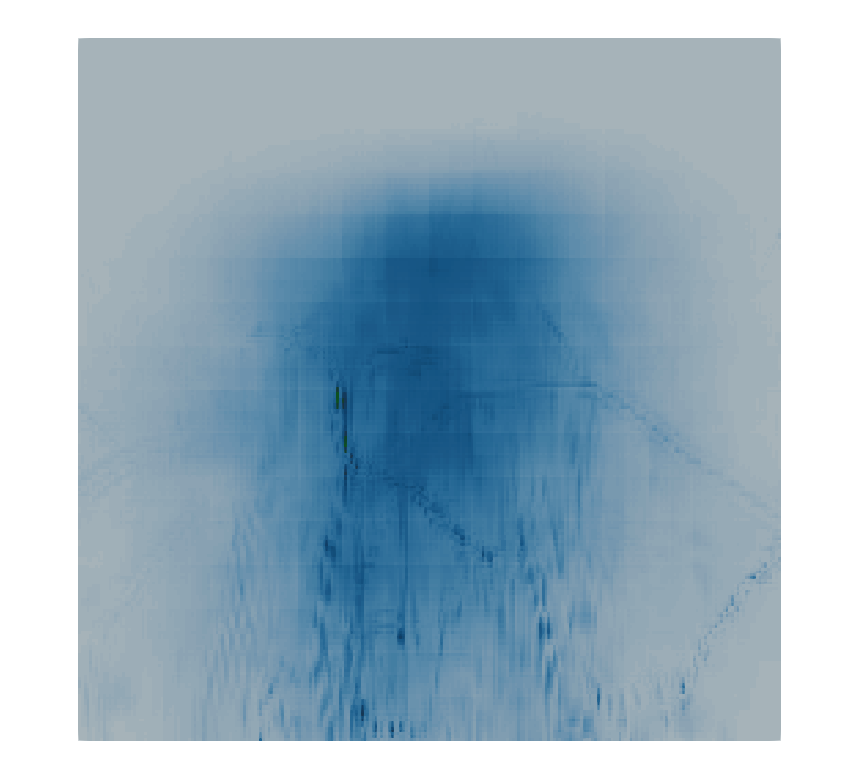}
    \hspace{-5mm}
    \includegraphics[width=0.215\textwidth]{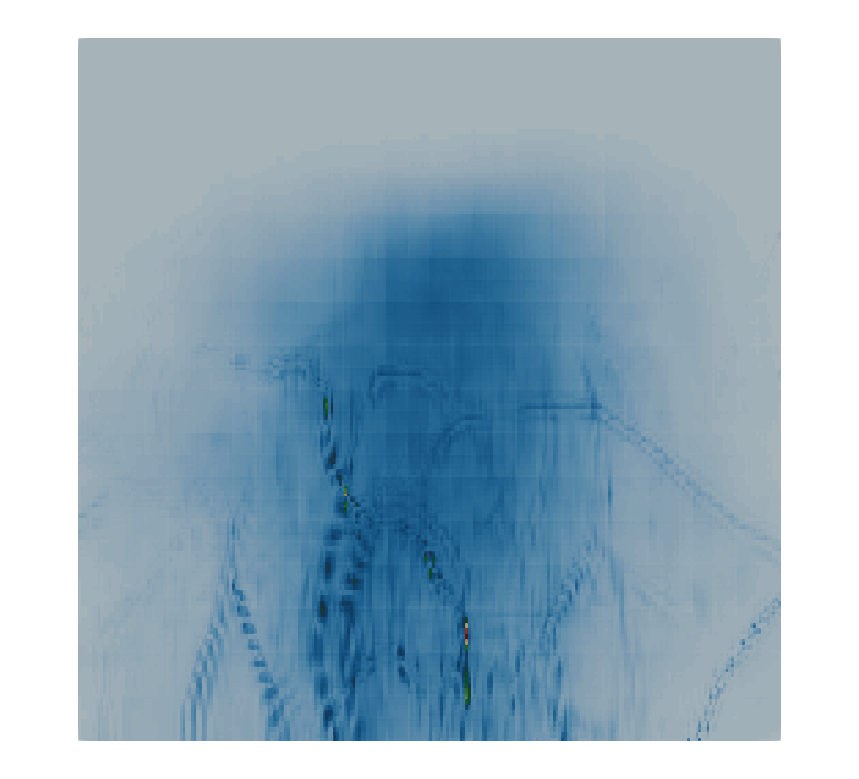}
    \hspace{-5mm}
    \includegraphics[width=0.215\textwidth]{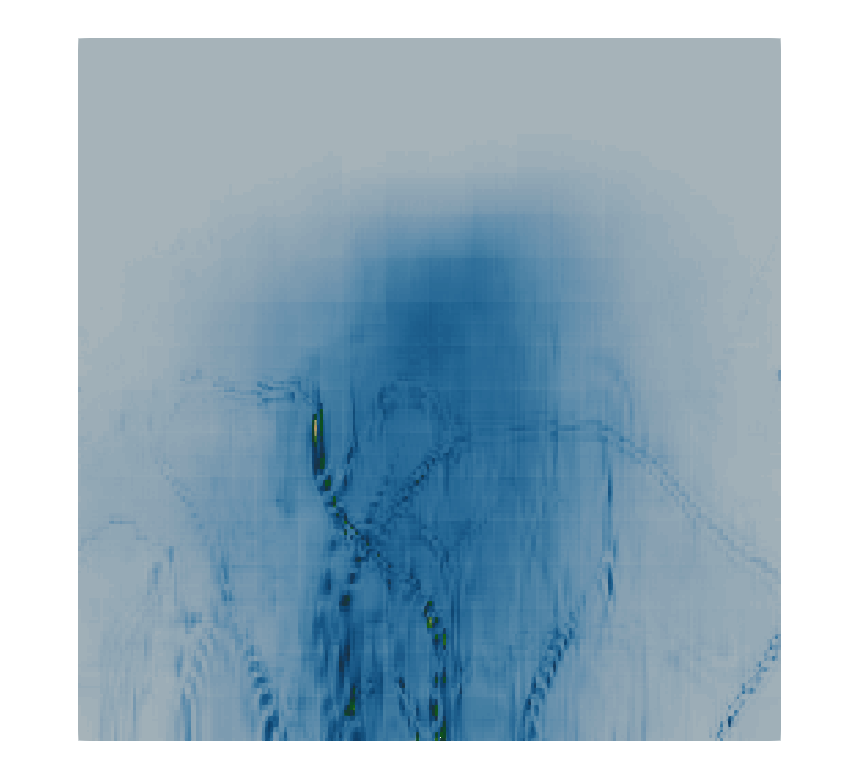}
    \vspace{-10mm}
    \caption{Multi-level marginal sample variance after the initial estimation round.}
    \label{fig:svar-init}

    \smallskip

    \includegraphics[width=0.215\textwidth]{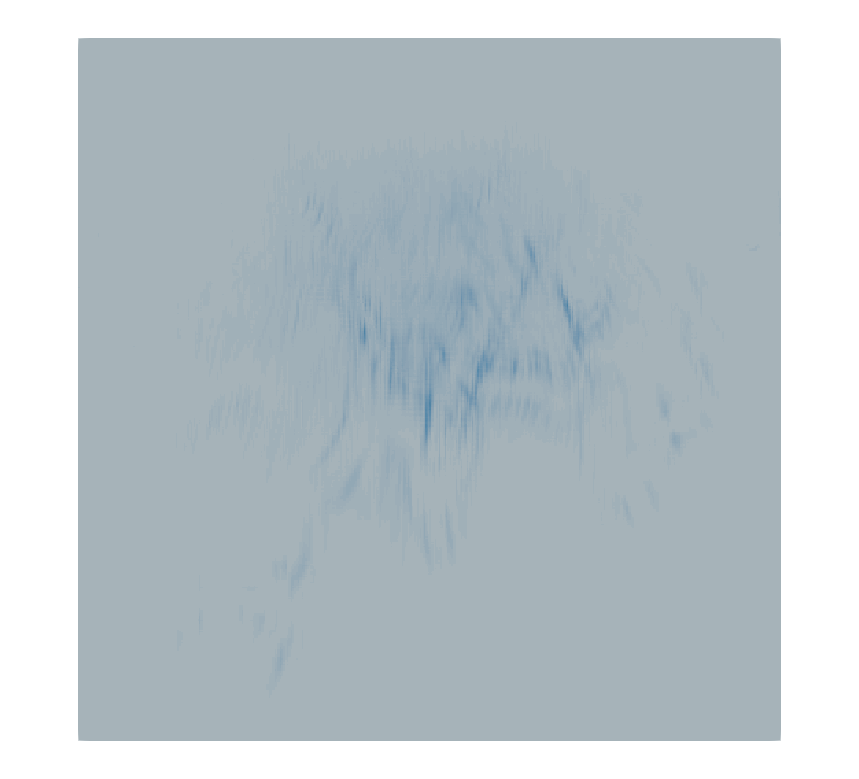}
    \hspace{-5mm}
    \includegraphics[width=0.215\textwidth]{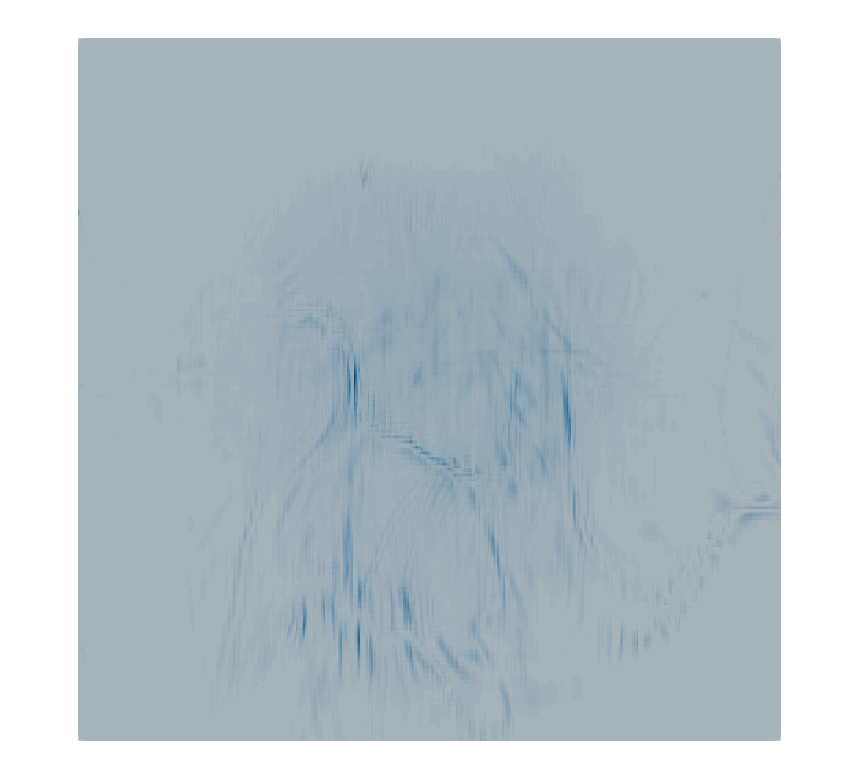}
    \hspace{-5mm}
    \includegraphics[width=0.215\textwidth]{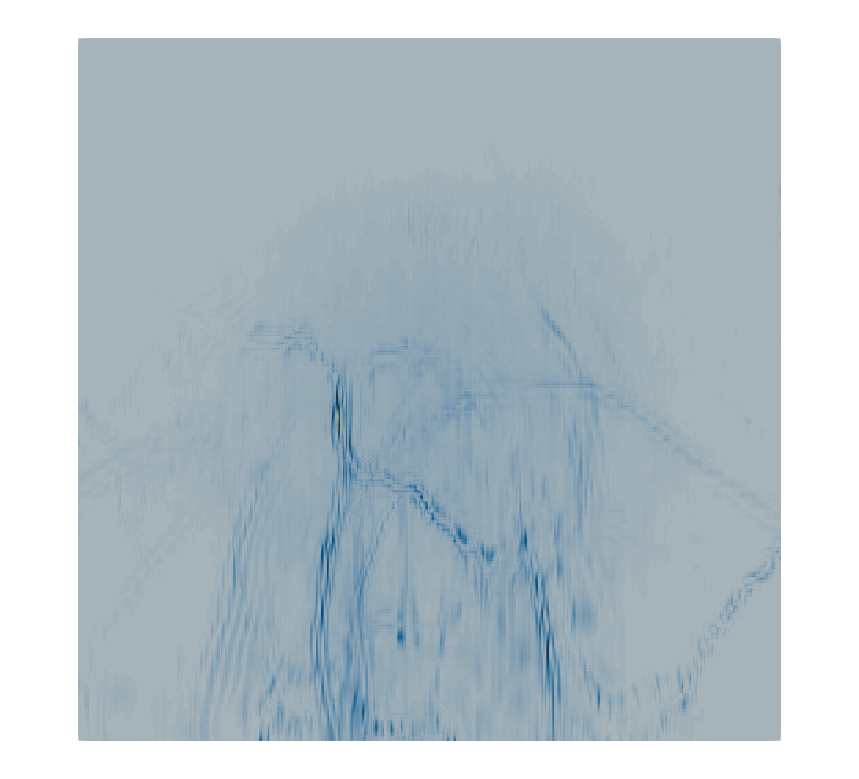}
    \hspace{-5mm}
    \includegraphics[width=0.215\textwidth]{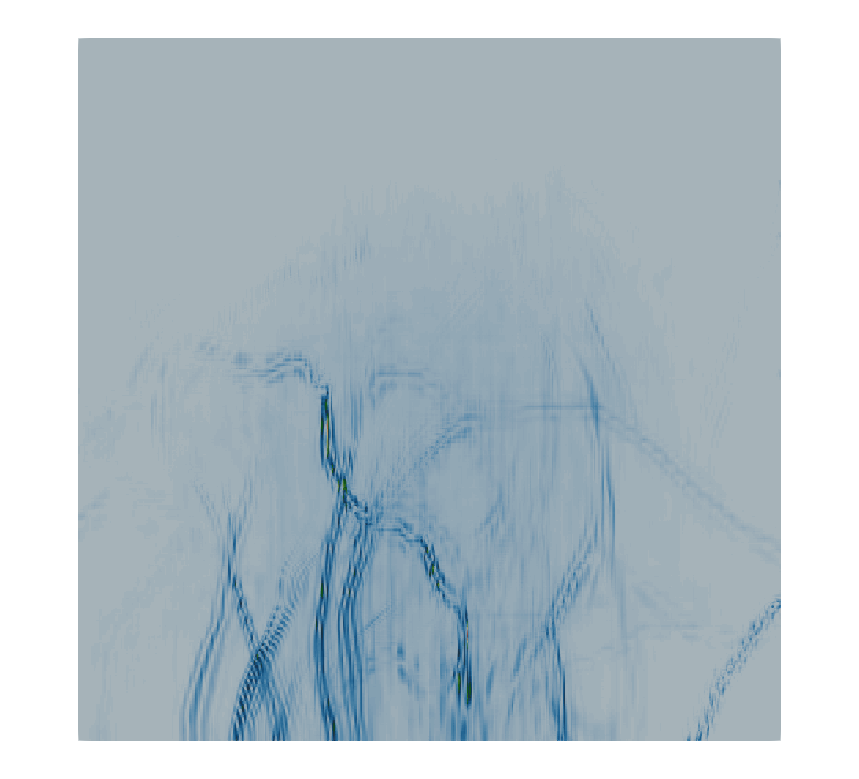}
    \hspace{-5mm}
    \includegraphics[width=0.215\textwidth]{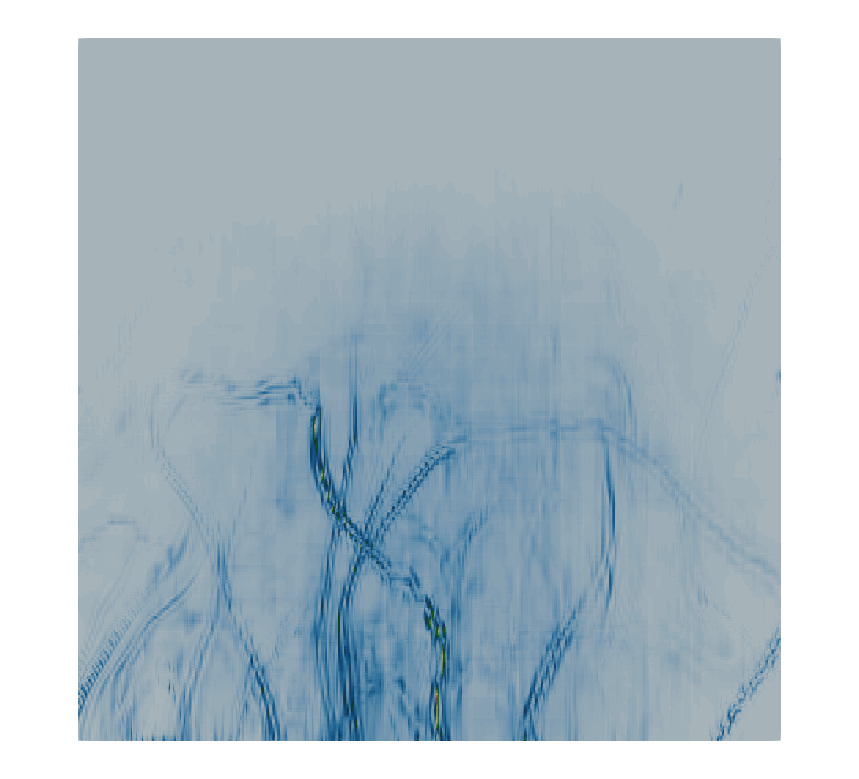}
    \vspace{-10mm}
    \caption{Multi-level marginal sample variance at an intermediate estimation round.}
    \label{fig:svar-intermediate}

\end{figure}

The columns in \Cref{fig:sample1}~-~\Cref{fig:svar-intermediate} show
the evolution in time and correspond to the time points $t_n = 0.1, 0.2, \dots 0.5$.
For an illustration of two samples 
of the transport solution over time, $\brho^{(1, n)}_{\ell}$ and $\brho^{(2,n)}_{\ell}$, we refer to \Cref{fig:sample1,fig:sample2}.
Already at $t_n = 0.1$, the two samples exhibit different behaviour,
deviating visibly from the Gaussian initial condition.
It is also worth noting that as time progresses,
due to the different
realizations of the flux field $\bq(\omega, \bx)$, the samples differ more and more from each other.
\Cref{fig:mean-para} shows the mean field approximation
$E_{M_\ell=4}^{\text{MC}}[\bu_{\ell}]$ over time, after the parallel computation
of four samples, including $\brho^{(1, n)}_{\ell}$ and $\brho^{(2,n)}_{\ell}$.
The influence of both samples can still be seen, but
the mean field approximation is already recognizable, especially for the earlier time points.
This figure demonstrates in particular the functionality of the implementation.
%
In~\Cref{fig:mean-init,fig:mean-intermediate},
the multi-level mean field approximations $\sum_{\ell=0}^{L_0} \rP_{\ell}^{L} E_{0, \ell}^{\text{MC}}[\bv_{\ell}]$
and $\sum_{\ell=0}^{L_{\ttti}} \rP_{\ell}^{L} E_{\ttti, \ell}^{\text{MC}}[\bv_{\ell}]$ according to~\eqref{eq:mlmc-estimator-full-solution}
are shown after the initial and an intermediate estimation round, respectively.
It is clearly visible that the mean field approximation becomes smoother
and more refined with additional estimation rounds.
\Cref{fig:svar-init,fig:svar-intermediate} show the marginal multi-level sample variances
$\sum_{\ell=0}^{L_0} V^{\text{MC}}_{0, \ell}[\bv_\ell]$ and
$\sum_{\ell=0}^{L_{\ttti}} V^{\text{MC}}_{\ttti, \ell}[\bv_\ell]$ across different time steps,
again after the initial and after an intermediate estimation round.
While the estimation of these quantities is not strong enough
to control the sampling error of the RF estimate,
it gives a good intuition of the uncertainty in the solution.
We see that, with additional estimation rounds,
the uncertainty in the solution decreases.
However, we also see that as the time steps
$t_n$ progress, fine-scale features emerge that incorporate the
uncertainty from the higher levels and the later time points into the estimate.
In conclusion, these figures demonstrate the functionality of the implementation and
give an intuition about the problem setup and the algorithmic capabilities.

%% file: src/computational-results.tex
\subsection{Computational Results}
\label{subsec:computational-results}

Committing to the setting of the numerical experiments
described~\Cref{subsec:problem-configuration-and-illustration},
the goal of this section is to support the theoretical findings,
to demonstrate the functionality,
and to verify the assumptions.

\paragraph{Full field estimates are no additional burden}

One central claim in this paper is that the full field estimates,
even though they provide less variance reduction than the estimates
for QoIs (cf.~\Cref{corollary:relation-beta}),
do not require additional computational resources in terms of memory or computing time.
To underline this, we compare the computational results of two implementations:
one compiled with the full field multi-index update algorithm
of~\Cref{subsec:sparse-multi-index-finite-element-update-algorithm} (ON),
and one compiled without the functionality (OFF).
The results are shown in~\Cref{fig:computational-results-1} using the
compute time budget of $\rT_{\rB}=1$ hour.
The top row verifies that~\Cref{assumption:mlmc} is justified
by presenting the estimates for
$\widehat{\alpha}_{\bu}, \widehat{\alpha}_{\rQ}, \widehat{\beta}_{\bv}, \widehat{\beta}_{\rY}, \widehat{\gamma}_{\rC\rT}, \widehat{\gamma}_{\mathrm{Mem}}$
computed with~\eqref{eq:alpha-fit}.
To keep both experiments comparable, we optimize the multi-level hierarchy
for the estimation of the QoI and compute the full field estimates simultaneously in the ON case.
From the two top plots on the left, we get confirmation for~\Cref{lemma:relation-Y-v},
while in the top right plot, we see that neither the memory footprint nor the computing time
cost is significantly larger with the full field estimates.
This plot also illustrates that~\Cref{lemma:memory-consumption} is valid since
the total memory footprint (blue and orange horizontal lines)
is just above the measured memory of the largest level (bar plot on the very right)
and below the budget (red horizontal line).
This also confirms that~\eqref{eq:knapsack-mlmc-mem} is respected.

The second row of~\Cref{fig:computational-results-1} shows on the very left all
multi-indices $(s, \ell)$ used during the simulation.
For some levels, multiple communication splits are tried to find the optimal load distribution,
since the minimum in~\eqref{eq:comm-split-formula} might change over the estimation rounds.
It is clearly visible that the memory intensive pairs $(s, \ell)$ in the top right corner
are avoided, presenting the solution to~\Cref{problem:approximated-mlmc-knapsack}
together with the optimal sample distribution plot in the middle.
Here, we can see again that full field estimates do not cause significant
losses in computing time as the optimal sample distribution is almost identical.
Lastly, the bar plot on the very right presents how much time is spent on every level.
The horizontal lines indicate the total time used, verifying that
the computing time constraint~\eqref{eq:knapsack-mlmc-ct} is respected, too.

\begin{figure}
    \centering

    \vspace{-0.3cm}

    \includegraphics[width=\textwidth]{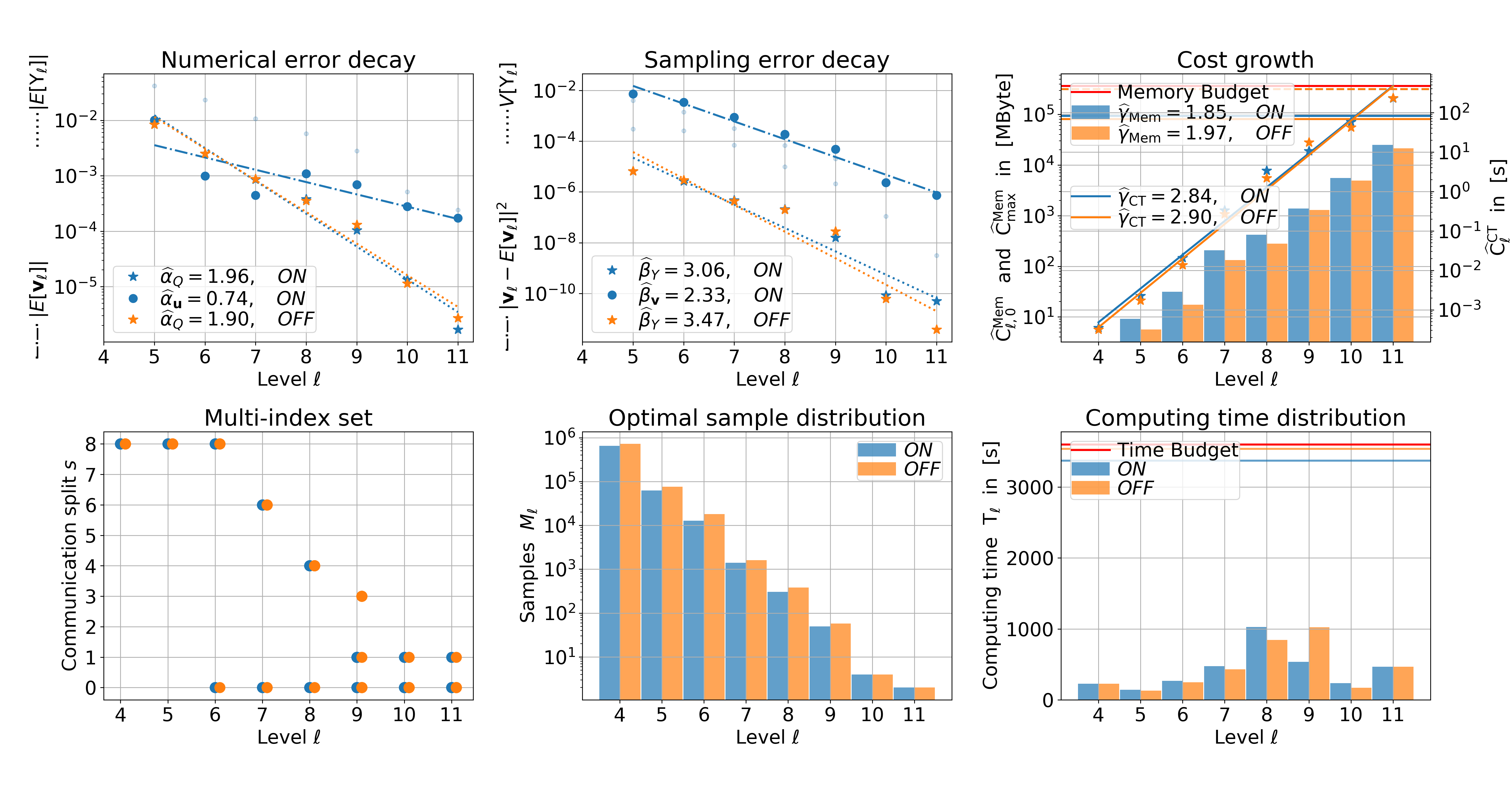}

    \vspace{-0.7cm}

    \caption{Implementations with (ON) and without (OFF) full solution update funcitonality.}
    \label{fig:computational-results-1}
\end{figure}

\paragraph{Full field guided multi-level hierarchy}
The experiment summarized in~\Cref{fig:computational-results-2},
optimizes the multi-level hierarchy with~\eqref{eq:optimal-Ml-estimated}
using the accumulative update of the Bochner norm presented in~\Cref{lemma:accumulation-of-z}.
Here, we extend the reserved compute time budgets from $1$ to $3$ hours.
In the top row, we see again that all assumptions are valid,
though convergence rates are weaker for the full field estimates.
As a consequence, the algorithm does not visit level 10 or 11 for $\rT_{\rB} = 1h$.
Even with extended computing times $\rT_{\rB} \in \set{2h, 3h}$, it still does not visit level 11.
This is due to the worse convergence of $\widehat{\alpha}_{\bu}$ compared to $\widehat{\alpha}_{\rQ}$,
i.e., the computational cost for the new level is not worth while for the
bias reduction it brings in~\eqref{eq:squared-bias-estimate-mlmc}.
This is also in line with~\Cref{corollary:bmlmc-corollary}
where the lower bound of the smallest possible error depends on $\alpha$.
Here,
$(\mathrm{Mem}_{\rB})^{-\widehat{\alpha}_{\bu}} \sim 512^{-0.78} \approx 0,0077$
is bigger than
$(\mathrm{Mem}_{\rB})^{-\widehat{\alpha}_{\rQ}} \sim 512^{-1.90} \approx 0,000007$
and thus, in the memory constraint setting, this is the limiting factor for the error.
Lastly, the plot of~\Cref{fig:computational-results-2} in the lower right corner shows
the estimates of $\widehat{\err}^{\mathrm{num}}$ over $\widehat{\err}^{\mathrm{sam}}$.
Again, the algorithm favors minimizing the sampling error rather than the bias.
A detailed discussion on this type of plot is given in~\cite{baumgarten2024fully},
only emphasizing that it illustrates the trade-off between the sampling error and the bias,
and the round-based dynamic programming approach of the algorithm.

\begin{figure}

    \vspace{-0.3cm}

    \includegraphics[width=\textwidth]{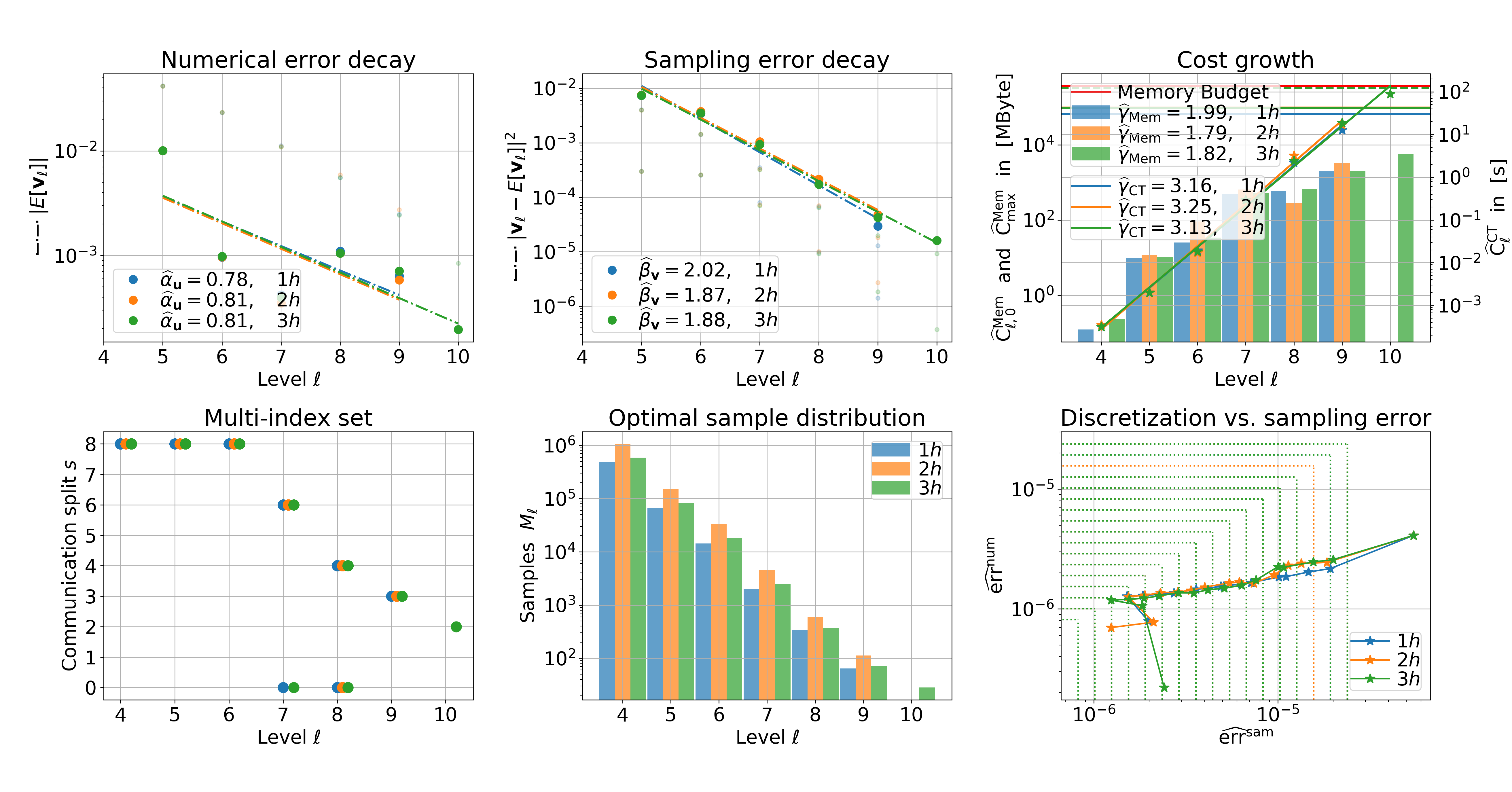}

    \vspace{-0.7cm}

    \caption{Extension of reserverd computing times, i.e., $\rT_{\rB} \in \set{1h, 2h, 3h}$.}
    \label{fig:computational-results-2}
\end{figure}

%% file: src/discussion-and-outlook.tex
\section{Discussion and Conclusion}
\label{sec:discussion-and-outlook}

We present a new adaptation to the budgeted MLMC method~\cite{baumgarten2024fully}
by incorporating a memory constraint and an accumulation algorithm 
to compute RF estimates with full spatial domain resolution.
The proposed algorithm is highly performant, robust, offers several optimal properties,
and is widely applicable to complex and intricate PDE systems.
High performance is achieved by tightly integrating the MLMC method
into the FEM solver and employing DDP techniques.
As a result, the method is more intrusive in implementation,
yet it maintains modularity across algorithms.
While not simple to program, the method is realized only on a minimal technology stack
(C++ and MPI), resulting in little overhead and friction.

The method’s optimality is rooted in solving the resource allocation~\Cref{problem:approximated-mlmc-knapsack}.
To this end, \Cref{lemma:accumulation-of-z}~allows to apply Bellman’s optimality principle
to optimally determine the number of samples on each level via~\eqref{eq:optimal-Ml-estimated}.
The result of ~\Cref{lemma:memory-consumption} ensures that the memory footprint is controlled.
In~\Cref{sec:numerical-experiments}, we numerically demonstrate that the
multi-index set we compute with~\eqref{eq:comm-split-formula} is effectively enforcing the constraints.
Lastly, \Cref{corollary:bmlmc-corollary} establishes a lower bound on the possible root mean squared error.
This lower bound is the result of the memory constraint and depends on the regularity of the solution
through the exponent $\alpha$.

The robustness of the algorithm stems from its adaptability
and by enforcing the hardware limits through the constraints~\eqref{eq:knapsack-mlmc-ct} and~\eqref{eq:knapsack-mlmc-mem}.
This enables high flexibility and allows application to a wide range of problems and algorithms.
The wide applicability also stems from the arbitrary choices in~\Cref{alg:mx-fem}.
In principle, any FEM-based application can be implemented here,
using the same distributed data structure for all applications.

%% file: src/acknowledgement.tex
\section*{Acknowledgement}

We thank the M++ development group, led by Christian Wieners,
particularly, the contributions from Daniele Corallo,
David Schneiderhan, Simon Wülker and Tobias Merkel.
All have contributed features used in this work.
Further, we acknowledge the access to the HoreKa supercomputer and the technical support
by the National High-Performance Computing Center (NHR) at KIT.